\documentclass[11pt]{amsart}

\usepackage{color, amsmath,amssymb, amsfonts, amstext,amsthm, latexsym}

   \usepackage{latexsym, oldlfont}

   \newtheorem{lemma}{Lemma}[section]
   \newtheorem{theorem}[lemma]{Theorem}
   \newtheorem{remark}[lemma]{Remark}
   \newtheorem{prop}[lemma]{Proposition}
   \newtheorem{coro}[lemma]{Corollary}

\newcommand{\R}{{\mathbb R}}

\newcommand{\beq}{\begin{equation}}
\newcommand{\eeq}{\end{equation}}
\newcommand{\Del}{\Delta}

\newcommand{\nab}{\nabla}
\newcommand{\lam}{\lambda}

\newcommand{\ome}{\omega}
\newcommand{\patl}{\partial}

\newcommand{\al}{\alpha}

\title[Navier-Stokes Equations]{Symplectic Symmetry and Radial
Symmetry Either Persistence or Breaking of Incompressible Fluid}

\author{Yongqian Han}

\date{}

\address{
Institute of Applied Physics and Computational Mathematics,
Beijing 100088, China}
\address{
National Key Laboratory of Computational Physics,
Beijing 100088, China}
\email{han\_yongqian@iapcm.ac.cn}



\begin{document}

\begin{abstract} The incompressible Navier-Stokes equations are
considered. We find that these equations have symplectic symmetry
structures. Two linearly independent symplectic symmetries form
moving frame. The velocity vector possesses symplectic representation
in a moving frame. The symplectic representation of two-dimensional
Navier-Stokes equations holds radial symmetry persistence. On the other
hand, we establish some results of radial symmetry either persistence
or breaking for the symplectic representations of three-dimensional
Navier-Stokes equations. Thanks radial symmetry persistence, we
construct infinite non-trivial solutions of static Euler equations
with given boundary condition.
\end{abstract}

\maketitle

~ ~ ~ ~

\section{Introduction}
\setcounter{equation}{0}

The Navier-Stokes equations in $\R^3$ with initial data are given by
\beq\label{NS1}
u_t-\nu\Del u+(u\cdot\nab)u+\nab P=0,
\eeq
\beq\label{NS2}
\nab\cdot u=0,
\eeq
\beq\label{NSi}
u|_{t=0}=u_0,
\eeq
where $u=u(t,x)=\big(u^1(t,x),u^2(t,x),u^3(t,x)\big)$ and $P=P(t,x)$
stand for the unknown velocity vector field of fluid and its
pressure, $u_0=u_0(x)=\big(u^1_0(x),u^2_0(x),u^3_0(x)\big)$ is
the given initial velocity vector field satisfying $\nab\cdot u_0=0$,
$\nu>0$ is the coefficient of viscosity. Here $\patl_{x_j}$ denotes
by $\patl_j$ ($j=1,2,3$).

For the mathematical setting of this problem, we introduce
Hilbert space
\[
H(\R^3)=\big\{u\in\big(L^2(\R^3)\big)^3\big|\nab \cdot u=0\big\}
\]
endowed with $\big(L^2(\R^3)\big)^3$ norm (resp. scalar product
denoted by $(\cdot,\cdot)$ ). For simplicity of presentation, space
$\big(H^m(\R^3)\big)^3$ denotes by $H^m$, where $m\ge0$. In what
follows we use the usual convention and sum over the repeated indices.

The Fourier transformation of $u(t,x)$ with respect to $x$ denotes by
$\hat{u}(t,\xi)$. Then equation \eqref{NS2} means that
\beq\label{NS2-F}
\xi_j\hat{u}^j=\xi_1\hat{u}^1+\xi_2\hat{u}^2+\xi_3\hat{u}^3=0.
\eeq
It is that $\hat{u}(t,\xi)$ is perpendicular to $\xi$. Denote by $\xi
\bot\hat{u}$. Equivalently $\hat{u}(t,\xi)\in T_{\xi}{\mathbb{S}}^2$.

Let $A=(a,b,c)\in\R^3-\{0\}$ and $\xi\in{\mathbb{S}}^2$. The vector
$\{A\times\xi\}\in T_{\xi}{\mathbb{S}}^2$ is called the 1st order
incompressible symplectic symmetry. Take $3\times3$ matrix $M=\big(
m^{ij}\big)$. The vector $\{\xi\times(M\xi)\}\in T_{\xi}{\mathbb{S}}
^2$ is called the 2nd order incompressible symplectic symmetry.
Generally let $T=\big(T^{i_1\cdots i_n}\big)$ be $n$th order tensor
and $T(\xi\cdots\xi)=T^{i_1i_2\cdots i_n}\xi_{i_2}\cdots\xi_{i_n}$.
The vector $\{\xi\times T(\xi\cdots\xi)\}\in T_{\xi}{\mathbb{S}}^2$
is called the $n$th order incompressible symplectic symmetry.

Let vectors $A\in\R^3-\{0\}$ and $B\in\R^3-\{0\}$ be linear
independent. Then $A\times\xi$ and $B\times\xi$ are linear independent
for any $\xi\in\R^3-\{0\}$, and form the basis of space $T_{\xi}
{\mathbb{S}}^2$ which is so-called moving frame. There exist $
\hat{\phi}(\xi)$ and $\hat{\psi}(\xi)$ such that
\[
\hat{u}(\xi)=\hat{\phi}(\xi)\cdot\{A\times\xi\}+\hat{\psi}(\xi)\cdot
\{B\times\xi\},\;\;\forall \hat{u}(\xi)\in T_{\xi}{\mathbb{S}}^2.
\]

Therefore for any velocity vector $u$ satisfying the equation
\eqref{NS2}, there exist linear independent vectors $A,B\in\R^3-\{0\}$
and real scalar functions $\phi$ and $\psi$ such that
\beq\label{Sym-Rep-11}
u(t,x)=\{A\times\nab\}\phi(t,x)+\{B\times\nab\}\psi(t,x).
\eeq
We call that the formulation \eqref{Sym-Rep-11} is (1,1)-symplectic
representation of velocity vector $u$.

Similarly we call that the following formulation \eqref{Sym-Rep-22}
\beq\label{Sym-Rep-22}
u(t,x)=\{(A\times\nab)\times\nab\}\phi(t,x)+\{(B\times\nab)\times\nab\}\psi(t,x),
\eeq
is (2,2)-symplectic representation of velocity vector $u$. And the
following formulation \eqref{Sym-Rep-12}
\beq\label{Sym-Rep-12}
u(t,x)=\{A\times\nab\}\phi(t,x)+\{(B\times\nab)\times\nab\}\psi(t,x),
\eeq
is called (1,2)-symplectic representation of velocity vector $u$.

There is a large literature studying the incompressible Navier-Stokes
equations. In 1934 Leray \cite{Le34} proved that there exists a global
weak solution to the problem \eqref{NS1}--\eqref{NSi} with initial data
in $L^2$. In 1951 Hopf \cite{Hop51} extended this result to bounded
smooth domain. Moreover Leray-Hopf weak solutions satisfy energy
inequality \cite{Te01}
\beq\label{En-Ineq}
\|u(t,\cdot)\|^2_{L^2}+2\int^t_0\|\nab u(\tau,\cdot)\|^2_{L^2}d\tau
\le \|u_0\|^2_{L^2},\;\;\;\;\forall t>0.
\eeq

The uniqueness and regularity of Leray-Hopf weak solution is a famous
open question. Numerous regularity criteria were proved \cite{ESS03,
Gig86,KeK11,KNSS09,La69,Pro59,Ser62,Wah86}.

Local existence and uniqueness of $H^m$ solution can be established
by using analytic semigroup \cite{Lun95} with initial data in
$H^m(\R^3)$, $m\ge1$. This result is stated as follows.

\begin{prop}[Local $H^m$ Solution]\label{NS-LHmS} Let $u_0\in H^m(\R^3)
\cap H(\R^3)$ and $m\ge1$. Then there exist $T_{max}=T_{max}\big(
\|u_0\|_{H^m}\big)>0$ and a unique solution $u$ of the problem
\eqref{NS1}--\eqref{NSi} such that $u\in C\big([0,T_{max});H^m(\R^3)
\cap H(\R^3)\big)$.
\end{prop}

Local existence and uniqueness of mild solution or strong solution
were established \cite{Cal90,GiM85,Ka84,KaF62,KaP94,Wei80} with
initial data in $L^p(\R^3)$, $p>3$. The main result is as follows.

\begin{prop}[Local Mild Solution]\label{NS-LMS} Let $u_0\in L^p(\R^3)$
satisfy \eqref{NS2} in distribution and $p>3$. Then there exist $
T_{max}=T_{max}\big(\|u_0\|_{L^p}\big)>0$ and a unique solution $u$ of
the problem \eqref{NS1}--\eqref{NSi} such that $u\in C([0,T_{max});
L^p(\R^3))$.
\end{prop}

The uniqueness in Proposition \ref{NS-LHmS} and \ref{NS-LMS} ensures
that the symplectic symmetries which are corresponding to initial data
$u_0$ can be kept by the solution $u$.

Besides the local-posedness, the lower bounds of possible blowup
solutions were considered \cite{CSTY08,CSTY09,CMP14,Gig86,Le34,
RSS12}. The concentration phenomena of possible blowup solutions was
studied \cite{LOW18}.

It is well-known that the equations \eqref{NS1} \eqref{NS2} are
scaling-invariant in the sense that if $u$ solves \eqref{NS1}
\eqref{NS2} with initial data $u_0$, so dose $u_{\lam}(t,x)=\lam
u(\lam^2t,\lam x)$ with initial data $\lam u_0(\lam x)$. A space $X$
defined on $\R^3$ is so-called to be critical provided $\|u_0\|_X=
\|\lam u_0(\lam\cdot)\|_X$ for any $\lam>0$. $L^3(\R^3)$ is one of
critical spaces. For the initial data in critical spaces, the
posedness of global solution of the equations \eqref{NS1}--\eqref{NSi}
is obtained \cite{Can97,Che99,KoT01,Pla96} with small
initial data. The regularity criterion was established \cite{ESS03,
GKP13,KeK11,LiW19,Sere12}. On the other hand, the ill-posedness was
showed \cite{BoP08,Ger08,Wang15,Yon10}.

It is also studied that solutions of the problem
\eqref{NS1}--\eqref{NSi} are in various function spaces \cite{Can03,
GiInM99,GiM89,Ka75,KoT01,Tay92}. Partial regularity of suitable weak
solutions was established \cite{CKN82,LaS99,Lin98,Sch76,Sch77,Vas07,
WaW14}. Non-existence of self-similar solutions was proved
\cite{NeRS96,Tsai98}. Decay of the solutions can be found in
\cite{Car96,MiS01,Sch91,Sch92}, etc.

In this paper, we study the radial symmetry of symplectic
representation for the solutions of the problem \eqref{NS1}--
\eqref{NSi}.

Firstly we consider two-dimensional Navier-Stokes equations. Here $x=
(x_1,x_2)\in{\mathbb{R}}^2$. The main result is as follows.

\begin{theorem}[Radial Symmetry Persistence]\label{SR11-R2-Thm}
For the problem \eqref{NS1}--\eqref{NSi}, let $x=(x_1,x_2)\in\R^2$,
the velocity vectors $u$ and $u_0$ respectively hold the following
(1,1)-symplectic representation
\beq\label{uSR11-R2}\begin{split}
u(t,x)=&\big(-\patl_2\phi(t,x),\patl_1\phi(t,x),0\big),\\
\end{split}\eeq
\beq\label{u0SR11-R2}\begin{split}
u_0(x)=&\big(-\patl_2\phi_0(x),\patl_1\phi_0(x),0\big).\\
\end{split}\eeq
Assume that the initial data $\phi_0(x)$ is radial symmetric function
and regular enough. Then there exists a unique global solution $u$ of
the problem \eqref{NS1}--\eqref{NSi} such that $u$ satisfies
\eqref{uSR11-R2}  and
\beq\label{SolPhi-NS1-R2}\begin{split}
\phi(t,x)=&\frac1{4\pi\nu t}\int_{\R^2}\exp\{-\frac{|y|^2}{4\nu t}\}
\phi_0(x-y)dy.\\
\end{split}\eeq
Moreover the function $\phi$ is radial symmetric function.
\end{theorem}

By the proof of Theorem \ref{SR11-R2-Thm}, equation \eqref{PB-R-R2}
means that $u_0$ defined by \eqref{u0SR11-R2} is static Euler flow
provided $\phi_0(x)$ is radial symmetric function. Let $\Phi(s)\in
C^{\infty}_c(0,1)$ and $\phi_0(x)=\phi_0(r)=\Phi(r/R_a)$ for any
$R_a>0$, where $x\in\R^2$ and $r^2=x\cdot x$. Then $u=\big(-\patl_2
\phi_0(r),\patl_1\phi_0(r),0\big)$ is the solution of static two
dimensional Euler equations
\beq\label{STE-eq}\begin{split}
&(u\cdot\nab)u+\nab P=0,\\
&\nab\cdot u=0,\\
\end{split}\eeq
with Dirichlet boundary condition
\beq\label{DBC-D2}\begin{split}
u|_{\patl B^2(r<R_a)}=0.
\end{split}\eeq
It is obvious that the number of solution of static two dimensional
Euler equations \eqref{STE-eq} \eqref{DBC-D2} is infinite.

Similarly let $\Phi_j(r)=\al\sin(jr)+\beta\cos(jr)$, $j=1,2,\cdots$, $
\al$ and $\beta$ be any constants. Then
\[\begin{split}
u(x)=&\big(-\patl_2\Phi_j(r),\patl_1\Phi_j(r),0\big)\\
\end{split}\]
is the solution of static two dimensional Euler equations
\eqref{STE-eq} with symplectic representation
\[\begin{split}
u(x)=&\big(-\patl_2\phi(x),\patl_1\phi(x),0\big),\\
\end{split}\]
and periodic boundary condition
\beq\label{PBC-d2}
\phi(r+2\pi)=\phi(r),\;\;\forall r\ge0
\eeq
for any $j=1,2,\cdots$.

For example $\phi_0(r)=r^{-\frac2p+1}$. By Theorem \ref{SR11-R2-Thm},
there exists a unique global solution $u$ although $\|u_0\|_{L^p(\R^2)}
=\infty$ for any $1\le p\le\infty$. The problem is whether this
solution $u$ has continuous dependence on initial data or not.

Contrast to the radial symmetry persistence of two-dimensional
Navier-Stokes equations, there appears more complicate and more
interesting phenomena for three-dimensional Navier-Stokes equations.

We find that the following two equations
\beq\label{12ph-AA-d3}\begin{split}
\frac1r\patl_r\psi\cdot\patl_r\Big(\frac1r\patl_r\phi\Big)-
\frac1r\patl_r\phi\cdot\patl_r\Big(\frac1r\patl_r\psi\Big)=0,\\
\end{split}\eeq
\beq\label{12ps-AA-d3}\begin{split}
\frac1r\patl_r\phi\cdot\patl_r\Big(\frac1r\patl_r\phi\Big)+
\frac1r\patl_r\psi\cdot\patl_r\Big(\frac1r\patl_r\Del\psi\Big)=0,\\
\end{split}\eeq
play key role to solve three-dimensional Navier-Stokes equations and
static three dimensional Euler equations.

\begin{theorem}[Radial Symmetry Either Persistence or Breaking]
\label{SR12-AEB-Thm}
For the problem \eqref{NS1}--\eqref{NSi}, assume that the velocity
vectors $u$ and $u_0\not=0$ respectively hold the following
(1,2)-symplectic representation
\beq\label{uSR12-R3}\begin{split}
u(t,x)=&(A\times\nab)\phi(t,x)
+\{(A\times\nab)\times\nab\}\psi(t,x),\\
\end{split}\eeq
\beq\label{u0SR12-R3}\begin{split}
u_0(x)=&(A\times\nab)\phi_0(x)
+\{(A\times\nab)\times\nab\}\psi_0(x),\\
\end{split}\eeq
where vector $A\in\R^3-\{0\}$. Let $\phi_0(x)$ and $\psi_0(x)$ be
regular enough.

Provided $(\phi,\psi)=\big(\Phi(r),\Psi(r)\big)$ satisfies equations
\eqref{12ph-AA-d3} \eqref{12ps-AA-d3}.

{\rm (I) (Static Euler Flow)} $u=(A\times\nab)\Phi+\{(A\times\nab)
\times\nab\}\Psi$ satisfies static three dimensional Euler equations
\eqref{STE-eq}.

{\rm (II) (Radial Symmetry Persistence)} Provided $(\phi_0,\psi_0)=
\big(\Phi(r),\Psi(r)\big)$. Let us define
\beq\label{NSphs-S-CS}\begin{split}
&\phi(t,x)=\phi(t,r)=e^{-\nu\Del t}\phi_0=e^{-\nu\Del t}\Phi,\\
&\psi(t,x)=\psi(t,r)=e^{-\nu\Del t}\psi_0=e^{-\nu\Del t}\Psi. \\
\end{split}\eeq
Then $u$ defined by \eqref{uSR12-R3} \eqref{NSphs-S-CS} is solution of
the problem \eqref{NS1}--\eqref{NSi}.

{\rm (III) (Radial Symmetry Breaking)} Provided $(\phi_0,\psi_0)\not=
\big(\Phi(r),\Psi(r)\big)$, $u$ defined by \eqref{uSR12-R3} is
solution of the problem \eqref{NS1}--\eqref{NSi}, and $[0,T_{max})$ is
the maximum existent interval of $t$ for this solution $u$. Assume
that there exists $t_0\in(0,T_{max})$ such that $\big(\phi(t,x),
\psi(t,x)\big)$ defined by \eqref{uSR12-R3} is radial with respect to
$x$ at $t=t_0$. Then $T_{max}=\infty$, and there exists $T_r\in(0,t_0]
$ such that this vector $\big(\phi(t,x),\psi(t,x)\big)$ is radial with
respect to $x$ for any $t\ge T_r$,  but it is not radial with respect
to $x$ for any $t\in(0,T_r)$. Otherwise this vector $\big(\phi(t,x),
\psi(t,x)\big)$ can not be radial with respect to $x$ for any $t\in
(0,T_{max})$ although the initial data $\phi_0(x)$ and $\psi_0(x)$
are radial symmetric functions.
\end{theorem}

\begin{coro}\label{SC-EE-PBC}
Let us define
\beq\label{EE-PBC-d3}\begin{split}
&\Phi_{\lam\al\beta}(r)=\lam\Psi_{\lam\al\beta}(r),\\
&\Psi_{\lam\al\beta}(r)=\al\frac1r\sin(\lam r)
+\beta\frac1r\cos(\lam r),\\
\end{split}\eeq
where $\lam,\al$ and $\beta$ are any real constants. Then $(\phi,\psi)=
\Big(\Phi_{\lam\al\beta}(r),\Psi_{\lam\al\beta}(r)\Big)$ satisfies
equations \eqref{12ph-AA-d3} \eqref{12ps-AA-d3}. $u=(A\times\nab)
\Phi_{\lam\al\beta}+\{(A\times\nab)\times\nab\}\Psi_{\lam\al\beta}$ is
solution of static three dimensional Euler equations \eqref{STE-eq}.

Let us take
\beq\label{EE-PBC-d3-pp}\begin{split}
&\phi(t,x)=\phi(t,r)=e^{-\nu\lam^2t}\Phi_{\lam\al\beta}
=\lam\psi(t,x),\\
&\psi(t,x)=\psi(t,r)=e^{-\nu\lam^2t}\Psi_{\lam\al\beta}. \\
\end{split}\eeq
Then $u$ defined by \eqref{uSR12-R3} \eqref{EE-PBC-d3-pp} is solution
of the problem \eqref{NS1}--\eqref{NSi}.
\end{coro}

In Corollary \ref{SC-EE-PBC}, let $\lam=j=1,2,\cdots$. Then
\[
u=(A\times\nab)\Phi_{j\al\beta}+\{(A\times\nab)\times\nab\}\Psi_{j\al\beta}
\]
is the solution of static three dimensional Euler equations
\eqref{STE-eq} with symplectic representation
\beq\label{SR12-d3SE}
u=(A\times\nab)\phi+\{(A\times\nab)\times\nab\}\psi
\eeq
and periodic boundary condition
\beq\label{PBC-d3}
\{\xi\phi(\xi)\}\big|_{\xi=r+2\pi}=r\phi(r),\;\;
\{\xi\psi(\xi)\}\big|_{\xi=r+2\pi}=r\psi(r),\;\;\forall r\ge0
\eeq
for any $j=1,2,\cdots$.

\begin{remark}\label{Bac-Tran} The equation \eqref{SR12-d3SE} is
so-called B$\ddot{a}$cklund transformation which changes equations
\eqref{12ph-AA-d3} \eqref{12ps-AA-d3} into static three dimensional
Euler equations \eqref{STE-eq}.

Provided $\big(\Phi(r),\Psi(r)\big)$ satisfies equations
\eqref{12ph-AA-d3} \eqref{12ps-AA-d3}. It is that
\beq\label{12Lph-AA-d3}\begin{split}
\frac1r\patl_r\Psi\cdot\patl_r\Big(\frac1r\patl_r\Phi\Big)-
\frac1r\patl_r\Phi\cdot\patl_r\Big(\frac1r\patl_r\Psi\Big)=0,\\
\end{split}\eeq
\beq\label{12Lps-AA-d3}\begin{split}
\frac1r\patl_r\Phi\cdot\patl_r\Big(\frac1r\patl_r\Phi\Big)+
\frac1r\patl_r\Psi\cdot\patl_r\Big(\frac1r\patl_r\Del\Psi\Big)=0.\\
\end{split}\eeq
Let
\beq\label{NSphs-S-CS1}\begin{split}
&\phi(t,x)=e^{-\nu\Del t}\Phi,\\
&\psi(t,x)=e^{-\nu\Del t}\Psi. \\
\end{split}\eeq
The equations \eqref{uSR12-R3} \eqref{NSphs-S-CS1} are also
B$\ddot{a}$cklund transformations which change equations
\eqref{12Lph-AA-d3} \eqref{12Lps-AA-d3} into
Navier-Stokes equations \eqref{NS1} \eqref{NS2}.

More B$\ddot{a}$cklund transformations can be found in
\cite{Bac75,RogS82}.
\end{remark}

\begin{theorem}[Radial Symmetry Breaking]\label{SR12-AVB-Thm}
Let $u$ be solution of the problem \eqref{NS1}--\eqref{NSi}. Assume
that the velocity vectors $u$ and $u_0$ respectively hold the
following (1,2)-symplectic representation
\beq\label{uSR22-R3}\begin{split}
u(t,x)=&(A\times\nab)\phi(t,x)+\{(B\times\nab)\times\nab\}\psi(t,x),\\
\end{split}\eeq
\beq\label{u0SR22-R3}\begin{split}
u_0(x)=&(A\times\nab)\phi_0(x)+\{(B\times\nab)\times\nab\}\psi_0(x),\\
\end{split}\eeq
where vectors $A\in\R^3-\{0\}$ and $B\in\R^3-\{0\}$ are perpendicular
each other $A\bot B$. Let $\phi_0(x)$ and $\psi_0(x)$ be regular
enough. Then the functions $\phi(t,x)$ and $\psi(t,x)$ can not be
radial symmetric functions of $x$ for any $t>0$ although the initial
data $\phi_0(x)$ and $\psi_0(x)$ are radial symmetric functions,
except that $\phi_0=f_2r^2+f_0$, $\psi_0=g_4r^4+g_2r^2+g_0$ and
$f_2g_4=0$. Here $f_j$ and $g_j$ (j=0, 2, 4) are arbitrary constants.

Especially provided $u_0\not=0$ and $u_0\in L^2(\R^3)$, then the
functions $\phi(t,x)$ and $\psi(t,x)$ can not be radial symmetric
functions of $x$ for any $t>0$ although the initial data $\phi_0(x)$
and $\psi_0(x)$ are radial symmetric functions.
\end{theorem}

\begin{theorem}[Radial Symmetry Breaking]\label{SR11-Thm}
Let $u$ be solution of the problem \eqref{NS1}--\eqref{NSi}. Assume
that the velocity vectors $u$ and $u_0\not=0$ respectively hold the
following (1,1)-symplectic representation
\beq\label{uSR11-R3}\begin{split}
u(t,x)=&(A\times\nab)\phi(t,x)+(B\times\nab)\psi(t,x),\\
\end{split}\eeq
\beq\label{u0SR11-R3}\begin{split}
u_0(x)=&(A\times\nab)\phi_0(x)+(B\times\nab)\psi_0(x),\\
\end{split}\eeq
where vectors $A\in\R^3-\{0\}$ and $B\in\R^3-\{0\}$ are linearly
independent. Let $\phi_0(x)$ and $\psi_0(x)$ be regular enough. Then
the functions $\phi(t,x)$ and $\psi(t,x)$ can not be radial symmetric
functions of $x$ for any $t>0$ although the initial data $\phi_0(x)$
and $\psi_0(x)$ are radial symmetric functions, except that $\phi_0=
f_2r^2+f_0$ and $\psi_0=g_2r^2+g_0$. Here $f_j$ and $g_j$ (j=0, 2) are
arbitrary constants.

Especially provided $u_0\not=0$ and $u_0\in L^2(\R^3)$, then the
functions $\phi(t,x)$ and $\psi(t,x)$ can not be radial symmetric
functions of $x$ for any $t>0$ although the initial data $\phi_0(x)$
and $\psi_0(x)$ are radial symmetric functions.
\end{theorem}

\begin{theorem}[Radial Symmetry Breaking]\label{SR22-Thm}
Let $u$ be solution of the problem \eqref{NS1}--\eqref{NSi}. Assume
that the velocity vectors $u$ and $u_0$ respectively hold the
following (2,2)-symplectic representation
\beq\label{uSR22-R3}\begin{split}
u(t,x)=&\{(A\times\nab)\times\nab\}\phi(t,x)
+\{(B\times\nab)\times\nab\}\psi(t,x),\\
\end{split}\eeq
\beq\label{u0SR22-R3}\begin{split}
u_0(x)=&\{(A\times\nab)\times\nab\}\phi_0(x)
+\{(B\times\nab)\times\nab\}\psi_0(x),\\
\end{split}\eeq
where vectors $A\in\R^3-\{0\}$ and $B\in\R^3-\{0\}$ are linearly
independent. Let $\phi_0(x)$ and $\psi_0(x)$ be regular enough. Then
the functions $\phi(t,x)$ and $\psi(t,x)$ can not be radial symmetric
functions of $x$ for any $t>0$ although the initial data $\phi_0(x)$
and $\psi_0(x)$ are radial symmetric functions, except that $\phi_0=
f_4r^4+f_2r^2+f_0$, $\psi_0=g_4r^4+g_2r^2+g_0$ and $f_2g_4=f_4g_2$.
Here $f_j$ and $g_j$ (j=0, 2, 4) are arbitrary constants.

Especially provided $u_0\not=0$ and $u_0\in L^2(\R^3)$, then the
functions $\phi(t,x)$ and $\psi(t,x)$ can not be radial symmetric
functions of $x$ for any $t>0$ although the initial data $\phi_0(x)$
and $\psi_0(x)$ are radial symmetric functions.
\end{theorem}

Here radial symmetry breaking is a kind of singularity of structure.
Appropriate orthogonal transformations are the main ingredient of
proving these theorems.

The plan of this paper is as follows. Section 2 is devoted to study
radial symmetry persistence of two-dimensional Navier-Stokes equations.
Section 3 is devoted to show radial symmetry either persistence or
breaking of three-dimensional Navier-Stokes equations with
(1,2)-symplectic representation. Section 4 is devoted to establish
radial symmetry breaking of three-dimensional Navier-Stokes equations
with (1,1)-symplectic representation. Section 5 is devoted to
investigate radial symmetry breaking of three-dimensional
Navier-Stokes equations with (2,2)-symplectic representation.

\section{(1,1)-Symplectic Representation and \\Radial Symmetry Persistence
in ${\mathbb{R}}^2$}
\setcounter{equation}{0}

In this section, we assume that $x=(x_1,x_2)\in{\mathbb{R}}^2$ and the
velocity vector $u$ holds the following (1,1)-symplectic representation
\beq\label{Sym-Rep-11-R2}\begin{split}
&u(t,x)=u(t,x_1,x_2)=\big(u^1(t,x_1,x_2),u^2(t,x_1,x_2),0\big),\\
&u^1(t,x_1,x_2)=-\patl_2\phi(t,x_1,x_2),\\
&u^2(t,x_1,x_2)=\patl_1\phi(t,x_1,x_2).\\
\end{split}\eeq
Putting together \eqref{NS1} and \eqref{Sym-Rep-11-R2}, we have
\beq\label{NS-SR11-R2}\begin{split}
\patl_t\Del\phi-\nu\Del^2\phi=&-\patl_1(u\cdot\nab)u^2
+\patl_2(u\cdot\nab)u^1\\
=&-\{\phi,\Del \phi\},\\
\end{split}\eeq
where $\Del=\patl_1^2+\patl_2^2$, Poisson
bracket $\{\cdot,\cdot\}$ is given by
\[
\{f,g\}=\patl_1f\patl_2g-\patl_2f\patl_1g.
\]
The equation \eqref{NS-SR11-R2} is well known Hasegawa-Mima equation
(\cite{HM77,HM78}).

Now we assume that $\phi$ is radial symmetric function with respect to
space variable $x\in\R^2$. It is that $\phi(t,x)=\phi(t,r)$ and $r^2=
x_1^2+x_2^2$. Then we have
\beq\label{PB-R-R2}\begin{split}
\{\phi,\Del \phi\}
=\frac{x_1}r\patl_r\phi\cdot\frac{x_2}r\patl_r\Del\phi
-\frac{x_2}r\patl_r\phi\cdot\frac{x_1}r\patl_r\Del\phi=0,\\
\end{split}\eeq
and the equation \eqref{NS-SR11-R2} is equivalent to
\beq\label{NS11-R2-1}\begin{split}
\patl_t\phi-\nu\Del\phi=&w_0(t),\\
\end{split}\eeq
where $w_0$ is any function of $t$.

\eqref{PB-R-R2} means that $u=(-\patl_2\phi,\patl_1\phi,0)$ is the
solution of two dimensional static Euler equation \eqref{STE-eq}.

There exists a global solution of equation \eqref{NS11-R2-1}
\beq\label{RadSol-R2}\begin{split}
\phi(t,x)=\frac1{4\pi\nu t}\int_{\R^2}\exp\{-\frac{|y|^2}{4\nu t}\}
\phi_0(x-y)dy+\int^t_0w_0(s)ds,\\
\end{split}\eeq
where $\phi_0(r)=\phi(t,r)|_{t=0}$.

$\phi(t,x)$ is radial symmetric function of $x$ since $\phi_0$
is radial.

Because $w_0(t)$ has no contribution on velocity $u$, we select $w_0=0$.

For any point $\xi\in{\mathbb S}^1$, the unit tangent vector $v\in
T_{\xi}{\mathbb S}^1$ is unique corresponding to given positive
direction. Then the moving frame on $T{\mathbb S}^1$ is also unique,
and the solution $u$ is unique.

Theorem \ref{SR11-R2-Thm} is proved.

\section{(1,2)-Symplectic Representation and Radial Symmetry Either
Persistence or Breaking in ${\mathbb{R}}^3$}
\setcounter{equation}{0}

In this section, we assume that the velocity vector $u$ holds the
following (1,2)-symplectic representation
\beq\label{SR12-R3}\begin{split}
u(t,x)=&(A\times\nab)\phi(t,x)+\{(B\times\nab)\times\nab\}\psi(t,x),\\
\end{split}\eeq
where vectors $A=(a_1,a_2,a_3)\in\R^3-\{0\}$ and $B=(b_1,b_2,b_3)\in
\R^3-\{0\}$.

Let
\beq\label{Def-SR12-ome}\begin{split}
\ome(t,x)=&\nab\times u(t,x)\\
=&-\{(A\times\nab)\times\nab\}\phi(t,x)+(B\times\nab)\Del\psi(t,x),\\
\end{split}\eeq
Taking curl with equation \eqref{NS1}, we have
\beq\label{NS-SR12-ome}
\ome_t-\nu\Del\ome+(u\cdot\nab)\ome-(\ome\cdot\nab)u=0.
\eeq

Thanks the following observations
\beq\label{SR12-phi}\begin{split}
(B\times\nab)\cdot u(t,x)&=(A\times\nab)\cdot(B\times\nab)\phi(t,x)\\
&=\big\{(A\cdot B)\Del-(A\cdot\nab)(B\cdot\nab)\big\}\phi(t,x),\\
\end{split}\eeq
\beq\label{SR12-psi}\begin{split}
(A\times\nab)\cdot \ome(t,x)&=(A\times\nab)\cdot(B\times\nab)\Del\psi(t,x)\\
&=\big\{(A\cdot B)\Del-(A\cdot\nab)(B\cdot\nab)\big\}\Del\psi(t,x),\\
\end{split}\eeq
taking scalar product of equation \eqref{NS1} with $B\times\nab$,
we have
\beq\label{NS-SR12-phi}
\big\{(A\cdot B)\Del-(A\cdot\nab)(B\cdot\nab)\big\}\{\phi_t-\nu\Del\phi\}
+(B\times\nab)\cdot\{(u\cdot\nab)u\}=0.
\eeq
And taking scalar product of equation \eqref{NS-SR12-ome} with
$A\times\nab$, we derive
\beq\label{NS-SR12-psi}
\big\{(A\cdot B)\Del-(A\cdot\nab)(B\cdot\nab)\big\}\Del\{\psi_t-\nu\Del\psi\}
+(A\times\nab)\cdot\{(u\cdot\nab)\ome-(\ome\cdot\nab)u\}=0.
\eeq

Now we assume that $\phi$ and $\psi$ are radial symmetric functions
with respect to space variable $x\in\R^3$. It is that $\phi(t,x)=
\phi(t,r)$, $\psi(t,x)=\psi(t,r)$  and $r^2=x_1^2+x_2^2+x_3^2$.
Then we have
\beq\label{12-Ru}\begin{split}
u(t,x)=&(A\times x)\frac1r\patl_r\phi
+(B\cdot x)x\frac1r\patl_r\Big(\frac1r\patl_r\psi\Big)\\
&+B\frac1r\patl_r\psi-B\Del\psi,\\
\end{split}\eeq
\beq\label{SR12-Rome}\begin{split}
\ome(t,x)=&-(A\cdot x)x\frac1r\patl_r\Big(\frac1r\patl_r\phi\Big)
-A\frac1r\patl_r\phi\\
&+A\Del\phi+(B\times x)\frac1r\patl_r\Del\psi,\\
\end{split}\eeq
\beq\label{12R-NTuu}\begin{split}
(u\cdot\nab)u=&\Big\{\frac1r\patl_r\phi\{(A\times x)\cdot\nab\}
+(B\cdot x)\frac1r\patl_r\Big(\frac1r\patl_r\psi\Big)\{x\cdot\nab\}\\
&+\frac1r\patl_r\psi\{B\cdot\nab\}-\Del\psi\{B\cdot\nab\}\Big\}\Big\{
(A\times x)\frac1r\patl_r\phi\\
&+(B\cdot x)x\frac1r\patl_r\Big(\frac1r\patl_r\psi\Big)
+B\frac1r\patl_r\psi-B\Del\psi\Big\}\\
\end{split}\eeq
\[\label{12Ruu}\begin{split}
=&\{A\times(A\times x)\}\frac1r\patl_r\phi\cdot\frac1r\patl_r\phi\\
&-x\{(A\times B)\cdot x\}\frac1r\patl_r\phi\cdot
\frac1r\patl_r\Big(\frac1r\patl_r\psi\Big)\\
&+2(A\times x)(B\cdot x)\frac1r\patl_r\phi\cdot
\frac1r\patl_r\Big(\frac1r\patl_r\psi\Big)\\
&+(A\times x)(B\cdot x)\patl_r\Big(\frac1r\patl_r\phi\Big)\cdot
\patl_r\Big(\frac1r\patl_r\psi\Big)\\
&+2x(B\cdot x)^2\frac1r\patl_r\Big(\frac1r\patl_r\psi\Big)\cdot
\frac1r\patl_r\Big(\frac1r\patl_r\psi\Big)\\
&+x(B\cdot x)^2\patl_r\Big(\frac1r\patl_r\psi\Big)\cdot
\patl_r\Big\{\frac1r\patl_r\Big(\frac1r\patl_r\psi\Big)\Big\}\\
&+B(B\cdot x)\patl_r\Big(\frac1r\patl_r\psi\Big)\cdot
\patl_r\Big(\frac1r\patl_r\psi\Big)\\
&-B(B\cdot x)\patl_r\Big(\frac1r\patl_r\psi\Big)\cdot
\patl_r\Del\psi\\
&+(A\times B)\frac1r\patl_r\phi\cdot\frac1r\patl_r\psi\\
&+(A\times x)(B\cdot x)\frac1r\patl_r\psi\cdot
\frac1r\patl_r\Big(\frac1r\patl_r\phi\Big)\\
&+x(B\cdot B)\frac1r\patl_r\psi\cdot
\frac1r\patl_r\Big(\frac1r\patl_r\psi\Big)\\
&+B(B\cdot x)\frac1r\patl_r\psi\cdot
\frac1r\patl_r\Big(\frac1r\patl_r\psi\Big)\\
&+x(B\cdot x)^2\frac1r\patl_r\psi\cdot
\frac1r\patl_r\Big\{\frac1r\patl_r\Big(\frac1r\patl_r\psi\Big)\Big\}\\
&+B(B\cdot x)\frac1r\patl_r\psi\cdot
\frac1r\patl_r\Big(\frac1r\patl_r\psi\Big)\\
&-B(B\cdot x)\frac1r\patl_r\psi\cdot\frac1r\patl_r\Del\psi\\
&-(A\times B)\Del\psi\cdot\frac1r\patl_r\phi\\
&-(A\times x)(B\cdot x)\Del\psi\cdot
\frac1r\patl_r\Big(\frac1r\patl_r\phi\Big)\\
&-x(B\cdot B)\Del\psi\cdot\frac1r\patl_r\Big(\frac1r\patl_r\psi\Big)\\
&-B(B\cdot x)\Del\psi\cdot\frac1r\patl_r\Big(\frac1r\patl_r\psi\Big)\\
&-x(B\cdot x)^2\Del\psi\cdot
\frac1r\patl_r\Big\{\frac1r\patl_r\Big(\frac1r\patl_r\psi\Big)\Big\}\\
&-B(B\cdot x)\Del\psi\cdot\frac1r\patl_r\Big(\frac1r\patl_r\psi\Big)\\
&+B(B\cdot x)\Del\psi\cdot\frac1r\patl_r\Del\psi,\\
\end{split}\]

\beq\label{12R-NTuu-bt}\begin{split}
&(B\times\nab)\cdot\{(u\cdot\nab)u\}\\
=&(B\times\nab)\cdot\Big\{\{A(A\cdot x)-x(A\cdot A)\}
\frac1r\patl_r\phi\cdot\frac1r\patl_r\phi\Big\}\\
&-(B\times\nab)\cdot\Big\{x\{(A\times B)\cdot x\}\frac1r\patl_r\phi\cdot
\frac1r\patl_r\Big(\frac1r\patl_r\psi\Big)\Big\}\\
&+2(B\times\nab)\cdot\Big\{(A\times x)(B\cdot x)\frac1r\patl_r\phi\cdot
\frac1r\patl_r\Big(\frac1r\patl_r\psi\Big)\Big\}\\
&+(B\times\nab)\cdot\Big\{(A\times x)(B\cdot x)\patl_r\Big(\frac1r
\patl_r\phi\Big)\cdot\patl_r\Big(\frac1r\patl_r\psi\Big)\Big\}\\
&+2(B\times\nab)\cdot\Big\{x(B\cdot x)^2\frac1r\patl_r\Big(\frac1r
\patl_r\psi\Big)\cdot\frac1r\patl_r\Big(\frac1r\patl_r\psi\Big)\Big\}\\
&+(B\times\nab)\cdot\Big\{x(B\cdot x)^2\patl_r\Big(\frac1r\patl_r\psi\Big)
\cdot\patl_r\Big\{\frac1r\patl_r\Big(\frac1r\patl_r\psi\Big)\Big\}\Big\}\\
&+(B\times\nab)\cdot\Big\{B(B\cdot x)\patl_r\Big(\frac1r\patl_r\psi\Big)
\cdot\patl_r\Big(\frac1r\patl_r\psi\Big)\Big\}\\
&-(B\times\nab)\cdot\Big\{B(B\cdot x)\patl_r\Big(\frac1r\patl_r\psi\Big)
\cdot\patl_r\Del\psi\Big\}\\
&+(B\times\nab)\cdot\Big\{(A\times B)\frac1r\patl_r\phi\cdot
\frac1r\patl_r\psi\Big\}\\
&+(B\times\nab)\cdot\Big\{(A\times x)(B\cdot x)\frac1r\patl_r\psi\cdot
\frac1r\patl_r\Big(\frac1r\patl_r\phi\Big)\Big\}\\
&+(B\times\nab)\cdot\Big\{x(B\cdot B)\frac1r\patl_r\psi\cdot
\frac1r\patl_r\Big(\frac1r\patl_r\psi\Big)\Big\}\\
&+(B\times\nab)\cdot\Big\{B(B\cdot x)\frac1r\patl_r\psi\cdot
\frac1r\patl_r\Big(\frac1r\patl_r\psi\Big)\Big\}\\
&+(B\times\nab)\cdot\Big\{x(B\cdot x)^2\frac1r\patl_r\psi\cdot\frac1r
\patl_r\Big\{\frac1r\patl_r\Big(\frac1r\patl_r\psi\Big)\Big\}\Big\}\\
&+(B\times\nab)\cdot\Big\{B(B\cdot x)\frac1r\patl_r\psi\cdot
\frac1r\patl_r\Big(\frac1r\patl_r\psi\Big)\Big\}\\
&-(B\times\nab)\cdot\Big\{B(B\cdot x)\frac1r\patl_r\psi\cdot
\frac1r\patl_r\Del\psi\Big\}\\
&-(B\times\nab)\cdot\Big\{(A\times B)\Del\psi\cdot
\frac1r\patl_r\phi\Big\}\\
&-(B\times\nab)\cdot\Big\{(A\times x)(B\cdot x)\Del\psi\cdot
\frac1r\patl_r\Big(\frac1r\patl_r\phi\Big)\Big\}\\
&-(B\times\nab)\cdot\Big\{x(B\cdot B)\Del\psi\cdot
\frac1r\patl_r\Big(\frac1r\patl_r\psi\Big)\Big\}\\
&-(B\times\nab)\cdot\Big\{B(B\cdot x)\Del\psi\cdot
\frac1r\patl_r\Big(\frac1r\patl_r\psi\Big)\Big\}\\
&-(B\times\nab)\cdot\Big\{x(B\cdot x)^2\Del\psi\cdot\frac1r\patl_r
\Big\{\frac1r\patl_r\Big(\frac1r\patl_r\psi\Big)\Big\}\Big\}\\
&-(B\times\nab)\cdot\Big\{B(B\cdot x)\Del\psi\cdot
\frac1r\patl_r\Big(\frac1r\patl_r\psi\Big)\Big\}\\
&+(B\times\nab)\cdot\Big\{B(B\cdot x)\Del\psi\cdot
\frac1r\patl_r\Del\psi\Big\}\\
\end{split}\eeq
\[\label{uu-bt}\begin{split}
=&\{(A\times B)\cdot x\}(A\cdot x)\frac1r\patl_r
\Big(\frac1r\patl_r\phi\cdot\frac1r\patl_r\phi\Big)\\
&+\{(A\cdot B)(B\cdot x)-(B\cdot B)(A\cdot x)\}\frac1r\patl_r\phi\cdot
\frac1r\patl_r\Big(\frac1r\patl_r\psi\Big)\\
&+4(A\cdot B)(B\cdot x)\frac1r\patl_r\phi\cdot
\frac1r\patl_r\Big(\frac1r\patl_r\psi\Big)\\
&+2\{(A\cdot B)r^2-(A\cdot x)(B\cdot x)\}(B\cdot x)\frac1r\patl_r\Big\{
\frac1r\patl_r\phi\cdot\frac1r\patl_r\Big(\frac1r\patl_r\psi\Big)\Big\}\\
&+2(A\cdot B)(B\cdot x)\patl_r\Big(\frac1r
\patl_r\phi\Big)\cdot\patl_r\Big(\frac1r\patl_r\psi\Big)\\
&+\{(A\cdot B)r^2-(A\cdot x)(B\cdot x)\}(B\cdot x)\frac1r\patl_r\Big\{
\patl_r\Big(\frac1r\patl_r\phi\Big)\cdot\patl_r
\Big(\frac1r\patl_r\psi\Big)\Big\}\\
&+\{(A\cdot B)(B\cdot x)-(B\cdot B)(A\cdot x)\}\frac1r\patl_r
\Big\{\frac1r\patl_r\phi\cdot\frac1r\patl_r\psi\Big\}\\
&+2(A\cdot B)(B\cdot x)\frac1r\patl_r\psi\cdot
\frac1r\patl_r\Big(\frac1r\patl_r\phi\Big)\\
&+\{(A\cdot B)r^2-(A\cdot x)(B\cdot x)\}(B\cdot x)\frac1r\patl_r\Big\{
\frac1r\patl_r\psi\cdot\frac1r\patl_r\Big(\frac1r\patl_r\phi\Big)\Big\}\\
&-\{(A\cdot B)(B\cdot x)-(B\cdot B)(A\cdot x)\}\frac1r\patl_r
\Big\{\Del\psi\cdot\frac1r\patl_r\phi\Big\}\\
&-2(A\cdot B)(B\cdot x)\Del\psi\cdot
\frac1r\patl_r\Big(\frac1r\patl_r\phi\Big)\\
&-\{(A\cdot B)r^2-(A\cdot x)(B\cdot x)\}(B\cdot x)\frac1r\patl_r\Big\{
\Del\psi\cdot\frac1r\patl_r\Big(\frac1r\patl_r\phi\Big)\Big\},\\
\end{split}\]

\beq\label{12R-NTuo}\begin{split}
(u\cdot\nab)\ome=&\Big\{\frac1r\patl_r\phi\{(A\times x)\cdot\nab\}
+(B\cdot x)\frac1r\patl_r\Big(\frac1r\patl_r\psi\Big)\{x\cdot\nab\}\\
&+\frac1r\patl_r\psi\{B\cdot\nab\}-\Del\psi\{B\cdot\nab\}\Big\}\\
&\Big\{-(A\cdot x)x\frac1r\patl_r\Big(\frac1r\patl_r\phi\Big)
-A\frac1r\patl_r\phi+A\Del\phi+(B\times x)\frac1r\patl_r\Del\psi\Big\}\\
\end{split}\eeq
\[\label{12Ruo}\begin{split}
=&-(A\times x)(A\cdot x)\frac1r\patl_r\phi\cdot
\frac1r\patl_r\Big(\frac1r\patl_r\phi\Big)\\
&+(B\times (A\times x))\frac1r\patl_r\phi\cdot\frac1r\patl_r\Del\psi\\
&-2x(A\cdot x)(B\cdot x)\frac1r\patl_r\Big(\frac1r\patl_r\psi\Big)
\cdot\frac1r\patl_r\Big(\frac1r\patl_r\phi\Big)\\
&-x(A\cdot x)(B\cdot x)\patl_r\Big(\frac1r\patl_r\psi\Big)\cdot
\patl_r\Big\{\frac1r\patl_r\Big(\frac1r\patl_r\phi\Big)\Big\}\\
&-A(B\cdot x)\patl_r\Big(\frac1r\patl_r\psi\Big)\cdot
\patl_r\Big(\frac1r\patl_r\phi\Big)\\
&+A(B\cdot x)\patl_r\Big(\frac1r\patl_r\psi\Big)\cdot\patl_r\Del\phi\\
&+(B\times x)(B\cdot x)\frac1r\patl_r\Big(\frac1r\patl_r\psi\Big)\cdot
\frac1r\patl_r\Del\psi\\
&+(B\times x)(B\cdot x)\patl_r\Big(\frac1r\patl_r\psi\Big)\cdot
\patl_r\Big(\frac1r\patl_r\Del\psi\Big)\\
&-x(A\cdot B)\frac1r\patl_r\psi\cdot
\frac1r\patl_r\Big(\frac1r\patl_r\phi\Big)\\
&-B(A\cdot x)\frac1r\patl_r\psi\cdot
\frac1r\patl_r\Big(\frac1r\patl_r\phi\Big)\\
&-x(A\cdot x)(B\cdot x)\frac1r\patl_r\psi\cdot
\frac1r\patl_r\Big\{\frac1r\patl_r\Big(\frac1r\patl_r\phi\Big)\Big\}\\
&-A(B\cdot x)\frac1r\patl_r\psi\cdot
\frac1r\patl_r\Big(\frac1r\patl_r\phi\Big)\\
&+A(B\cdot x)\frac1r\patl_r\psi\cdot\frac1r\patl_r\Del\phi\\
&+(B\times x)(B\cdot x)\frac1r\patl_r\psi\cdot
\frac1r\patl_r\Big(\frac1r\patl_r\Del\psi\Big)\\
&+x(A\cdot B)\Del\psi\cdot\frac1r\patl_r\Big(\frac1r\patl_r\phi\Big)\\
&+B(A\cdot x)\Del\psi\cdot\frac1r\patl_r\Big(\frac1r\patl_r\phi\Big)\\
&+x(A\cdot x)(B\cdot x)\Del\psi\cdot
\frac1r\patl_r\Big\{\frac1r\patl_r\Big(\frac1r\patl_r\phi\Big)\Big\}\\
&+A(B\cdot x)\Del\psi\cdot\frac1r\patl_r\Big(\frac1r\patl_r\phi\Big)\\
&-A(B\cdot x)\Del\psi\cdot\frac1r\patl_r\Del\phi\\
&-(B\times x)(B\cdot x)\Del\psi\cdot
\frac1r\patl_r\Big(\frac1r\patl_r\Del\psi\Big)\\
\end{split}\]

\[\label{12Ruo}\begin{split}
=&-(A\times x)(A\cdot x)\frac1r\patl_r\phi\cdot
\frac1r\patl_r\Big(\frac1r\patl_r\phi\Big)\\
&+A(B\cdot x)\frac1r\patl_r\phi\cdot\frac1r\patl_r\Del\psi\\
&+A(B\cdot x)\frac2r\patl_r\psi\cdot
\frac1r\patl_r\Big(\frac1r\patl_r\phi\Big)\\
&-A(B\cdot x)\frac2r\patl_r\psi\cdot\frac1r\patl_r\Del\phi\\
&-x2(A\cdot x)(B\cdot x)\frac1r\patl_r\Big(\frac1r\patl_r\psi\Big)
\cdot\frac1r\patl_r\Big(\frac1r\patl_r\phi\Big)\\
&+x(A\cdot x)(B\cdot x)\frac2r\patl_r\psi\cdot
\frac1r\patl_r\Big\{\frac1r\patl_r\Big(\frac1r\patl_r\phi\Big)\Big\}\\
&-x(A\cdot B)\frac1r\patl_r\phi\cdot\frac1r\patl_r\Del\psi\\
&-x(A\cdot B)\frac1r\patl_r\psi\cdot
\frac1r\patl_r\Big(\frac1r\patl_r\phi\Big)\\
&+x(A\cdot B)\Del\psi\cdot\frac1r\patl_r\Big(\frac1r\patl_r\phi\Big)\\
&-B(A\cdot x)\frac1r\patl_r\psi\cdot
\frac1r\patl_r\Big(\frac1r\patl_r\phi\Big)\\
&+B(A\cdot x)\Del\psi\cdot\frac1r\patl_r\Big(\frac1r\patl_r\phi\Big)\\
&+(B\times x)(B\cdot x)\frac1r\patl_r\Big(\frac1r\patl_r\psi\Big)\cdot
\frac1r\patl_r\Del\psi\\
&-(B\times x)(B\cdot x)\frac2r\patl_r\psi\cdot
\frac1r\patl_r\Big(\frac1r\patl_r\Del\psi\Big),\\
\end{split}\]

\beq\label{12R-NTou}\begin{split}
(\ome\cdot\nab)u=&\Big\{-(A\cdot x)\frac1r\patl_r\Big(\frac1r\patl_r
\phi\Big)\{x\cdot\nab\}-\frac1r\patl_r\phi\{A\cdot\nab\}\\
&+\Del\phi\{A\cdot\nab\}+\frac1r\patl_r\Del\psi\{(B\times x)\cdot\nab\}\Big\}\\
&\Big\{(A\times x)\frac1r\patl_r\phi+(B\cdot x)x\frac1r\patl_r\Big(
\frac1r\patl_r\psi\Big)+B\frac1r\patl_r\psi-B\Del\psi\Big\}\\
\end{split}\eeq
\[\label{12Rou}\begin{split}
=&-(A\times x)(A\cdot x)\frac1r\patl_r\Big(\frac1r\patl_r\phi\Big)\cdot
\frac1r\patl_r\phi\\
&-(A\times x)(A\cdot x)\patl_r\Big(\frac1r\patl_r\phi\Big)\cdot
\patl_r\Big(\frac1r\patl_r\phi\Big)\\
&-2x(A\cdot x)(B\cdot x)\frac1r\patl_r\Big(\frac1r\patl_r\phi\Big)\cdot
\frac1r\patl_r\Big(\frac1r\patl_r\psi\Big)\\
&-x(A\cdot x)(B\cdot x)\patl_r\Big(\frac1r\patl_r\phi\Big)\cdot
\patl_r\Big\{\frac1r\patl_r\Big(\frac1r\patl_r\psi\Big)\Big\}\\
&-B(A\cdot x)\patl_r\Big(\frac1r\patl_r\phi\Big)\cdot
\patl_r\Big(\frac1r\patl_r\psi\Big)\\
&+B(A\cdot x)\patl_r\Big(\frac1r\patl_r\phi\Big)\cdot\patl_r\Del\psi\\
&-(A\times x)(A\cdot x)\frac1r\patl_r\phi\cdot
\frac1r\patl_r\Big(\frac1r\patl_r\phi\Big)\\
&-x(A\cdot B)\frac1r\patl_r\phi\cdot\frac1r\patl_r\Big(\frac1r\patl_r\psi\Big)\\
&-A(B\cdot x)\frac1r\patl_r\phi\cdot\frac1r\patl_r\Big(\frac1r\patl_r\psi\Big)\\
&-x(A\cdot x)(B\cdot x)\frac1r\patl_r\phi\cdot
\frac1r\patl_r\Big\{\frac1r\patl_r\Big(\frac1r\patl_r\psi\Big)\Big\}\\
&-B(A\cdot x)\frac1r\patl_r\phi\cdot\frac1r\patl_r\Big(\frac1r\patl_r\psi\Big)\\
&+B(A\cdot x)\frac1r\patl_r\phi\cdot\frac1r\patl_r\Del\psi\\
&+(A\times x)(A\cdot x)\Del\phi\cdot\frac1r\patl_r\Big(\frac1r\patl_r\phi\Big)\\
&+x(A\cdot B)\Del\phi\cdot\frac1r\patl_r\Big(\frac1r\patl_r\psi\Big)\\
&+A(B\cdot x)\Del\phi\cdot\frac1r\patl_r\Big(\frac1r\patl_r\psi\Big)\\
&+x(A\cdot x)(B\cdot x)\Del\phi\cdot\frac1r\patl_r\Big\{
\frac1r\patl_r\Big(\frac1r\patl_r\psi\Big)\Big\}\\
&+B(A\cdot x)\Del\phi\cdot\frac1r\patl_r\Big(\frac1r\patl_r\psi\Big)\\
&-B(A\cdot x)\Del\phi\cdot\frac1r\patl_r\Del\psi\\
&+(A\times(B\times x))\frac1r\patl_r\Del\psi\cdot\frac1r\patl_r\phi\\
&+(B\times x)(B\cdot x)\frac1r\patl_r\Del\psi\cdot
\frac1r\patl_r\Big(\frac1r\patl_r\psi\Big)\\
\end{split}\]

\[\label{12Rou}\begin{split}
=&(A\times x)(A\cdot x)\frac1r\patl_r\phi\cdot
\frac1r\patl_r\Big(\frac1r\patl_r\phi\Big)\\
&-A(B\cdot x)\frac1r\patl_r\phi\cdot\frac1r\patl_r\Big(\frac1r\patl_r\psi\Big)\\
&+A(B\cdot x)\Del\phi\cdot\frac1r\patl_r\Big(\frac1r\patl_r\psi\Big)\\
&-2x(A\cdot x)(B\cdot x)\frac1r\patl_r\Big(\frac1r\patl_r\phi\Big)\cdot
\frac1r\patl_r\Big(\frac1r\patl_r\psi\Big)\\
&+x(A\cdot x)(B\cdot x)\frac2r\patl_r\phi\cdot\frac1r\patl_r\Big\{
\frac1r\patl_r\Big(\frac1r\patl_r\psi\Big)\Big\}\\
&-x(A\cdot B)\frac1r\patl_r\phi\cdot\frac1r\patl_r\Big(\frac1r\patl_r\psi\Big)\\
&+x(A\cdot B)\Del\phi\cdot\frac1r\patl_r\Big(\frac1r\patl_r\psi\Big)\\
&-x(A\cdot B)\frac1r\patl_r\phi\cdot\frac1r\patl_r\Del\psi  \\
&+B(A\cdot x)\frac2r\patl_r\phi\cdot\frac1r\patl_r\Big(\frac1r\patl_r\psi\Big)\\
&-B(A\cdot x)\frac1r\patl_r\phi\cdot\frac1r\patl_r\Del\psi\\
&+(B\times x)(B\cdot x)\frac1r\patl_r\Del\psi\cdot
\frac1r\patl_r\Big(\frac1r\patl_r\psi\Big),\\
\end{split}\]

\beq\label{12R-NTuou-at}\begin{split}
&(A\times\nab)\cdot\{(u\cdot\nab)\ome-(\ome\cdot\nab)u\}\\
=&(A\times\nab)\cdot\Big\{-2(A\times x)(A\cdot x)\frac1r\patl_r\phi\cdot
\frac1r\patl_r\Big(\frac1r\patl_r\phi\Big)\\
&+A(B\cdot x)\frac1r\patl_r\phi\cdot\frac1r\patl_r\Del\psi\\
&+A(B\cdot x)\frac2r\patl_r\psi\cdot
\frac1r\patl_r\Big(\frac1r\patl_r\phi\Big)\\
&-A(B\cdot x)\frac2r\patl_r\psi\cdot\frac1r\patl_r\Del\phi\\
&+A(B\cdot x)\frac1r\patl_r\phi\cdot\frac1r\patl_r\Big(\frac1r\patl_r\psi\Big)\\
&-A(B\cdot x)\Del\phi\cdot\frac1r\patl_r\Big(\frac1r\patl_r\psi\Big)\\
&+x(A\cdot x)(B\cdot x)\frac2r\patl_r\psi\cdot
\frac1r\patl_r\Big\{\frac1r\patl_r\Big(\frac1r\patl_r\phi\Big)\Big\}\\
&-x(A\cdot x)(B\cdot x)\frac2r\patl_r\phi\cdot\frac1r\patl_r\Big\{
\frac1r\patl_r\Big(\frac1r\patl_r\psi\Big)\Big\}\\
&-x(A\cdot B)\frac1r\patl_r\psi\cdot\frac1r\patl_r\Big(\frac1r\patl_r\phi\Big)\\
&+x(A\cdot B)\Del\psi\cdot\frac1r\patl_r\Big(\frac1r\patl_r\phi\Big)\\
&+x(A\cdot B)\frac1r\patl_r\phi\cdot\frac1r\patl_r\Big(\frac1r\patl_r\psi\Big)\\
&-x(A\cdot B)\Del\phi\cdot\frac1r\patl_r\Big(\frac1r\patl_r\psi\Big)\\
&-B(A\cdot x)\frac1r\patl_r\psi\cdot
\frac1r\patl_r\Big(\frac1r\patl_r\phi\Big)\\
&+B(A\cdot x)\Del\psi\cdot\frac1r\patl_r\Big(\frac1r\patl_r\phi\Big)\\
&-B(A\cdot x)\frac2r\patl_r\phi\cdot\frac1r\patl_r\Big(\frac1r\patl_r\psi\Big)\\
&+B(A\cdot x)\frac1r\patl_r\phi\cdot\frac1r\patl_r\Del\psi\\
&-(B\times x)(B\cdot x)\frac2r\patl_r\psi\cdot
\frac1r\patl_r\Big(\frac1r\patl_r\Del\psi\Big)\Big\}\\
\end{split}\eeq
\[\label{uou-at}\begin{split}
=&-(A\cdot A)(A\cdot x)\frac4r\patl_r\phi\cdot
\frac1r\patl_r\Big(\frac1r\patl_r\phi\Big)\\
&-\{(A\cdot A)r^2-(A\cdot x)^2\}(A\cdot x)\frac2r\patl_r\Big\{\frac1r
\patl_r\phi\cdot\frac1r\patl_r\Big(\frac1r\patl_r\phi\Big)\Big\}\\
&+\{(A\times B)\cdot x\}(A\cdot x)\frac2r\patl_r\psi\cdot
\frac1r\patl_r\Big\{\frac1r\patl_r\Big(\frac1r\patl_r\phi\Big)\Big\}\\
&-\{(A\times B)\cdot x\}(A\cdot x)\frac2r\patl_r\phi\cdot\frac1r\patl_r
\Big\{\frac1r\patl_r\Big(\frac1r\patl_r\psi\Big)\Big\}\\
&+\{(A\times B)\cdot x\}(A\cdot x)\frac1r\patl_r\Big\{\frac1r\patl_r
\psi\cdot\frac1r\patl_r\Big(\frac1r\patl_r\phi\Big)\Big\}\\
&-\{(A\times B)\cdot x\}(A\cdot x)\frac1r\patl_r\Big\{\Del\psi\cdot
\frac1r\patl_r\Big(\frac1r\patl_r\phi\Big)\Big\}\\
&+\{(A\times B)\cdot x\}(A\cdot x)\frac1r\patl_r\Big\{\frac2r\patl_r
\phi\cdot\frac1r\patl_r\Big(\frac1r\patl_r\psi\Big)\Big\}\\
&-\{(A\times B)\cdot x\}(A\cdot x)\frac1r\patl_r\Big\{\frac1r\patl_r
\phi\cdot\frac1r\patl_r\Del\psi\Big\}\\
&-(A\cdot B)(B\cdot x)\frac4r\patl_r\psi\cdot
\frac1r\patl_r\Big(\frac1r\patl_r\Del\psi\Big)\\
&-\{(A\cdot B)(B\cdot x)-(B\cdot B)(A\cdot x)\}\frac2r\patl_r\psi\cdot
\frac1r\patl_r\Big(\frac1r\patl_r\Del\psi\Big)\\
&-\{(A\cdot B)r^2-(A\cdot x)(B\cdot x)\}(B\cdot x)\frac2r\patl_r\Big\{\frac1r
\patl_r\psi\cdot\frac1r\patl_r\Big(\frac1r\patl_r\Del\psi\Big)\Big\}.\\
\end{split}\]

Putting \eqref{12R-NTuu-bt} into \eqref{NS-SR12-phi}, we have
\beq\label{NS12-Rphi}\begin{split}
&\Big\{(A\cdot x)(B\cdot x)\frac1r\patl_r\Big(\frac1r\patl_r\Big)
-(A\cdot B)\{\frac2r\patl_r+r\patl_r\Big(\frac1r\patl_r\Big)\}\Big\}
\{\phi_t-\nu\Del\phi\}\\
=&\{(A\times B)\cdot x\}(A\cdot x)\frac1r\patl_r
\Big(\frac1r\patl_r\phi\cdot\frac1r\patl_r\phi\Big)\\
&+\{(A\cdot B)(B\cdot x)-(B\cdot B)(A\cdot x)\}\frac1r\patl_r\phi\cdot
\frac1r\patl_r\Big(\frac1r\patl_r\psi\Big)\\
&+4(A\cdot B)(B\cdot x)\frac1r\patl_r\phi\cdot
\frac1r\patl_r\Big(\frac1r\patl_r\psi\Big)\\
&+2\{(A\cdot B)r^2-(A\cdot x)(B\cdot x)\}(B\cdot x)\frac1r\patl_r\Big\{
\frac1r\patl_r\phi\cdot\frac1r\patl_r\Big(\frac1r\patl_r\psi\Big)\Big\}\\
&+2(A\cdot B)(B\cdot x)\patl_r\Big(\frac1r
\patl_r\phi\Big)\cdot\patl_r\Big(\frac1r\patl_r\psi\Big)\\
&+\{(A\cdot B)r^2-(A\cdot x)(B\cdot x)\}(B\cdot x)\frac1r\patl_r\Big\{
\patl_r\Big(\frac1r\patl_r\phi\Big)\cdot\patl_r
\Big(\frac1r\patl_r\psi\Big)\Big\}\\
&+\{(A\cdot B)(B\cdot x)-(B\cdot B)(A\cdot x)\}\frac1r\patl_r
\Big\{\frac1r\patl_r\phi\cdot\frac1r\patl_r\psi\Big\}\\
&+2(A\cdot B)(B\cdot x)\frac1r\patl_r\psi\cdot
\frac1r\patl_r\Big(\frac1r\patl_r\phi\Big)\\
&+\{(A\cdot B)r^2-(A\cdot x)(B\cdot x)\}(B\cdot x)\frac1r\patl_r\Big\{
\frac1r\patl_r\psi\cdot\frac1r\patl_r\Big(\frac1r\patl_r\phi\Big)\Big\}\\
&-\{(A\cdot B)(B\cdot x)-(B\cdot B)(A\cdot x)\}\frac1r\patl_r
\Big\{\Del\psi\cdot\frac1r\patl_r\phi\Big\}\\
&-2(A\cdot B)(B\cdot x)\Del\psi\cdot
\frac1r\patl_r\Big(\frac1r\patl_r\phi\Big)\\
&-\{(A\cdot B)r^2-(A\cdot x)(B\cdot x)\}(B\cdot x)\frac1r\patl_r\Big\{
\Del\psi\cdot\frac1r\patl_r\Big(\frac1r\patl_r\phi\Big)\Big\}.\\
\end{split}\eeq

Putting \eqref{12R-NTuou-at} into \eqref{NS-SR12-psi}, we derive
\beq\label{NS12-Rpsi}\begin{split}
&\Big\{(A\cdot x)(B\cdot x)\frac1r\patl_r\Big(\frac1r\patl_r\Big)
-(A\cdot B)\{\frac2r\patl_r+r\patl_r\Big(\frac1r\patl_r\Big)\}\Big\}
\Del\{\psi_t-\nu\Del\psi\}\\
=&-(A\cdot A)(A\cdot x)\frac4r\patl_r\phi\cdot
\frac1r\patl_r\Big(\frac1r\patl_r\phi\Big)\\
&-\{(A\cdot A)r^2-(A\cdot x)^2\}(A\cdot x)\frac2r\patl_r\Big\{\frac1r
\patl_r\phi\cdot\frac1r\patl_r\Big(\frac1r\patl_r\phi\Big)\Big\}\\
&+\{(A\times B)\cdot x\}(A\cdot x)\frac2r\patl_r\psi\cdot
\frac1r\patl_r\Big\{\frac1r\patl_r\Big(\frac1r\patl_r\phi\Big)\Big\}\\
&-\{(A\times B)\cdot x\}(A\cdot x)\frac2r\patl_r\phi\cdot\frac1r\patl_r
\Big\{\frac1r\patl_r\Big(\frac1r\patl_r\psi\Big)\Big\}\\
&+\{(A\times B)\cdot x\}(A\cdot x)\frac1r\patl_r\Big\{\frac1r\patl_r
\psi\cdot\frac1r\patl_r\Big(\frac1r\patl_r\phi\Big)\Big\}\\
&-\{(A\times B)\cdot x\}(A\cdot x)\frac1r\patl_r\Big\{\Del\psi\cdot
\frac1r\patl_r\Big(\frac1r\patl_r\phi\Big)\Big\}\\
&+\{(A\times B)\cdot x\}(A\cdot x)\frac1r\patl_r\Big\{\frac2r\patl_r
\phi\cdot\frac1r\patl_r\Big(\frac1r\patl_r\psi\Big)\Big\}\\
&-\{(A\times B)\cdot x\}(A\cdot x)\frac1r\patl_r\Big\{\frac1r\patl_r
\phi\cdot\frac1r\patl_r\Del\psi\Big\}\\
&-(A\cdot B)(B\cdot x)\frac4r\patl_r\psi\cdot
\frac1r\patl_r\Big(\frac1r\patl_r\Del\psi\Big)\\
&-\{(A\cdot B)(B\cdot x)-(B\cdot B)(A\cdot x)\}\frac2r\patl_r\psi\cdot
\frac1r\patl_r\Big(\frac1r\patl_r\Del\psi\Big)\\
&-\{(A\cdot B)r^2-(A\cdot x)(B\cdot x)\}(B\cdot x)\frac2r\patl_r\Big\{\frac1r
\patl_r\psi\cdot\frac1r\patl_r\Big(\frac1r\patl_r\Del\psi\Big)\Big\}.\\
\end{split}\eeq

\begin{proof}[Proof of Theorem \ref{SR12-AEB-Thm}]
Firstly we consider the case which vectors $A$ and $B$ are linearly
dependent. Without loss of generality, we assume $A=B$. The equations
\eqref{NS12-Rphi} \eqref{NS12-Rpsi} are as follows
\beq\label{12-Rph-AEB}\begin{split}
&\Big\{(A\cdot x)^2\frac1r\patl_r-(A\cdot A)\{2+r\patl_r\}\Big\}
\frac1r\patl_r\{\phi_t-\nu\Del\phi\}\\
=&-(A\cdot x)\Big\{(A\cdot x)^2\frac1r\patl_r-(A\cdot A)\{2+r\patl_r\}\Big\}
\frac2r\patl_r\phi\cdot\frac1r\patl_r\Big(\frac1r\patl_r\psi\Big)\\
&-(A\cdot x)\Big\{(A\cdot x)^2\frac1r\patl_r-(A\cdot A)\{2+r\patl_r\}\Big\}
\patl_r\Big(\frac1r\patl_r\phi\Big)\cdot
\patl_r\Big(\frac1r\patl_r\psi\Big)\\
&-(A\cdot x)\Big\{(A\cdot x)^2\frac1r\patl_r-(A\cdot A)\{2+r\patl_r\}\Big\}
\frac1r\patl_r\psi\cdot\frac1r\patl_r\Big(\frac1r\patl_r\phi\Big)\\
&+(A\cdot x)\Big\{(A\cdot x)^2\frac1r\patl_r-(A\cdot A)\{2+r\patl_r\}\Big\}
\Del\psi\cdot\frac1r\patl_r\Big(\frac1r\patl_r\phi\Big)\\
=&\Big\{(A\cdot x)^2\frac1r\patl_r-(A\cdot A)\{2+r\patl_r\}\Big\}
\frac1r\patl_r\Big((A\cdot x)\cdot\\
&\int^r_0s\Big\{\frac2s\patl_s\psi\cdot\frac1s\patl_s\Big(\frac1s
\patl_s\phi\Big)-\frac2s\patl_s\phi\cdot\frac1s\patl_s\Big(
\frac1s\patl_s\psi\Big)\Big\}ds\Big),\\
\end{split}\eeq
\beq\label{12-Rps-AEB}\begin{split}
&\Big\{(A\cdot x)^2\frac1r\patl_r-(A\cdot A)\{2+r\patl_r\}\Big\}
\frac1r\patl_r\Del\{\psi_t-\nu\Del\psi\}\\
=&\Big\{(A\cdot x)^2\frac1r\patl_r-(A\cdot A)\{2+r\patl_r\}\Big\}
\frac1r\patl_r\Big((A\cdot x)\cdot\\
&\int^r_0s\Big\{\frac2s\patl_s\phi\cdot\frac1s\patl_s\Big(\frac1s
\patl_s\phi\Big)+\frac2s\patl_s\psi\cdot\frac1s\patl_s
\Big(\frac1s\patl_s\Del\psi\Big)\Big\}ds\Big),\\
\end{split}\eeq
where we have used the following facts
\beq\label{ATN-adx}
(A\times\nab)(A\cdot x)=A\times A=0,
\eeq
\beq\label{DATN-RF}\begin{split}
&(A\times\nab)\cdot(A\times\nab)\varphi(r)\\
=&\Big\{(A\cdot x)^2\frac1r\patl_r-(A\cdot A)\{2+r\patl_r\}\Big\}
\frac1r\patl_r\varphi(r)
\end{split}\eeq
for any radial function $\varphi(r)$. The equations \eqref{12-Rph-AEB}
\eqref{12-Rps-AEB} imply that
\beq\label{12-Rph-AEB1}\begin{split}
&\phi_t-\nu\Del\phi\\
=&(A\cdot x)\int^r_0s\Big\{\frac2s\patl_s\psi\cdot\frac1s\patl_s\Big(
\frac1s\patl_s\phi\Big)-\frac2s\patl_s\phi\cdot\frac1s\patl_s\Big(
\frac1s\patl_s\psi\Big)\Big\}ds,\\
\end{split}\eeq
\beq\label{12-Rps-AEB1}\begin{split}
&\Del\{\psi_t-\nu\Del\psi\}\\
=&(A\cdot x)\int^r_0s\Big\{\frac2s\patl_s\phi\cdot\frac1s\patl_s\Big(
\frac1s\patl_s\phi\Big)+\frac2s\patl_s\psi\cdot\frac1s\patl_s
\Big(\frac1s\patl_s\Del\psi\Big)\Big\}ds.\\
\end{split}\eeq

Let us select orthogonal transformation $\rho$ as follows
\beq\label{SR12-OT}
y=\rho x=x\left(\begin{array}{ccc}
0&0&1\\1&0&0\\0&1&0\\
\end{array}\right)=(x_2,x_3,x_1).
\eeq
Then $r^2=y\cdot y=\rho x\cdot\rho x=x\cdot x$.

Applying the orthogonal transformation \eqref{SR12-OT} in the equations
\eqref{12-Rph-AEB1} \eqref{12-Rps-AEB1}, we obtain
\beq\label{12-Rph-AEB1y}\begin{split}
&\phi_t-\nu\Del\phi\\
=&(A\cdot y)\int^r_0s\Big\{\frac2s\patl_s\psi\cdot\frac1s\patl_s\Big(
\frac1s\patl_s\phi\Big)-\frac2s\patl_s\phi\cdot\frac1s\patl_s\Big(
\frac1s\patl_s\psi\Big)\Big\}ds,\\
\end{split}\eeq
\beq\label{12-Rps-AEB1y}\begin{split}
&\Del\{\psi_t-\nu\Del\psi\}\\
=&(A\cdot y)\int^r_0s\Big\{\frac2s\patl_s\phi\cdot\frac1s\patl_s\Big(
\frac1s\patl_s\phi\Big)+\frac2s\patl_s\psi\cdot\frac1s\patl_s
\Big(\frac1s\patl_s\Del\psi\Big)\Big\}ds,\\
\end{split}\eeq
where $y=\rho x$. Employing the equations \eqref{12-Rph-AEB1}
\eqref{12-Rph-AEB1y}, we get
\beq\label{12ph-AEB1yx}\begin{split}
&(A\cdot(\rho x-x))\int^r_0s\Big\{\frac2s\patl_s\psi\cdot\frac1s\patl_s\Big(
\frac1s\patl_s\phi\Big)-\frac2s\patl_s\phi\cdot\frac1s\patl_s\Big(
\frac1s\patl_s\psi\Big)\Big\}ds=0.\\
\end{split}\eeq
Similarly using \eqref{12-Rps-AEB1} \eqref{12-Rps-AEB1y}, we derive
\beq\label{12ps-AEB1yx}\begin{split}
&(A\cdot(\rho x-x))\int^r_0s\Big\{\frac2s\patl_s\phi\cdot\frac1s\patl_s\Big(
\frac1s\patl_s\phi\Big)+\frac2s\patl_s\psi\cdot\frac1s\patl_s
\Big(\frac1s\patl_s\Del\psi\Big)\Big\}ds=0.\\
\end{split}\eeq
Given $r$, thanks $x\in{\mathbb S}^2_r$ is arbitrary, the
equations \eqref{12ph-AEB1yx} and \eqref{12ps-AEB1yx} imply that
\beq\label{12ph-AEB2}\begin{split}
\frac1r\patl_r\psi\cdot\patl_r\Big(\frac1r\patl_r\phi\Big)-
\frac1r\patl_r\phi\cdot\patl_r\Big(\frac1r\patl_r\psi\Big)=0,\\
\end{split}\eeq
\beq\label{12ps-AEB2}\begin{split}
\frac1r\patl_r\phi\cdot\patl_r\Big(\frac1r\patl_r\phi\Big)+
\frac1r\patl_r\psi\cdot\patl_r\Big(\frac1r\patl_r\Del\psi\Big)=0.\\
\end{split}\eeq
Putting \eqref{12ph-AEB2} into \eqref{12-Rph-AEB1} and putting
\eqref{12ps-AEB2} into \eqref{12-Rps-AEB1}, we derive that
\beq\label{12-Rph-AEB3}\begin{split}
&\phi_t-\nu\Del\phi=0,\\
\end{split}\eeq
\beq\label{12-Rps-AEB3}\begin{split}
&\Del\{\psi_t-\nu\Del\psi\}=0.\\
\end{split}\eeq

Let $(\phi,\psi)=\big(\Phi(r),\Psi(r)\big)$ be solution of
\eqref{12ph-AEB2} \eqref{12ps-AEB2}. Then
\[
u=(A\times\nab)\Phi+\{(A\times\nab)\times\nab\}\Psi
\]
satisfies static three dimensional Euler equation
\[\begin{split}
&(u\cdot\nab)u+\nab P=0,\\
&\nab\cdot u=0
\end{split}\]
by all calculations of \eqref{SR12-R3}--\eqref{12ps-AEB2}.

Result (I) is proved.

Provided
\beq\label{12phs-R100}\begin{split}
(\patl_r\phi,\patl_r\psi)=(0,0),
\end{split}\eeq
then equations \eqref{12ph-AEB2}--\eqref{12-Rps-AEB3} are satisfied.
Here the velocity $u=0$. This is trivial.

Provided $(\patl_r\phi,\patl_r\psi)\not=(0,0)$, then the equation
\eqref{12ph-AEB2} means that
\beq\label{12phs-R1}\begin{split}
\phi=h(t)\psi,\\
\end{split}\eeq
where $h(t)$ is any function of $t$. Putting \eqref{12phs-R1} into
\eqref{12ps-AEB2}, we have
\beq\label{12ps-AEB4}\begin{split}
h^2(t)\patl_r\Big(\frac1r\patl_r\psi\Big)+
\patl_r\Big(\frac1r\patl_r\Del\psi\Big)=0.\\
\end{split}\eeq

Provided $\patl_r\phi=0$ and $\patl_r\psi\not=0$, the equation
\eqref{12ph-AEB2} is satisfied. The equation \eqref{12ps-AEB2}
means that
\beq\label{12ps-AEB02}\begin{split}
\Big(\patl_r\phi,\patl_r(\frac1r\patl_r\Del\psi)\Big)=(0,0).\\
\end{split}\eeq

Provided $\patl_r\phi\not=0$ and $\patl_r\psi=0$, the equation
\eqref{12ph-AEB2} is satisfied. The equation \eqref{12ps-AEB2}
means that
\beq\label{12ph-AEB20}\begin{split}
\Big(\patl_r(\frac1r\patl_r\phi),\patl_r\psi\Big)=(0,0).\\
\end{split}\eeq

There exists a unique solution
\beq\label{12phs-S-CS0}\begin{split}
&\phi(t,r)=e^{-\nu\Del t}\phi_0(r),\\
&\psi(t,r)=e^{-\nu\Del t}\psi_0(r), \\
\end{split}\eeq
of equations \eqref{12-Rph-AEB3} \eqref{12-Rps-AEB3} with initial data
\[\begin{split}
&\phi(t,r)|_{t=0}=\phi_0(r),\\
&\psi(t,r)|_{t=0}=\psi_0(r). \\
\end{split}\]

\begin{lemma}\label{SolEq-Vrr}
Let $V=f(t)r^2+g(t)$, $f(t)$ and $g(t)$ be any functions of $t$. Then
$V$ is solution of the following equation
\beq\label{Eq-Vrr}\begin{split}
\patl_r\Big(\frac1r\patl_rV\Big)=0.\\
\end{split}\eeq
Moreover we have
\[
\tilde{V}=e^{-\nu\Del t}\{f(0)r^2+g(0)\}=f(0)r^2+6\nu f(0)t+g(0),
\]
and $\tilde{V}$ is also solution of the equation \eqref{Eq-Vrr}.
\end{lemma}

Take $(\phi_0,\psi_0)=\big(\Phi(r),\Psi(r)\big)$. Then
\beq\label{12phs-S-CS}\begin{split}
&\phi(t,r)=e^{-\nu\Del t}\phi_0(r)
=e^{-\nu\Del t}\Phi(r),\\
&\psi(t,r)=e^{-\nu\Del t}\psi_0(r)=e^{-\nu\Del t}\Psi(r). \\
\end{split}\eeq
Since $(\phi,\psi)=(\phi_0,\psi_0)$ satisfies equations
\eqref{12ph-AEB2} \eqref{12ps-AEB2}, this vector satisfies equation
\eqref{12phs-R100} or equations \eqref{12phs-R1} \eqref{12ps-AEB4}
with $h(t)=h(0)$ or \eqref{12ps-AEB02} or \eqref{12ph-AEB20}. By using
Lemma \ref{SolEq-Vrr} and the following fact
\[
(\lam-\Del)e^{-\nu\Del t}=e^{-\nu\Del t}(\lam-\Del),\;\;\lam\;
\mbox{ is constant},
\]
we derive that $\big(\phi(t,r),\psi(t,r)\big)$ defined by
\eqref{12phs-S-CS} also satisfies respective equation of
\eqref{12phs-R100}--\eqref{12ph-AEB20}. Thus $\big(\phi(t,r),\psi(t,r)
\big)$ satisfies equations \eqref{12ph-AEB2} \eqref{12ps-AEB2}. This
means that $\big(\phi(t,r),\psi(t,r)\big)$ is solution of the equations
\eqref{12-Rph-AEB} \eqref{12-Rps-AEB}.

Result (II) is proved.

On the other hand, let $(\phi_0,\psi_0)\not=\big(\Phi(r),\Psi(r)
\big)$ and $(\phi_0,\psi_0)$ be regular enough if it is necessary.
By theory of local well-posedness of Navier-Stokes equations, there
exist $T_{max}>0$ and a unique solution $\big(\phi(t,x),\psi(t,x)\big)
$ of equations \eqref{NS-SR12-phi} \eqref{NS-SR12-psi} with $A=B$
such that
\[\begin{split}
&\phi(t,x)\in C([0,T_{max});H^m),\\
&\psi(t,x)\in C([0,T_{max});H^{m+1}), \\
\end{split}\]
where and $m>0$ large enough.

Assume that there exist $t_n>0$ such that $t_n\rightarrow0$ as $n
\rightarrow\infty$ and $\big(\phi(t_n,x),\psi(t_n,x)\big)$ is radial
function of $x$. Then $\big(\phi(t_n,x),\psi(t_n,x)\big)$ satisfies
equations \eqref{12ph-AEB2} \eqref{12ps-AEB2} for any $t_n$. There
exists a unique solution
\[\begin{split}
&\phi^{\prime}(t,x)=\phi^{\prime}(t,r)=e^{-\nu\Del t}\phi(t_n,x)
=e^{-\nu\Del t}\phi(t_n,r),\;\;\forall t\ge0,\\
&\psi^{\prime}(t,x)=\psi^{\prime}(t,r)=e^{-\nu\Del t}\psi(t_n,x)
=e^{-\nu\Del t}\psi(t_n,r),\;\;\forall t\ge0 \\
\end{split}\]
of equations \eqref{NS-SR12-phi} \eqref{NS-SR12-psi} with initial data
$\big(\phi(t_n,x),\psi(t_n,x)\big)$. This solution $\phi^{\prime}$ is
global and radial. Thanks the uniqueness of solution $\big(\phi(t,x),
\psi(t,x)\big)$, then
\beq\label{12phs-SCS-tn}\begin{split}
&\phi(t_n+t,x)=\phi(t_n+t,r)=e^{-\nu\Del t}\phi(t_n,x)
=e^{-\nu\Del t}\phi(t_n,r),\;\;\forall t\ge0,\\
&\psi(t_n+t,x)=\psi(t_n+t,r)=e^{-\nu\Del t}\psi(t_n,x)
=e^{-\nu\Del t}\psi(t_n,r),\;\;\forall t\ge0. \\
\end{split}\eeq
Therefore this solution $\big(\phi(t,x),\psi(t,x)\big)$ is global with
respect to $t>0$ and radial with respect to $x$ for any $t\ge t_n$.

Letting $n\rightarrow\infty$ and $t_n\rightarrow0$, we derive $(\phi,
\psi)=(\phi_0,\psi_0)$ satisfies equations \eqref{12ph-AEB2}
\eqref{12ps-AEB2} since $\big(\phi(t,x),\psi(t,x)\big)$ is continuous
with respect to $t\ge 0$. This is contradictory.

In summary, provided there exists $t_0\in (0,T_{max})$ such that
solution $\big(\phi(t,x),\psi(t,x)\big)$ is radial with respect to $x$
at $t=t_0$, then $T_{max}=\infty$ and this solution is radial with
respect to $x$ for any $t\ge t_0$. Let $T_r=\min\{t_0\}$. Thus this
solution is radial with respect to $x$ for any $t\ge T_r$, but this
solution is not radial with respect to $x$ for any $t\in(0,T_r)$.
Otherwise this solution can not be radial with respect to $x$ for any
$t\in (0,T_{max})$.

Theorem \ref{SR12-AEB-Thm} is proved.
\end{proof}

\begin{proof}[Proof of Corollary \ref{SC-EE-PBC}]
Now we consider the following equations
\beq\label{12phs-R1S}\begin{split}
\phi=\lam\psi,\\
\end{split}\eeq
\beq\label{12ps-AEB4S}\begin{split}
\lam^2\psi+\Del\psi=0,\\
\end{split}\eeq
where $\lam$ is any real constant. Equations \eqref{12phs-R1S}
\eqref{12ps-AEB4S} are the special case of equations \eqref{12phs-R1}
\eqref{12ps-AEB4}.

Equation \eqref{12ps-AEB4S} can be written
\beq\label{12ps-AA1}\begin{split}
\lam^2(r\psi)+\patl^2_r(r\psi)=0.\\
\end{split}\eeq
There exists solution
\[
\psi=\Psi_{\lam\al\beta}(r)=\al\frac1r\sin(\lam r)
+\beta\frac1r\cos(\lam r)
\]
of equation \eqref{12ps-AA1} for any real constants $\lam, \al, \beta$.

Let $\Phi_{\lam\al\beta}(r)=\lam\Psi_{\lam\al\beta}(r)$. Then $
(\phi,\psi)=\Big(\Phi_{\lam\al\beta}(r),\Psi_{\lam\al\beta}(r)\Big)$
satisfies equations \eqref{12ph-AEB2} \eqref{12ps-AEB2}.

By Theorem \ref{SR12-AEB-Thm}, this corollary is proved.
\end{proof}

\begin{proof}[Proof of Theorem \ref{SR12-AVB-Thm}]
Now we consider the case which vectors $A$ and $B$ are perpendicular
$A\bot B$. The equations \eqref{NS12-Rphi} \eqref{NS12-Rpsi} are
as follows
\beq\label{12-Rph-AVB}\begin{split}
&(B\cdot x)\frac1r\patl_r\Big(\frac1r\patl_r\Big)
\{\phi_t-\nu\Del\phi\}\\
=&\{(A\times B)\cdot x\}\frac1r\patl_r
\Big(\frac1r\patl_r\phi\cdot\frac1r\patl_r\phi\Big)\\
&-(B\cdot B)\frac1r\patl_r\phi\cdot
\frac1r\patl_r\Big(\frac1r\patl_r\psi\Big)\\
&-(B\cdot B)\frac1r\patl_r\Big\{\frac1r\patl_r\phi\cdot
\frac1r\patl_r\psi\Big\}\\
&+(B\cdot B)\frac1r\patl_r\Big\{\frac1r\patl_r\phi\cdot
\Del\psi\Big\}\\
&-(B\cdot x)^2\frac1r\patl_r\Big\{\frac2r\patl_r\phi\cdot
\frac1r\patl_r\Big(\frac1r\patl_r\psi\Big)\Big\}\\
&+(B\cdot x)^2\frac1r\patl_r\Big\{\frac2r\patl_r\psi\cdot
\frac1r\patl_r\Big(\frac1r\patl_r\phi\Big)\Big\},\\
\end{split}\eeq
\beq\label{12-Rps-AVB}\begin{split}
&(B\cdot x)\frac1r\patl_r\Big(\frac1r\patl_r\Big)
\Del\{\psi_t-\nu\Del\psi\}\\
=&-(A\cdot A)\frac4r\patl_r\phi\cdot
\frac1r\patl_r\Big(\frac1r\patl_r\phi\Big)\\
&-\{(A\cdot A)r^2-(A\cdot x)^2\}\frac2r\patl_r\Big\{\frac1r
\patl_r\phi\cdot\frac1r\patl_r\Big(\frac1r\patl_r\phi\Big)\Big\}\\
&+\{(A\times B)\cdot x\}\frac2r\patl_r\psi\cdot
\frac1r\patl_r\Big\{\frac1r\patl_r\Big(\frac1r\patl_r\phi\Big)\Big\}\\
&-\{(A\times B)\cdot x\}\frac2r\patl_r\phi\cdot\frac1r\patl_r
\Big\{\frac1r\patl_r\Big(\frac1r\patl_r\psi\Big)\Big\}\\
&+\{(A\times B)\cdot x\}\frac1r\patl_r\Big\{\frac1r\patl_r
\psi\cdot\frac1r\patl_r\Big(\frac1r\patl_r\phi\Big)\Big\}\\
&-\{(A\times B)\cdot x\}\frac1r\patl_r\Big\{\Del\psi\cdot
\frac1r\patl_r\Big(\frac1r\patl_r\phi\Big)\Big\}\\
&+\{(A\times B)\cdot x\}\frac1r\patl_r\Big\{\frac2r\patl_r
\phi\cdot\frac1r\patl_r\Big(\frac1r\patl_r\psi\Big)\Big\}\\
&-\{(A\times B)\cdot x\}\frac1r\patl_r\Big\{\frac1r\patl_r
\phi\cdot\frac1r\patl_r\Del\psi\Big\}\\
&+(B\cdot B)\frac2r\patl_r\psi\cdot
\frac1r\patl_r\Big(\frac1r\patl_r\Del\psi\Big)\\
&+(B\cdot x)^2\frac2r\patl_r\Big\{\frac1r
\patl_r\psi\cdot\frac1r\patl_r\Big(\frac1r\patl_r\Del\psi\Big)\Big\}.\\
\end{split}\eeq

Choose the orthogonal transformation $\rho_b$ which rotation axis is
vector $B$, such that
\beq\label{12-RoB}\begin{split}
B\cdot(\rho_bx)=(\rho_b^tB)\cdot x=B\cdot x,
\end{split}\eeq
where $\rho_b^t$ is the adjoint operator of $\rho_b$. For example
\beq\label{12-RoB-pi}\begin{split}
y=\rho_bx=xM_b\left(\begin{array}{ccc}
0&1&0\\-1&0&0\\0&0&1\\\end{array}\right)M_b^t,
\end{split}\eeq
\beq\label{12-RoB-Mb}\begin{split}
M_b=\left(\begin{array}{ccc}
0&-\frac{b_2^2+b_3^2}{|B|^2}&\frac{b_1}{|B|}\\
\frac{b_3}{|B|}&\frac{b_1b_2}{|B|^2}&\frac{b_2}{|B|}\\
-\frac{b_2}{|B|}&\frac{b_1b_3}{|B|^2}&\frac{b_3}{|B|}\\
\end{array}\right),\;\;M_b^t=\left(\begin{array}{ccc}
0&\frac{b_3}{|B|}&-\frac{b_2}{|B|}\\
-\frac{b_2^2+b_3^2}{|B|^2}&\frac{b_1b_2}{|B|^2}&\frac{b_1b_3}{|B|^2}\\
\frac{b_1}{|B|}&\frac{b_2}{|B|}&\frac{b_3}{|B|}\\
\end{array}\right).
\end{split}\eeq

Applying the orthogonal transformation $y=\rho_bx$ in the equation
\eqref{12-Rph-AVB} \eqref{12-Rps-AVB}, we obtain
\beq\label{12-Rph-AVBy}\begin{split}
&(B\cdot y)\frac1r\patl_r\Big(\frac1r\patl_r\Big)
\{\phi_t-\nu\Del\phi\}\\
=&\{(A\times B)\cdot y\}\frac1r\patl_r
\Big(\frac1r\patl_r\phi\cdot\frac1r\patl_r\phi\Big)\\
&-(B\cdot B)\frac1r\patl_r\phi\cdot
\frac1r\patl_r\Big(\frac1r\patl_r\psi\Big)\\
&-(B\cdot B)\frac1r\patl_r\Big\{\frac1r\patl_r\phi\cdot
\frac1r\patl_r\psi\Big\}\\
&+(B\cdot B)\frac1r\patl_r\Big\{\frac1r\patl_r\phi\cdot
\Del\psi\Big\}\\
&-(B\cdot y)^2\frac1r\patl_r\Big\{\frac2r\patl_r\phi\cdot
\frac1r\patl_r\Big(\frac1r\patl_r\psi\Big)\Big\}\\
&+(B\cdot y)^2\frac1r\patl_r\Big\{\frac2r\patl_r\psi\cdot
\frac1r\patl_r\Big(\frac1r\patl_r\phi\Big)\Big\},\\
\end{split}\eeq
\beq\label{12-Rps-AVBy}\begin{split}
&(B\cdot y)\frac1r\patl_r\Big(\frac1r\patl_r\Big)
\Del\{\psi_t-\nu\Del\psi\}\\
=&-(A\cdot A)\frac4r\patl_r\phi\cdot
\frac1r\patl_r\Big(\frac1r\patl_r\phi\Big)\\
&-\{(A\cdot A)r^2-(A\cdot y)^2\}\frac2r\patl_r\Big\{\frac1r
\patl_r\phi\cdot\frac1r\patl_r\Big(\frac1r\patl_r\phi\Big)\Big\}\\
&+\{(A\times B)\cdot y\}\frac2r\patl_r\psi\cdot
\frac1r\patl_r\Big\{\frac1r\patl_r\Big(\frac1r\patl_r\phi\Big)\Big\}\\
&-\{(A\times B)\cdot y\}\frac2r\patl_r\phi\cdot\frac1r\patl_r
\Big\{\frac1r\patl_r\Big(\frac1r\patl_r\psi\Big)\Big\}\\
&+\{(A\times B)\cdot y\}\frac1r\patl_r\Big\{\frac1r\patl_r
\psi\cdot\frac1r\patl_r\Big(\frac1r\patl_r\phi\Big)\Big\}\\
&-\{(A\times B)\cdot y\}\frac1r\patl_r\Big\{\Del\psi\cdot
\frac1r\patl_r\Big(\frac1r\patl_r\phi\Big)\Big\}\\
&+\{(A\times B)\cdot y\}\frac1r\patl_r\Big\{\frac2r\patl_r
\phi\cdot\frac1r\patl_r\Big(\frac1r\patl_r\psi\Big)\Big\}\\
&-\{(A\times B)\cdot y\}\frac1r\patl_r\Big\{\frac1r\patl_r
\phi\cdot\frac1r\patl_r\Del\psi\Big\}\\
&+(B\cdot B)\frac2r\patl_r\psi\cdot
\frac1r\patl_r\Big(\frac1r\patl_r\Del\psi\Big)\\
&+(B\cdot y)^2\frac2r\patl_r\Big\{\frac1r
\patl_r\psi\cdot\frac1r\patl_r\Big(\frac1r\patl_r\Del\psi\Big)\Big\}.\\
\end{split}\eeq

Putting \eqref{12-Rph-AVB} \eqref{12-RoB} \eqref{12-Rph-AVBy} together,
we get
\beq\label{12-Rph-AVByx}\begin{split}
\{(A\times B)\cdot(x-\rho_bx)\}\frac1r\patl_r
\Big(\frac1r\patl_r\phi\cdot\frac1r\patl_r\phi\Big)=0.\\
\end{split}\eeq
Given $r$, thanks $x\in{\mathbb S}^2_r$ is arbitrary, we have
\beq\label{12ph-AVB1}\begin{split}
\patl_r\Big(\frac1r\patl_r\phi\cdot\frac1r\patl_r\phi\Big)=0.\\
\end{split}\eeq
This equation \eqref{12ph-AVB1} implies that
\beq\label{12ph-AVB2}\begin{split}
&\phi=f_2r^2+f_0,\\
\end{split}\eeq
where $f_0$ and $f_2$ are any functions of $t$.

Provided $f_2=0$, putting \eqref{12ph-AVB2} into \eqref{12-Rph-AVB},
the equation \eqref{12-Rph-AVB} is satisfied. Putting \eqref{12ph-AVB2}
into \eqref{12-Rps-AVB}, we obtain that
\beq\label{12ps-AVB-P0}\begin{split}
&(B\cdot x)\frac1r\patl_r\Big(\frac1r\patl_r\Big)
\Del\{\psi_t-\nu\Del\psi\}\\
=&(B\cdot B)\frac2r\patl_r\psi\cdot
\frac1r\patl_r\Big(\frac1r\patl_r\Del\psi\Big)\\
&+(B\cdot x)^2\frac2r\patl_r\Big\{\frac1r
\patl_r\psi\cdot\frac1r\patl_r\Big(\frac1r\patl_r\Del\psi\Big)\Big\}.\\
\end{split}\eeq

Applying the orthogonal transformation $y=\rho x$ defined by
\eqref{SR12-OT} in the equation \eqref{12ps-AVB-P0}, we have
\beq\label{12ps-AVB-P0y}\begin{split}
&(B\cdot y)\frac1r\patl_r\Big(\frac1r\patl_r\Big)
\Del\{\psi_t-\nu\Del\psi\}\\
=&(B\cdot B)\frac2r\patl_r\psi\cdot
\frac1r\patl_r\Big(\frac1r\patl_r\Del\psi\Big)\\
&+(B\cdot y)^2\frac2r\patl_r\Big\{\frac1r\patl_r\psi\cdot
\frac1r\patl_r\Big(\frac1r\patl_r\Del\psi\Big)\Big\}.\\
\end{split}\eeq
Solving difference of \eqref{12ps-AVB-P0} and \eqref{12ps-AVB-P0y},
we derive
\beq\label{12ps-AVB-P0yx}\begin{split}
&\frac1r\patl_r\Big(\frac1r\patl_r\Big)\Del\{\psi_t-\nu\Del\psi\}\\
=&\{B\cdot(x+\rho x)\}\frac2r\patl_r\Big\{\frac1r\patl_r\psi\cdot
\frac1r\patl_r\Big(\frac1r\patl_r\Del\psi\Big)\Big\}.\\
\end{split}\eeq
Firstly \eqref{12ps-AVB-P0yx} is satisfied provided $x\in\R^3-\{x|\rho
x=x\}$. Finally for $x\in\{x|\rho x=x\}$, selecting $x_n\in\R^3-\{x|
\rho x=x\}$ such that $x_n\rightarrow x$ as $n\rightarrow\infty$, we
can prove that \eqref{12ps-AVB-P0yx} is also satisfied by $n\rightarrow
\infty$.

Let us select another orthogonal transformation $O_r$ as follows
\beq\label{SR12-OrT}
z=O_r x=x\left(\begin{array}{ccc}
0&1&0\\0&0&1\\1&0&0\\
\end{array}\right)=(x_3,x_1,x_2).
\eeq
Applying the orthogonal transformation $z=O_rx$ in the equation
\eqref{12ps-AVB-P0}, by the same arguments as in the proof of
\eqref{12ps-AVB-P0yx}, we have
\beq\label{12ps-AVB-P0zx}\begin{split}
&\frac1r\patl_r\Big(\frac1r\patl_r\Big)\Del\{\psi_t-\nu\Del\psi\}\\
=&\{B\cdot(x+O_rx)\}\frac2r\patl_r\Big\{\frac1r\patl_r\psi\cdot
\frac1r\patl_r\Big(\frac1r\patl_r\Del\psi\Big)\Big\}.\\
\end{split}\eeq

Solving difference of \eqref{12ps-AVB-P0zx} and \eqref{12ps-AVB-P0yx},
we derive
\beq\label{12ps-AVB-P0yz}\begin{split}
\{B\cdot(\rho x-O_rx)\}\patl_r\Big\{\frac1r\patl_r\psi\cdot
\frac1r\patl_r\Big(\frac1r\patl_r\Del\psi\Big)\Big\}=0.\\
\end{split}\eeq
Given $r$, thanks $x\in{\mathbb S}^2_r$ is arbitrary, we have
\beq\label{12ps-AVB-P01}\begin{split}
\patl_r\Big\{\frac1r\patl_r\psi\cdot
\frac1r\patl_r\Big(\frac1r\patl_r\Del\psi\Big)\Big\}=0.\\
\end{split}\eeq
Inserting \eqref{12ps-AVB-P01} into \eqref{12ps-AVB-P0zx}, we obtain
\beq\label{12ps-AVB-P02}\begin{split}
\patl_r\Big(\frac1r\patl_r\Big)\Del\{\psi_t-\nu\Del\psi\}=0.\\
\end{split}\eeq
Putting \eqref{12ps-AVB-P01} \eqref{12ps-AVB-P02} into
\eqref{12ps-AVB-P0}, we get
\beq\label{12ps-AVB-P03}\begin{split}
\patl_r\psi\cdot\patl_r\Big(\frac1r\patl_r\Del\psi\Big)=0.\\
\end{split}\eeq

If $\patl_r\psi=0$, then equations \eqref{12-Rph-AVB} \eqref{12-Rps-AVB}
are satisfied by $(\phi,\psi)$ which is the solution of $(\patl_r\phi,
\patl_r\psi)=0$. In this case, the corresponding velocity $u=0$. It is
trivial.

Now provided $\patl_r\psi\not=0$, then the equation \eqref{12ps-AVB-P03}
implies that $\patl_r\Big(\frac1r\patl_r\Del\psi\Big)=0$ and
\beq\label{12ps-AVB2}\begin{split}
&\Del\psi=20g_4r^2+6g_2,\\
&\psi=g_4r^4+g_2r^2+g_0,\\
\end{split}\eeq
where $g_0$, $g_2$ and $g_4$ are any functions of $t$.

Putting \eqref{12ph-AVB2} \eqref{12ps-AVB2} into \eqref{12-Rph-AVB}
\eqref{12-Rps-AVB}, the equations \eqref{12-Rph-AVB} and
\eqref{12-Rps-AVB} are satisfied provided $f_2=0$.

Provided $f_2\not=0$, putting \eqref{12-Rps-AVB} \eqref{12-RoB}
\eqref{12-Rps-AVBy} \eqref{12ph-AVB2} together, we derive
\beq\label{12-Rps-AVByx}\begin{split}
&\{(A\times B)\cdot(x-\rho_bx)\}\frac1r\patl_r\Big\{\frac1r\patl_r
\phi\cdot\frac1r\patl_r\Del\psi\Big\}=0.\\
\end{split}\eeq
Given $r$, since $x\in{\mathbb S}^2_r$ is arbitrary, we have
\beq\label{12ps-AVB1}\begin{split}
&\patl_r\Big\{\frac1r\patl_r\phi\cdot\frac1r\patl_r\Del\psi\Big\}=0.\\
\end{split}\eeq
We can also derive \eqref{12ps-AVB2} from the equation \eqref{12ps-AVB1}.

Putting \eqref{12ph-AVB2} \eqref{12ps-AVB2} into \eqref{12-Rph-AVB}
\eqref{12-Rps-AVB}, the equations \eqref{12-Rph-AVB} and
\eqref{12-Rps-AVB} are satisfied provided $f_2g_4=0$.

The velocity $u$ corresponding to $(\phi,\psi)$ defined by
\eqref{12ph-AVB2}\eqref{12ps-AVB2} is as follows
\beq\label{12-u-phs5}\begin{split}
u(t,x)=&(A\times\nab)\phi+\{(B\times\nab)\times\nab\}\psi\\
=&2f_2(A\times x)+8g_4(B\cdot x)x-B(16g_4r^2+4g_2).\\
\end{split}\eeq
It is obvious that
\[
\int_{\R^3}|u(t,x)|^2dx=\infty,\;\;\forall t\ge 0.
\]

On the other hand, provided that at least one of \eqref{12ph-AVB2} and
\eqref{12ps-AVB2} is not satisfied, or $f_2g_4\not=0$ although
\eqref{12ph-AVB2} and \eqref{12ps-AVB2} are satisfied, then the
equations \eqref{NS12-Rphi} \eqref{NS12-Rpsi} can not be satisfied by
any radial symmetry functions $\phi$ and $\psi$.

In summary, Theorem \ref{SR12-AVB-Thm} is proved.
\end{proof}

\section{(1,1)-Symplectic Representation and \\Radial Symmetry Breaking
in ${\mathbb{R}}^3$}
\setcounter{equation}{0}

In this section, we assume that the velocity vector $u$ holds the
following (1,1)-symplectic representation
\beq\label{Sym-Rep-11-R3}\begin{split}
u(t,x)=&\{A\times\nab\}\phi(t,x)+\{B\times\nab\}\psi(t,x),\\
\end{split}\eeq
where vectors $A=(a_1,a_2,a_3)\in\R^3-\{0\}$ and $B=(b_1,b_2,b_3)\in
\R^3-\{0\}$ are linearly independent.

Let
\beq\label{Def-SR11-ome}\begin{split}
\ome(t,x)=&\nab\times u(t,x)\\
=&-\{(A\times\nab)\times\nab\}\phi(t,x)
-\{(B\times\nab)\times\nab\}\psi(t,x).\\
\end{split}\eeq
Here vectors $(A\times\nab)\times\nab=(A\cdot\nab)\nab-A\Del$ and
$(B\times\nab)\times\nab=(B\cdot\nab)\nab-B\Del$.

Taking curl with equation \eqref{NS1}, we have
\beq\label{NS-SR11-ome}
\ome_t-\nu\Del\ome+(u\cdot\nab)\ome-(\ome\cdot\nab)u=0.
\eeq

Thanks the following observations
\beq\label{SR11-phi}\begin{split}
(B\times\nab)\cdot \ome(t,x)&=\Big(\{(A\times\nab)\times(B\times\nab)\}
\cdot\nab\Big)\phi(t,x)\\
&=(A\times B)\cdot\nab\Del\phi(t,x),\\
\end{split}\eeq
\beq\label{SR11-psi}\begin{split}
(A\times\nab)\cdot \ome(t,x)&=-\Big(\{(A\times\nab)\times(B\times\nab)\}
\cdot\nab\Big)\psi(t,x)\\
&=-(A\times B)\cdot\nab\Del\psi(t,x),\\
\end{split}\eeq
taking scalar product of equation \eqref{NS-SR11-ome} with
$B\times\nab$, we have
\beq\label{NS-SR11-phi}
(A\times B)\cdot\nab\Del\{\phi_t-\nu\Del\phi\}
+(B\times\nab)\cdot\{(u\cdot\nab)\ome-(\ome\cdot\nab)u\}=0.
\eeq
And taking scalar product of equation \eqref{NS-SR11-ome} with
$A\times\nab$, we derive
\beq\label{NS-SR11-psi}
(A\times B)\cdot\nab\Del\{\psi_t-\nu\Del\psi\}
-(A\times\nab)\cdot\{(u\cdot\nab)\ome-(\ome\cdot\nab)u\}=0.
\eeq

Now we assume that $\phi$ and $\psi$ are radial symmetric functions
with respect to space variable $x\in\R^3$. It is that $\phi(t,x)=
\phi(t,r)$, $\psi(t,x)=\psi(t,r)$  and $r^2=x_1^2+x_2^2+x_3^2$.
Then we have
\beq\label{SR11-uR3R}\begin{split}
u(t,x)=&(A\times x)\frac1r\patl_r\phi
+(B\times x)\frac1r\patl_r\psi,\\
\end{split}\eeq
\beq\label{SR11-oR3R}\begin{split}
\ome(t,x)
=&A\Big(\Del\phi-\frac1r\patl_r\phi\Big)-x(A\cdot x)\frac1r\patl_r
\Big(\frac1r\patl_r\phi\Big)\\
&+B\Big(\Del\psi-\frac1r\patl_r\psi\Big)-x(B\cdot x)\frac1r\patl_r
\Big(\frac1r\patl_r\psi\Big),\\
\end{split}\eeq
\beq\label{NT-SR11-uo}\begin{split}
(u\cdot\nab)\ome=&\Big(\frac1r\patl_r\phi(A\times x)\cdot\nab
+\frac1r\patl_r\psi(B\times x)\cdot\nab\Big)\\
&\Big\{A\Big(\Del\phi-\frac1r\patl_r\phi\Big)-x(A\cdot x)
\frac1r\patl_r\Big(\frac1r\patl_r\phi\Big)\\
&+B\Big(\Del\psi-\frac1r\patl_r\psi\Big)-x(B\cdot x)
\frac1r\patl_r\Big(\frac1r\patl_r\psi\Big)\Big\}\\
=&-(A\times x)(A\cdot x)\frac1r\patl_r\phi\cdot
\frac1r\patl_r\Big(\frac1r\patl_r\phi\Big)\\
&-(A\times x)(B\cdot x)\frac1r\patl_r\phi\cdot
\frac1r\patl_r\Big(\frac1r\patl_r\psi\Big)\\
&+x\{(A\times B)\cdot x\}\frac1r\patl_r\phi\cdot
\frac1r\patl_r\Big(\frac1r\patl_r\psi\Big)\\
&-x\{(A\times B)\cdot x\}\frac1r\patl_r\psi\cdot
\frac1r\patl_r\Big(\frac1r\patl_r\phi\Big)\\
&-(B\times x)(A\cdot x)\frac1r\patl_r\psi\cdot
\frac1r\patl_r\Big(\frac1r\patl_r\phi\Big)\\
&-(B\times x)(B\cdot x)\frac1r\patl_r\psi\cdot
\frac1r\patl_r\Big(\frac1r\patl_r\psi\Big),\\
\end{split}\eeq
\beq\label{NT-SR11-ou}\begin{split}
(\ome\cdot\nab)u
=&\Big\{\Big(\Del\phi-\frac1r\patl_r\phi\Big)A\cdot\nab-(A\cdot x)\frac1r\patl_r
\Big(\frac1r\patl_r\phi\Big)x\cdot\nab\\
&+\Big(\Del\psi-\frac1r\patl_r\psi\Big)B\cdot\nab-(B\cdot x)\frac1r\patl_r
\Big(\frac1r\patl_r\psi\Big)x\cdot\nab\Big\}\\
&\Big\{(A\times x)\frac1r\patl_r\phi
+(B\times x)\frac1r\patl_r\psi\Big\}\\
=&(A\times x)(A\cdot x)\frac1r\patl_r\Big(\frac1r\patl_r\phi\Big)\cdot
\Big(\Del\phi-\frac1r\patl_r\phi\Big)\\
&+(B\times x)(A\cdot x)\frac1r\patl_r\Big(\frac1r\patl_r\psi\Big)\cdot
\Big(\Del\phi-\frac1r\patl_r\phi\Big)\\
&-(A\times B)\frac1r\patl_r\psi\cdot
\Big(\Del\phi-\frac1r\patl_r\phi\Big)\\
&-(A\times x)(A\cdot x)\patl_r\Big(\frac1r\patl_r\phi\Big)
\cdot\patl_r\Big(\frac1r\patl_r\phi\Big)\\
&-(A\times x)(A\cdot x)\frac1r\patl_r\phi
\cdot\frac1r\patl_r\Big(\frac1r\patl_r\phi\Big)\\
&-(B\times x)(A\cdot x)\patl_r\Big(\frac1r\patl_r\psi\Big)
\cdot\patl_r\Big(\frac1r\patl_r\phi\Big)\\
&-(B\times x)(A\cdot x)\frac1r\patl_r\psi
\cdot\frac1r\patl_r\Big(\frac1r\patl_r\phi\Big)\\
&+(A\times x)(B\cdot x)\frac1r\patl_r\Big(\frac1r\patl_r\phi\Big)\cdot
\Big(\Del\psi-\frac1r\patl_r\psi\Big)\\
&+(A\times B)\frac1r\patl_r\phi\cdot
\Big(\Del\psi-\frac1r\patl_r\psi\Big)\\
&+(B\times x)(B\cdot x)\frac1r\patl_r\Big(\frac1r\patl_r\psi\Big)\cdot
\Big(\Del\psi-\frac1r\patl_r\psi\Big)\\
&-(A\times x)(B\cdot x)\patl_r\Big(\frac1r\patl_r\phi\Big)
\cdot\patl_r\Big(\frac1r\patl_r\psi\Big)\\
&-(A\times x)(B\cdot x)\frac1r\patl_r\phi
\cdot\frac1r\patl_r\Big(\frac1r\patl_r\psi\Big)\\
&-(B\times x)(B\cdot x)\patl_r\Big(\frac1r\patl_r\psi\Big)
\cdot\patl_r\Big(\frac1r\patl_r\psi\Big)\\
&-(B\times x)(B\cdot x)\frac1r\patl_r\psi
\cdot\frac1r\patl_r\Big(\frac1r\patl_r\psi\Big)\\
\end{split}\eeq
\[\begin{split}
=&(A\times x)(A\cdot x)\frac1r\patl_r\phi\cdot
\frac1r\patl_r\Big(\frac1r\patl_r\phi\Big)\\
&-(A\times x)(B\cdot x)\frac1r\patl_r\phi
\cdot\frac1r\patl_r\Big(\frac1r\patl_r\psi\Big)\\
&+(A\times x)(B\cdot x)\frac2r\patl_r\psi\cdot
\frac1r\patl_r\Big(\frac1r\patl_r\phi\Big)\\
&+(A\times B)\patl_r\phi\cdot\patl_r\Big(\frac1r\patl_r\psi\Big)\\
&-(A\times B)\patl_r\psi\cdot\patl_r\Big(\frac1r\patl_r\phi\Big)\\
&+(B\times x)(A\cdot x)\frac2r\patl_r\phi\cdot
\frac1r\patl_r\Big(\frac1r\patl_r\psi\Big)\\
&-(B\times x)(A\cdot x)\frac1r\patl_r\psi\cdot
\frac1r\patl_r\Big(\frac1r\patl_r\phi\Big)\\
&+(B\times x)(B\cdot x)\frac1r\patl_r\psi\cdot
\frac1r\patl_r\Big(\frac1r\patl_r\psi\Big),\\
\end{split}\]

\beq\label{NT-SR11-uoou}\begin{split}
&(u\cdot\nab)\ome-(\ome\cdot\nab)u\\
=&-2(A\times x)(A\cdot x)\frac1r\patl_r\phi\cdot
\frac1r\patl_r\Big(\frac1r\patl_r\phi\Big)\\
&-2(A\times x)(B\cdot x)\frac1r\patl_r\psi\cdot
\frac1r\patl_r\Big(\frac1r\patl_r\phi\Big)\\
&+x\{(A\times B)\cdot x\}\frac1r\patl_r\phi\cdot
\frac1r\patl_r\Big(\frac1r\patl_r\psi\Big)\\
&-(A\times B)\patl_r\phi\cdot\patl_r\Big(\frac1r\patl_r\psi\Big)\\
&+(A\times B)\patl_r\psi\cdot\patl_r\Big(\frac1r\patl_r\phi\Big)\\
&-x\{(A\times B)\cdot x\}\frac1r\patl_r\psi\cdot
\frac1r\patl_r\Big(\frac1r\patl_r\phi\Big)\\
&-2(B\times x)(B\cdot x)\frac1r\patl_r\psi\cdot
\frac1r\patl_r\Big(\frac1r\patl_r\psi\Big)\\
&-2(B\times x)(A\cdot x)\frac1r\patl_r\phi\cdot
\frac1r\patl_r\Big(\frac1r\patl_r\psi\Big),\\
\end{split}\eeq

\beq\label{NT-SR11-bt}\begin{split}
&(B\times\nab)\cdot\{(u\cdot\nab)\ome-(\ome\cdot\nab)u\}\\
=&(B\times\nab)\cdot\Big\{-2(A\times x)(A\cdot x)\frac1r\patl_r\phi\cdot
\frac1r\patl_r\Big(\frac1r\patl_r\phi\Big)\\
&-2(A\times x)(B\cdot x)\frac1r\patl_r\psi\cdot
\frac1r\patl_r\Big(\frac1r\patl_r\phi\Big)\\
&+x\{(A\times B)\cdot x)\}\frac1r\patl_r\phi\cdot
\frac1r\patl_r\Big(\frac1r\patl_r\psi\Big)\\
&-(A\times B)\patl_r\phi\cdot\patl_r\Big(\frac1r\patl_r\psi\Big)\\
&+(A\times B)\patl_r\psi\cdot\patl_r\Big(\frac1r\patl_r\phi\Big)\\
&-x\{(A\times B)\cdot x\}\frac1r\patl_r\psi\cdot
\frac1r\patl_r\Big(\frac1r\patl_r\phi\Big)\\
&-2(B\times x)(B\cdot x)\frac1r\patl_r\psi\cdot
\frac1r\patl_r\Big(\frac1r\patl_r\psi\Big)\\
&-2(B\times x)(A\cdot x)\frac1r\patl_r\phi\cdot
\frac1r\patl_r\Big(\frac1r\patl_r\psi\Big)\Big\}\\
=&-4(A\cdot B)(A\cdot x)\frac1r\patl_r\phi\cdot
\frac1r\patl_r\Big(\frac1r\patl_r\phi\Big)\\
&+2\{(A\times B)\times A\}\cdot x\frac1r\patl_r\phi\cdot
\frac1r\patl_r\Big(\frac1r\patl_r\phi\Big)\\
&-2\{(A\cdot B)r^2-(A\cdot x)(B\cdot x)\}(A\cdot x)\frac1r\patl_r\Big\{\frac1r
\patl_r\phi\cdot\frac1r\patl_r\Big(\frac1r\patl_r\phi\Big)\Big\}\\
&-4(A\cdot B)(B\cdot x)\frac1r\patl_r\psi\cdot
\frac1r\patl_r\Big(\frac1r\patl_r\phi\Big)\\
&-2\{(A\cdot B)r^2-(A\cdot x)(B\cdot x)\}(B\cdot x)\frac1r\patl_r\Big\{\frac1r
\patl_r\psi\cdot\frac1r\patl_r\Big(\frac1r\patl_r\phi\Big)\Big\}\\
&+\{(A\times B)\times B\}\cdot x\frac1r\patl_r\phi\cdot
\frac1r\patl_r\Big(\frac1r\patl_r\psi\Big)\\
&-\{(A\times B)\times B\}\cdot x\frac1r\patl_r\Big\{\patl_r\phi\cdot
\patl_r\Big(\frac1r\patl_r\psi\Big)\Big\}\\
&+\{(A\times B)\times B\}\cdot x\frac1r\patl_r\Big\{\patl_r\psi\cdot
\patl_r\Big(\frac1r\patl_r\phi\Big)\Big\}\\
&+\{(A\times B)\times B\}\cdot x\frac1r\patl_r\psi\cdot
\frac1r\patl_r\Big(\frac1r\patl_r\phi\Big)\\
&-4|B|^2(B\cdot x)\frac1r\patl_r\psi\cdot
\frac1r\patl_r\Big(\frac1r\patl_r\psi\Big)\\
&-2\{|B|^2r^2-(B\cdot x)^2\}(B\cdot x)\frac1r\patl_r\Big\{\frac1r
\patl_r\psi\cdot\frac1r\patl_r\Big(\frac1r\patl_r\psi\Big\}\\
&-4|B|^2(A\cdot x)\frac1r\patl_r\phi\cdot
\frac1r\patl_r\Big(\frac1r\patl_r\psi\Big)\\
&-2\{|B|^2r^2-(B\cdot x)^2\}(A\cdot x)\frac1r\patl_r\Big\{\frac1r
\patl_r\phi\cdot\frac1r\patl_r\Big(\frac1r\patl_r\psi\Big)\Big\},\\
\end{split}\eeq

\beq\label{NT-SR11-at}\begin{split}
&(A\times\nab)\cdot\{(u\cdot\nab)\ome-(\ome\cdot\nab)u\}\\
=&(A\times\nab)\cdot\Big\{-2(A\times x)(A\cdot x)\frac1r\patl_r\phi\cdot
\frac1r\patl_r\Big(\frac1r\patl_r\phi\Big)\\
&-2(A\times x)(B\cdot x)\frac1r\patl_r\psi\cdot
\frac1r\patl_r\Big(\frac1r\patl_r\phi\Big)\\
&+x\{(A\times B)\cdot x\}\frac1r\patl_r\phi\cdot
\frac1r\patl_r\Big(\frac1r\patl_r\psi\Big)\\
&-(A\times B)\patl_r\phi\cdot\patl_r\Big(\frac1r\patl_r\psi\Big)\\
&+(A\times B)\patl_r\psi\cdot\patl_r\Big(\frac1r\patl_r\phi\Big)\\
&-x\{(A\times B)\cdot x\}\frac1r\patl_r\psi\cdot
\frac1r\patl_r\Big(\frac1r\patl_r\phi\Big)\\
&-2(B\times x)(B\cdot x)\frac1r\patl_r\psi\cdot
\frac1r\patl_r\Big(\frac1r\patl_r\psi\Big)\\
&-2(B\times x)(A\cdot x)\frac1r\patl_r\phi\cdot
\frac1r\patl_r\Big(\frac1r\patl_r\psi\Big)\Big\}\\
=&-4|A|^2(A\cdot x)\frac1r\patl_r\phi\cdot
\frac1r\patl_r\Big(\frac1r\patl_r\phi\Big)\\
&-2\{|A|^2r^2-(A\cdot x)^2\}(A\cdot x)\frac1r\patl_r\Big\{\frac1r
\patl_r\phi\cdot\frac1r\patl_r\Big(\frac1r\patl_r\phi\Big)\Big\}\\
&-4|A|^2(B\cdot x)\frac1r\patl_r\psi\cdot
\frac1r\patl_r\Big(\frac1r\patl_r\phi\Big)\\
&-2\{|A|^2r^2-(A\cdot x)^2\}(B\cdot x)\frac1r\patl_r\Big\{\frac1r
\patl_r\psi\cdot\frac1r\patl_r\Big(\frac1r\patl_r\phi\Big)\Big\}\\
&-\{(A\times B)\times A\}\cdot x\frac1r\patl_r\phi\cdot
\frac1r\patl_r\Big(\frac1r\patl_r\psi\Big)\\
&-\{(A\times B)\times A\}\cdot x\frac1r\patl_r\Big\{\patl_r\phi\cdot
\patl_r\Big(\frac1r\patl_r\psi\Big)\Big\}\\
&+\{(A\times B)\times A\}\cdot x\frac1r\patl_r\Big\{\patl_r\psi\cdot
\patl_r\Big(\frac1r\patl_r\phi\Big)\Big\}\\
&-\{(A\times B)\times A\}\cdot x\frac1r\patl_r\psi\cdot
\frac1r\patl_r\Big(\frac1r\patl_r\phi\Big)\\
&-4(A\cdot B)(B\cdot x)\frac1r\patl_r\psi\cdot
\frac1r\patl_r\Big(\frac1r\patl_r\psi\Big)\\
&-2\{(A\times B)\times B\}\cdot x\frac1r\patl_r\psi\cdot
\frac1r\patl_r\Big(\frac1r\patl_r\psi\Big)\\
&-2\{(A\cdot B)r^2-(A\cdot x)(B\cdot x)\}(B\cdot x)\frac1r\patl_r\Big\{\frac1r
\patl_r\psi\cdot\frac1r\patl_r\Big(\frac1r\patl_r\psi\Big)\Big\}\\
&-4(A\cdot B)(A\cdot x)\frac1r\patl_r\phi\cdot
\frac1r\patl_r\Big(\frac1r\patl_r\psi\Big)\\
&-2\{(A\cdot B)r^2-(A\cdot x)(B\cdot x)\}(A\cdot x)\frac1r\patl_r\Big\{\frac1r
\patl_r\phi\cdot\frac1r\patl_r\Big(\frac1r\patl_r\psi\Big)\Big\},\\
\end{split}\eeq
where
\[
\Del f=\frac3r\patl_rf+r\patl_r(\frac1r\patl_rf).
\]

Employing \eqref{NS-SR11-phi} and \eqref{NT-SR11-bt}, we obtain that
\[\label{NS-SR11-phi1}\begin{split}
&\{(A\times B)\cdot x\}\frac1r\patl_r\Del\{\phi_t-\nu\Del\phi\}\\
&-4(A\cdot B)(A\cdot x)\frac1r\patl_r\phi\cdot
\frac1r\patl_r\Big(\frac1r\patl_r\phi\Big)\\
&+2\{(A\cdot A)(B\cdot x)-(A\cdot B)(A\cdot x)\}\frac1r\patl_r\phi\cdot
\frac1r\patl_r\Big(\frac1r\patl_r\phi\Big)\\
&-2\{(A\cdot B)r^2-(A\cdot x)(B\cdot x)\}(A\cdot x)\frac1r\patl_r\Big\{\frac1r
\patl_r\phi\cdot\frac1r\patl_r\Big(\frac1r\patl_r\phi\Big)\Big\}\\
&-4(A\cdot B)(B\cdot x)\frac1r\patl_r\psi\cdot
\frac1r\patl_r\Big(\frac1r\patl_r\phi\Big)\\
&-2\{(A\cdot B)r^2-(A\cdot x)(B\cdot x)\}(B\cdot x)\frac1r\patl_r\Big\{\frac1r
\patl_r\psi\cdot\frac1r\patl_r\Big(\frac1r\patl_r\phi\Big)\Big\}\\
&+\{(A\cdot B)(B\cdot x)-(B\cdot B)(A\cdot x)\}\frac1r\patl_r\phi\cdot
\frac1r\patl_r\Big(\frac1r\patl_r\psi\Big)\\
&-\{(A\cdot B)(B\cdot x)-(B\cdot B)(A\cdot x)\}\frac1r\patl_r\Big\{\patl_r\phi\cdot
\patl_r\Big(\frac1r\patl_r\psi\Big)\Big\}\\
&+\{(A\cdot B)(B\cdot x)-(B\cdot B)(A\cdot x)\}\frac1r\patl_r\Big\{\patl_r\psi\cdot
\patl_r\Big(\frac1r\patl_r\phi\Big)\Big\}\\
&+\{(A\cdot B)(B\cdot x)-(B\cdot B)(A\cdot x)\}\frac1r\patl_r\psi\cdot
\frac1r\patl_r\Big(\frac1r\patl_r\phi\Big)\\
&-4|B|^2(B\cdot x)\frac1r\patl_r\psi\cdot
\frac1r\patl_r\Big(\frac1r\patl_r\psi\Big)\\
&-2\{|B|^2r^2-(B\cdot x)^2\}(B\cdot x)\frac1r\patl_r\Big\{\frac1r
\patl_r\psi\cdot\frac1r\patl_r\Big(\frac1r\patl_r\psi\Big)\Big\}\\
&-4|B|^2(A\cdot x)\frac1r\patl_r\phi\cdot
\frac1r\patl_r\Big(\frac1r\patl_r\psi\Big)\\
&-2\{|B|^2r^2-(B\cdot x)^2\}(A\cdot x)\frac1r\patl_r\Big\{\frac1r
\patl_r\phi\cdot\frac1r\patl_r\Big(\frac1r\patl_r\psi\Big)\Big\}\\
\end{split}\]
\beq\label{NS-SR11-phi1}\begin{split}
=&\{(A\times B)\cdot x\}\frac1r\patl_r\Del\{\phi_t-\nu\Del\phi\}\\
&-6(A\cdot B)(A\cdot x)\frac1r\patl_r\phi\cdot
\frac1r\patl_r\Big(\frac1r\patl_r\phi\Big)\\
&-2(A\cdot B)(A\cdot x)r\patl_r\Big\{\frac1r
\patl_r\phi\cdot\frac1r\patl_r\Big(\frac1r\patl_r\phi\Big)\Big\}\\
&-2(B\cdot B)(A\cdot x)\frac1r\patl_r\phi\cdot
\frac1r\patl_r\Big(\frac1r\patl_r\psi\Big)\\
&-(B\cdot B)(A\cdot x)\patl_r\Big\{\patl_r\phi\cdot
\frac1r\patl_r\Big(\frac1r\patl_r\psi\Big)\Big\}\\
&-2(B\cdot B)(A\cdot x)\frac1r\patl_r\psi\cdot
\frac1r\patl_r\Big(\frac1r\patl_r\phi\Big)\\
&-(B\cdot B)(A\cdot x)\patl_r\Big\{\patl_r\psi\cdot
\frac1r\patl_r\Big(\frac1r\patl_r\phi\Big)\Big\}\\
&+2(A\cdot A)(B\cdot x)\frac1r\patl_r\phi\cdot
\frac1r\patl_r\Big(\frac1r\patl_r\phi\Big)\\
&-(A\cdot B)(B\cdot x)\patl_r\Big\{\patl_r\psi\cdot
\frac1r\patl_r\Big(\frac1r\patl_r\phi\Big)\Big\}\\
&-(A\cdot B)(B\cdot x)\patl_r\Big\{\patl_r\phi\cdot
\frac1r\patl_r\Big(\frac1r\patl_r\psi\Big)\Big\}\\
&-4(B\cdot B)(B\cdot x)\frac1r\patl_r\psi\cdot
\frac1r\patl_r\Big(\frac1r\patl_r\psi\Big)\\
&-2\{(B\cdot B)r^2-(B\cdot x)^2\}(B\cdot x)\frac1r\patl_r\Big\{\frac1r
\patl_r\psi\cdot\frac1r\patl_r\Big(\frac1r\patl_r\psi\Big)\Big\}\\
&+2(A\cdot x)^2(B\cdot x)\frac1r\patl_r\Big\{\frac1r
\patl_r\phi\cdot\frac1r\patl_r\Big(\frac1r\patl_r\phi\Big)\Big\}\\
&+2(A\cdot x)(B\cdot x)^2\frac1r\patl_r\Big\{\frac1r
\patl_r\psi\cdot\frac1r\patl_r\Big(\frac1r\patl_r\phi\Big)\Big\}\\
&+2(A\cdot x)(B\cdot x)^2\frac1r\patl_r\Big\{\frac1r
\patl_r\phi\cdot\frac1r\patl_r\Big(\frac1r\patl_r\psi\Big)\Big\}\\
=&0.\\
\end{split}\eeq

Applying \eqref{NS-SR11-psi} and \eqref{NT-SR11-at}, we get that
\[\label{NS-SR11-psi1}\begin{split}
&\{(A\times B)\cdot x\}\frac1r\patl_r\Del\{\psi_t-\nu\Del\psi\}\\
&+4|A|^2(A\cdot x)\frac1r\patl_r\phi\cdot
\frac1r\patl_r\Big(\frac1r\patl_r\phi\Big)\\
&+2\{|A|^2r^2-(A\cdot x)^2\}(A\cdot x)\frac1r\patl_r\Big\{\frac1r
\patl_r\phi\cdot\frac1r\patl_r\Big(\frac1r\patl_r\phi\Big)\Big\}\\
&+4|A|^2(B\cdot x)\frac1r\patl_r\psi\cdot
\frac1r\patl_r\Big(\frac1r\patl_r\phi\Big)\\
&+2\{|A|^2r^2-(A\cdot x)^2\}(B\cdot x)\frac1r\patl_r\Big\{\frac1r
\patl_r\psi\cdot\frac1r\patl_r\Big(\frac1r\patl_r\phi\Big)\Big\}\\
&+\{(A\cdot A)(B\cdot x)-(A\cdot B)(A\cdot x)\}\frac1r\patl_r\phi\cdot
\frac1r\patl_r\Big(\frac1r\patl_r\psi\Big)\\
&+\{(A\cdot A)(B\cdot x)-(A\cdot B)(A\cdot x)\}\frac1r\patl_r\Big\{\patl_r\phi\cdot
\patl_r\Big(\frac1r\patl_r\psi\Big)\Big\}\\
&-\{(A\cdot A)(B\cdot x)-(A\cdot B)(A\cdot x)\}\frac1r\patl_r\Big\{\patl_r\psi\cdot
\patl_r\Big(\frac1r\patl_r\phi\Big)\Big\}\\
&+\{(A\cdot A)(B\cdot x)-(A\cdot B)(A\cdot x)\}\frac1r\patl_r\psi\cdot
\frac1r\patl_r\Big(\frac1r\patl_r\phi\Big)\\
&+4(A\cdot B)(B\cdot x)\frac1r\patl_r\psi\cdot
\frac1r\patl_r\Big(\frac1r\patl_r\psi\Big)\\
&+2\{(A\cdot B)(B\cdot x)-(B\cdot B)(A\cdot x)\}\frac1r\patl_r\psi\cdot
\frac1r\patl_r\Big(\frac1r\patl_r\psi\Big)\\
&+2\{(A\cdot B)r^2-(A\cdot x)(B\cdot x)\}(B\cdot x)\frac1r\patl_r\Big\{\frac1r
\patl_r\psi\cdot\frac1r\patl_r\Big(\frac1r\patl_r\psi\Big)\Big\}\\
&+4(A\cdot B)(A\cdot x)\frac1r\patl_r\phi\cdot
\frac1r\patl_r\Big(\frac1r\patl_r\psi\Big)\\
&+2\{(A\cdot B)r^2-(A\cdot x)(B\cdot x)\}(A\cdot x)\frac1r\patl_r\Big\{\frac1r
\patl_r\phi\cdot\frac1r\patl_r\Big(\frac1r\patl_r\psi\Big)\Big\}\\
\end{split}\]
\beq\label{NS-SR11-psi1}\begin{split}
=&\{(A\times B)\cdot x\}\frac1r\patl_r\Del\{\psi_t-\nu\Del\psi\}\\
&+4(A\cdot A)(A\cdot x)\frac1r\patl_r\phi\cdot
\frac1r\patl_r\Big(\frac1r\patl_r\phi\Big)\\
&+2\{(A\cdot A)r^2-(A\cdot x)^2\}(A\cdot x)\frac1r\patl_r\Big\{\frac1r
\patl_r\phi\cdot\frac1r\patl_r\Big(\frac1r\patl_r\phi\Big)\Big\}\\
&+(A\cdot B)(A\cdot x)\patl_r\Big\{\patl_r\phi\cdot
\frac1r\patl_r\Big(\frac1r\patl_r\psi\Big)\Big\}\\
&+(A\cdot B)(A\cdot x)\patl_r\Big\{\patl_r\psi\cdot
\frac1r\patl_r\Big(\frac1r\patl_r\phi\Big)\Big\}\\
&-2(B\cdot B)(A\cdot x)\frac1r\patl_r\psi\cdot
\frac1r\patl_r\Big(\frac1r\patl_r\psi\Big)\\
&+2(A\cdot A)(B\cdot x)\frac1r\patl_r\psi\cdot
\frac1r\patl_r\Big(\frac1r\patl_r\phi\Big)\\
&+(A\cdot A)(B\cdot x)\patl_r\Big\{\patl_r\psi\cdot
\frac1r\patl_r\Big(\frac1r\patl_r\phi\Big)\Big\}\\
&+2(A\cdot A)(B\cdot x)\frac1r\patl_r\phi\cdot
\frac1r\patl_r\Big(\frac1r\patl_r\psi\Big)\\
&+(A\cdot A)(B\cdot x)\patl_r\Big\{\patl_r\phi\cdot
\frac1r\patl_r\Big(\frac1r\patl_r\psi\Big)\Big\}\\
&+6(A\cdot B)(B\cdot x)\frac1r\patl_r\psi\cdot
\frac1r\patl_r\Big(\frac1r\patl_r\psi\Big)\\
&+2(A\cdot B)(B\cdot x)r\patl_r\Big\{\frac1r\patl_r\psi\cdot
\frac1r\patl_r\Big(\frac1r\patl_r\psi\Big)\Big\}\\
&-2(A\cdot x)^2(B\cdot x)\frac1r\patl_r\Big\{\frac1r
\patl_r\psi\cdot\frac1r\patl_r\Big(\frac1r\patl_r\phi\Big)\Big\}\\
&-2(A\cdot x)^2(B\cdot x)\frac1r\patl_r\Big\{\frac1r
\patl_r\phi\cdot\frac1r\patl_r\Big(\frac1r\patl_r\psi\Big)\Big\}\\
&-2(A\cdot x)(B\cdot x)^2\frac1r\patl_r\Big\{\frac1r
\patl_r\psi\cdot\frac1r\patl_r\Big(\frac1r\patl_r\psi\Big)\Big\}\\
=&0.\\
\end{split}\eeq

\begin{proof}[Proof of Theorem \ref{SR11-Thm}]
Since vectors $A$ and $B$ are linearly independent, they can span a
plane $\Gamma_{AB}=\mbox{span}\{A,B\}$. We select the reflection
transformation $\rho_{ab}$, that $\Gamma_{AB}$ is invariant plane,
such that
\beq\label{11-RefT-p}\begin{split}
&r^2=x\cdot x=\rho_{ab}x\cdot\rho_{ab}x,\\
&A\cdot\rho_{ab}x=A\cdot x,\\
&B\cdot\rho_{ab}x=B\cdot x.\\
\end{split}\eeq
Here $|A|^2=A\cdot A$, $B^{\prime}=B-\frac{A\cdot B}{A\cdot A}A=
(b^{\prime}_1,b^{\prime}_2,b_3^{\prime})$,
\beq\label{11-OT-M}\begin{split}
&M_{ab}=\left(\begin{array}{ccc}
\frac{a_1}{|A|}&\frac{b^{\prime}_1}{|B^{\prime}|}&\frac{m_1}{|A||B|}\\
\frac{a_2}{|A|}&\frac{b^{\prime}_2}{|B^{\prime}|}&\frac{m_2}{|A||B|}\\
\frac{a_3}{|A|}&\frac{b^{\prime}_3}{|B^{\prime}|}&\frac{m_3}{|A||B|}
\end{array}\right),\\
&m_1=a_2b_3-a_3b_2,\;\;m_2=a_3b_1-a_1b_3,
\;\;m_3=a_1b_2-a_2b_1,\\
\end{split}\eeq
\beq\label{11-Rho-AB}\begin{split}
&y=\rho_{ab}x=xM_{ab}\left(\begin{array}{ccc}
1&0&0\\0&1&0\\0&0&-1\end{array}\right)M^t_{ab},\\
\end{split}\eeq
where $M^t_{ab}$ is the adjoint matrix of $M_{ab}$.

Applying the orthogonal transformation $y=\rho_{ab}x$ in the equation
\eqref{NS-SR11-phi1}, we obtain
\beq\label{11-phi1y}\begin{split}
&\{(A\times B)\cdot y\}\frac1r\patl_r\Del\{\phi_t-\nu\Del\phi\}\\
&-6(A\cdot B)(A\cdot x)\frac1r\patl_r\phi\cdot
\frac1r\patl_r\Big(\frac1r\patl_r\phi\Big)\\
&-2(A\cdot B)(A\cdot x)r\patl_r\Big\{\frac1r
\patl_r\phi\cdot\frac1r\patl_r\Big(\frac1r\patl_r\phi\Big)\Big\}\\
&-2(B\cdot B)(A\cdot x)\frac1r\patl_r\phi\cdot
\frac1r\patl_r\Big(\frac1r\patl_r\psi\Big)\\
&-(B\cdot B)(A\cdot x)\patl_r\Big\{\patl_r\phi\cdot
\frac1r\patl_r\Big(\frac1r\patl_r\psi\Big)\Big\}\\
&-2(B\cdot B)(A\cdot x)\frac1r\patl_r\psi\cdot
\frac1r\patl_r\Big(\frac1r\patl_r\phi\Big)\\
&-(B\cdot B)(A\cdot x)\patl_r\Big\{\patl_r\psi\cdot
\frac1r\patl_r\Big(\frac1r\patl_r\phi\Big)\Big\}\\
&+2(A\cdot A)(B\cdot x)\frac1r\patl_r\phi\cdot
\frac1r\patl_r\Big(\frac1r\patl_r\phi\Big)\\
&-(A\cdot B)(B\cdot x)\patl_r\Big\{\patl_r\psi\cdot
\frac1r\patl_r\Big(\frac1r\patl_r\phi\Big)\Big\}\\
&-(A\cdot B)(B\cdot x)\patl_r\Big\{\patl_r\phi\cdot
\frac1r\patl_r\Big(\frac1r\patl_r\psi\Big)\Big\}\\
&-4(B\cdot B)(B\cdot x)\frac1r\patl_r\psi\cdot
\frac1r\patl_r\Big(\frac1r\patl_r\psi\Big)\\
&-2\{(B\cdot B)r^2-(B\cdot x)^2\}(B\cdot x)\frac1r\patl_r\Big\{\frac1r
\patl_r\psi\cdot\frac1r\patl_r\Big(\frac1r\patl_r\psi\Big)\Big\}\\
&+2(A\cdot x)^2(B\cdot x)\frac1r\patl_r\Big\{\frac1r
\patl_r\phi\cdot\frac1r\patl_r\Big(\frac1r\patl_r\phi\Big)\Big\}\\
&+2(A\cdot x)(B\cdot x)^2\frac1r\patl_r\Big\{\frac1r
\patl_r\psi\cdot\frac1r\patl_r\Big(\frac1r\patl_r\phi\Big)\Big\}\\
&+2(A\cdot x)(B\cdot x)^2\frac1r\patl_r\Big\{\frac1r
\patl_r\phi\cdot\frac1r\patl_r\Big(\frac1r\patl_r\psi\Big)\Big\}\\
=&0,\\
\end{split}\eeq
where we have used \eqref{11-RefT-p}.

The equations \eqref{NS-SR11-phi1} \eqref{11-phi1y} imply
\beq\label{11-phi1yx}\begin{split}
\{(A\times B)\cdot(x-\rho_{ab}x)\}\frac1r\patl_r\Del\{\phi_t
-\nu\Del\phi\}=&0.
\end{split}\eeq
Given $r$, thanks $x\in{\mathbb S}^2_r$ is arbitrary, we get
\beq\label{11ph-RAB}\begin{split}
\patl_r\Del\{\phi_t-\nu\Del\phi\}=0.\\
\end{split}\eeq

Similarly using the same arguments as in the proof of \eqref{11ph-RAB},
for example employing the orthogonal transformation $y=\rho_{ab}x$ in
the equation \eqref{NS-SR11-psi1} etc., we derive
\beq\label{11ps-RAB}\begin{split}
\patl_r\Del\{\psi_t-\nu\Del\psi\}=0.\\
\end{split}\eeq

Putting \eqref{11ph-RAB} into \eqref{NS-SR11-phi1}, we have
\beq\label{11ph-RAB1}\begin{split}
&-6(A\cdot B)(A\cdot x)\frac1r\patl_r\phi\cdot
\frac1r\patl_r\Big(\frac1r\patl_r\phi\Big)\\
&-2(A\cdot B)(A\cdot x)r\patl_r\Big\{\frac1r
\patl_r\phi\cdot\frac1r\patl_r\Big(\frac1r\patl_r\phi\Big)\Big\}\\
&-2(B\cdot B)(A\cdot x)\frac1r\patl_r\phi\cdot
\frac1r\patl_r\Big(\frac1r\patl_r\psi\Big)\\
&-(B\cdot B)(A\cdot x)\patl_r\Big\{\patl_r\phi\cdot
\frac1r\patl_r\Big(\frac1r\patl_r\psi\Big)\Big\}\\
&-2(B\cdot B)(A\cdot x)\frac1r\patl_r\psi\cdot
\frac1r\patl_r\Big(\frac1r\patl_r\phi\Big)\\
&-(B\cdot B)(A\cdot x)\patl_r\Big\{\patl_r\psi\cdot
\frac1r\patl_r\Big(\frac1r\patl_r\phi\Big)\Big\}\\
&+2(A\cdot A)(B\cdot x)\frac1r\patl_r\phi\cdot
\frac1r\patl_r\Big(\frac1r\patl_r\phi\Big)\\
&-(A\cdot B)(B\cdot x)\patl_r\Big\{\patl_r\psi\cdot
\frac1r\patl_r\Big(\frac1r\patl_r\phi\Big)\Big\}\\
&-(A\cdot B)(B\cdot x)\patl_r\Big\{\patl_r\phi\cdot
\frac1r\patl_r\Big(\frac1r\patl_r\psi\Big)\Big\}\\
&-4(B\cdot B)(B\cdot x)\frac1r\patl_r\psi\cdot
\frac1r\patl_r\Big(\frac1r\patl_r\psi\Big)\\
&-2\{(B\cdot B)r^2-(B\cdot x)^2\}(B\cdot x)\frac1r\patl_r\Big\{\frac1r
\patl_r\psi\cdot\frac1r\patl_r\Big(\frac1r\patl_r\psi\Big)\Big\}\\
&+2(A\cdot x)^2(B\cdot x)\frac1r\patl_r\Big\{\frac1r
\patl_r\phi\cdot\frac1r\patl_r\Big(\frac1r\patl_r\phi\Big)\Big\}\\
&+2(A\cdot x)(B\cdot x)^2\frac1r\patl_r\Big\{\frac1r
\patl_r\psi\cdot\frac1r\patl_r\Big(\frac1r\patl_r\phi\Big)\Big\}\\
&+2(A\cdot x)(B\cdot x)^2\frac1r\patl_r\Big\{\frac1r
\patl_r\phi\cdot\frac1r\patl_r\Big(\frac1r\patl_r\psi\Big)\Big\}=0.\\
\end{split}\eeq

Putting \eqref{11ps-RAB} into \eqref{NS-SR11-psi1}, we have
\beq\label{11ps-RAB1}\begin{split}
&4(A\cdot A)(A\cdot x)\frac1r\patl_r\phi\cdot
\frac1r\patl_r\Big(\frac1r\patl_r\phi\Big)\\
&+2\{(A\cdot A)r^2-(A\cdot x)^2\}(A\cdot x)\frac1r\patl_r\Big\{\frac1r
\patl_r\phi\cdot\frac1r\patl_r\Big(\frac1r\patl_r\phi\Big)\Big\}\\
&+(A\cdot B)(A\cdot x)\patl_r\Big\{\patl_r\phi\cdot
\frac1r\patl_r\Big(\frac1r\patl_r\psi\Big)\Big\}\\
&+(A\cdot B)(A\cdot x)\patl_r\Big\{\patl_r\psi\cdot
\frac1r\patl_r\Big(\frac1r\patl_r\phi\Big)\Big\}\\
&-2(B\cdot B)(A\cdot x)\frac1r\patl_r\psi\cdot
\frac1r\patl_r\Big(\frac1r\patl_r\psi\Big)\\
&+2(A\cdot A)(B\cdot x)\frac1r\patl_r\psi\cdot
\frac1r\patl_r\Big(\frac1r\patl_r\phi\Big)\\
&+(A\cdot A)(B\cdot x)\patl_r\Big\{\patl_r\psi\cdot
\frac1r\patl_r\Big(\frac1r\patl_r\phi\Big)\Big\}\\
&+2(A\cdot A)(B\cdot x)\frac1r\patl_r\phi\cdot
\frac1r\patl_r\Big(\frac1r\patl_r\psi\Big)\\
&+(A\cdot A)(B\cdot x)\patl_r\Big\{\patl_r\phi\cdot
\frac1r\patl_r\Big(\frac1r\patl_r\psi\Big)\Big\}\\
&+6(A\cdot B)(B\cdot x)\frac1r\patl_r\psi\cdot
\frac1r\patl_r\Big(\frac1r\patl_r\psi\Big)\\
&+2(A\cdot B)(B\cdot x)r\patl_r\Big\{\frac1r\patl_r\psi\cdot
\frac1r\patl_r\Big(\frac1r\patl_r\psi\Big)\Big\}\\
&-2(A\cdot x)^2(B\cdot x)\frac1r\patl_r\Big\{\frac1r
\patl_r\psi\cdot\frac1r\patl_r\Big(\frac1r\patl_r\phi\Big)\Big\}\\
&-2(A\cdot x)^2(B\cdot x)\frac1r\patl_r\Big\{\frac1r
\patl_r\phi\cdot\frac1r\patl_r\Big(\frac1r\patl_r\psi\Big)\Big\}\\
&-2(A\cdot x)(B\cdot x)^2\frac1r\patl_r\Big\{\frac1r
\patl_r\psi\cdot\frac1r\patl_r\Big(\frac1r\patl_r\psi\Big)\Big\}=0.\\
\end{split}\eeq

Observe that $(A\cdot x)(B\cdot x)$ is equal to either the scalar
product of vectors $(B\cdot x)A$ and $x$, or the scalar product of
vectors $(A\cdot x)B$ and $x$. Use this observation, the equations
\eqref{11ph-RAB1} can be rewritten as follows
\beq\label{11ph-RAB2}\begin{split}
&-6(A\cdot B)A\frac1r\patl_r\phi\cdot
\frac1r\patl_r\Big(\frac1r\patl_r\phi\Big)\\
&-2(A\cdot B)Ar\patl_r\Big\{\frac1r
\patl_r\phi\cdot\frac1r\patl_r\Big(\frac1r\patl_r\phi\Big)\Big\}\\
&-2(B\cdot B)A\frac1r\patl_r\phi\cdot
\frac1r\patl_r\Big(\frac1r\patl_r\psi\Big)\\
&-(B\cdot B)A\patl_r\Big\{\patl_r\phi\cdot
\frac1r\patl_r\Big(\frac1r\patl_r\psi\Big)\Big\}\\
&-2(B\cdot B)A\frac1r\patl_r\psi\cdot
\frac1r\patl_r\Big(\frac1r\patl_r\phi\Big)\\
&-(B\cdot B)A\patl_r\Big\{\patl_r\psi\cdot
\frac1r\patl_r\Big(\frac1r\patl_r\phi\Big)\Big\}\\
&+2A\xi(A\cdot x)(B\cdot x)\frac1r\patl_r\Big\{\frac1r
\patl_r\phi\cdot\frac1r\patl_r\Big(\frac1r\patl_r\phi\Big)\Big\}\\
&+2A\eta(B\cdot x)^2\frac1r\patl_r\Big\{\frac1r
\patl_r\psi\cdot\frac1r\patl_r\Big(\frac1r\patl_r\phi\Big)\Big\}\\
&+2A\zeta(B\cdot x)^2\frac1r\patl_r\Big\{\frac1r
\patl_r\phi\cdot\frac1r\patl_r\Big(\frac1r\patl_r\psi\Big)\Big\}\\
&+2(A\cdot A)B\frac1r\patl_r\phi\cdot
\frac1r\patl_r\Big(\frac1r\patl_r\phi\Big)\\
&-(A\cdot B)B\patl_r\Big\{\patl_r\psi\cdot
\frac1r\patl_r\Big(\frac1r\patl_r\phi\Big)\Big\}\\
&-(A\cdot B)B\patl_r\Big\{\patl_r\phi\cdot
\frac1r\patl_r\Big(\frac1r\patl_r\psi\Big)\Big\}\\
&-4(B\cdot B)B\frac1r\patl_r\psi\cdot
\frac1r\patl_r\Big(\frac1r\patl_r\psi\Big)\\
&-2\{(B\cdot B)r^2-(B\cdot x)^2\}B\frac1r\patl_r\Big\{\frac1r
\patl_r\psi\cdot\frac1r\patl_r\Big(\frac1r\patl_r\psi\Big)\Big\}\\
&+2B(1-\xi)(A\cdot x)^2\frac1r\patl_r\Big\{\frac1r
\patl_r\phi\cdot\frac1r\patl_r\Big(\frac1r\patl_r\phi\Big)\Big\}\\
&+2B(1-\eta)(A\cdot x)(B\cdot x)\frac1r\patl_r\Big\{\frac1r
\patl_r\psi\cdot\frac1r\patl_r\Big(\frac1r\patl_r\phi\Big)\Big\}\\
&+2B(1-\zeta)(A\cdot x)(B\cdot x)\frac1r\patl_r\Big\{\frac1r
\patl_r\phi\cdot\frac1r\patl_r\Big(\frac1r\patl_r\psi\Big)\Big\}=0,\\
\end{split}\eeq
where parameters $\xi,\eta,\zeta\in\R$.

Since the vectors $A$ and $B$ are linearly independent, on the left
hand of equation \eqref{11ph-RAB2}, all coefficients of
$A$ and $B$ are zero. Thus we have
\beq\label{11ph-AC}\begin{split}
&-6(A\cdot B)\frac1r\patl_r\phi\cdot
\frac1r\patl_r\Big(\frac1r\patl_r\phi\Big)\\
&-2(A\cdot B)r\patl_r\Big\{\frac1r
\patl_r\phi\cdot\frac1r\patl_r\Big(\frac1r\patl_r\phi\Big)\Big\}\\
&-2(B\cdot B)\frac1r\patl_r\phi\cdot
\frac1r\patl_r\Big(\frac1r\patl_r\psi\Big)\\
&-(B\cdot B)\patl_r\Big\{\patl_r\phi\cdot
\frac1r\patl_r\Big(\frac1r\patl_r\psi\Big)\Big\}\\
&-2(B\cdot B)\frac1r\patl_r\psi\cdot
\frac1r\patl_r\Big(\frac1r\patl_r\phi\Big)\\
&-(B\cdot B)\patl_r\Big\{\patl_r\psi\cdot
\frac1r\patl_r\Big(\frac1r\patl_r\phi\Big)\Big\}\\
&+2\xi(A\cdot x)(B\cdot x)\frac1r\patl_r\Big\{\frac1r
\patl_r\phi\cdot\frac1r\patl_r\Big(\frac1r\patl_r\phi\Big)\Big\}\\
&+2\eta(B\cdot x)^2\frac1r\patl_r\Big\{\frac1r
\patl_r\psi\cdot\frac1r\patl_r\Big(\frac1r\patl_r\phi\Big)\Big\}\\
&+2\zeta(B\cdot x)^2\frac1r\patl_r\Big\{\frac1r
\patl_r\phi\cdot\frac1r\patl_r\Big(\frac1r\patl_r\psi\Big)\Big\}=0,\\
\end{split}\eeq
\beq\label{11ph-BC}\begin{split}
&2(A\cdot A)\frac1r\patl_r\phi\cdot
\frac1r\patl_r\Big(\frac1r\patl_r\phi\Big)\\
&-(A\cdot B)\patl_r\Big\{\patl_r\psi\cdot
\frac1r\patl_r\Big(\frac1r\patl_r\phi\Big)\Big\}\\
&-(A\cdot B)\patl_r\Big\{\patl_r\phi\cdot
\frac1r\patl_r\Big(\frac1r\patl_r\psi\Big)\Big\}\\
&-4(B\cdot B)\frac1r\patl_r\psi\cdot
\frac1r\patl_r\Big(\frac1r\patl_r\psi\Big)\\
&-2\{(B\cdot B)r^2-(B\cdot x)^2\}\frac1r\patl_r\Big\{\frac1r
\patl_r\psi\cdot\frac1r\patl_r\Big(\frac1r\patl_r\psi\Big)\Big\}\\
&+2(1-\xi)(A\cdot x)^2\frac1r\patl_r\Big\{\frac1r
\patl_r\phi\cdot\frac1r\patl_r\Big(\frac1r\patl_r\phi\Big)\Big\}\\
&+2(1-\eta)(A\cdot x)(B\cdot x)\frac1r\patl_r\Big\{\frac1r
\patl_r\psi\cdot\frac1r\patl_r\Big(\frac1r\patl_r\phi\Big)\Big\}\\
&+2(1-\zeta)(A\cdot x)(B\cdot x)\frac1r\patl_r\Big\{\frac1r
\patl_r\phi\cdot\frac1r\patl_r\Big(\frac1r\patl_r\psi\Big)\Big\}=0.\\
\end{split}\eeq

Let us select orthogonal transformation $\rho_a$ which rotation axis
is vector $A$, and orthogonal transformation $\rho_b$ which rotation
axis is vector $B$.

Applying the orthogonal transformation $y=\rho_bx$ in \eqref{11ph-AC},
we obtain
\beq\label{11ph-AC-RBy}\begin{split}
&-6(A\cdot B)\frac1r\patl_r\phi\cdot
\frac1r\patl_r\Big(\frac1r\patl_r\phi\Big)\\
&-2(A\cdot B)r\patl_r\Big\{\frac1r
\patl_r\phi\cdot\frac1r\patl_r\Big(\frac1r\patl_r\phi\Big)\Big\}\\
&-2(B\cdot B)\frac1r\patl_r\phi\cdot
\frac1r\patl_r\Big(\frac1r\patl_r\psi\Big)\\
&-(B\cdot B)\patl_r\Big\{\patl_r\phi\cdot
\frac1r\patl_r\Big(\frac1r\patl_r\psi\Big)\Big\}\\
&-2(B\cdot B)\frac1r\patl_r\psi\cdot
\frac1r\patl_r\Big(\frac1r\patl_r\phi\Big)\\
&-(B\cdot B)\patl_r\Big\{\patl_r\psi\cdot
\frac1r\patl_r\Big(\frac1r\patl_r\phi\Big)\Big\}\\
&+2\xi(A\cdot y)(B\cdot y)\frac1r\patl_r\Big\{\frac1r
\patl_r\phi\cdot\frac1r\patl_r\Big(\frac1r\patl_r\phi\Big)\Big\}\\
&+2\eta(B\cdot y)^2\frac1r\patl_r\Big\{\frac1r
\patl_r\psi\cdot\frac1r\patl_r\Big(\frac1r\patl_r\phi\Big)\Big\}\\
&+2\zeta(B\cdot y)^2\frac1r\patl_r\Big\{\frac1r
\patl_r\phi\cdot\frac1r\patl_r\Big(\frac1r\patl_r\psi\Big)\Big\}=0.\\
\end{split}\eeq
Calculating the difference of \eqref{11ph-AC} and \eqref{11ph-AC-RBy},
we get
\beq\label{11ph-AC-RByx}\begin{split}
&\xi\{A\cdot(x-\rho_bx)\}(B\cdot x)\frac1r\patl_r\Big\{\frac1r
\patl_r\phi\cdot\frac1r\patl_r\Big(\frac1r\patl_r\phi\Big)\Big\}=0,\\
\end{split}\eeq
\beq\label{11ph-AC-NR1}\begin{split}
&\xi\patl_r\Big\{\frac1r\patl_r\phi\cdot
\frac1r\patl_r\Big(\frac1r\patl_r\phi\Big)\Big\}=0.\\
\end{split}\eeq
In fact, firstly the equation \eqref{11ph-AC-NR1} is satisfied for any
$x\in\R^3-\{x|x=\rho_bx\}$. Next for $x\in\{x|x=\rho_bx\}$, choosing
$x_n\in\R^3-\{x|x=\rho_bx\}$ such that $x_n\rightarrow x$ as $n
\rightarrow\infty$, we can prove that the equation \eqref{11ph-AC-NR1}
is also satisfied.

Putting \eqref{11ph-AC-NR1} into \eqref{11ph-AC}, and using the
orthogonal transformation $z=\rho_ax$, we derive
\beq\label{11ph-AC-RAz}\begin{split}
&-6(A\cdot B)\frac1r\patl_r\phi\cdot
\frac1r\patl_r\Big(\frac1r\patl_r\phi\Big)\\
&-2(A\cdot B)r\patl_r\Big\{\frac1r
\patl_r\phi\cdot\frac1r\patl_r\Big(\frac1r\patl_r\phi\Big)\Big\}\\
&-2(B\cdot B)\frac1r\patl_r\phi\cdot
\frac1r\patl_r\Big(\frac1r\patl_r\psi\Big)\\
&-(B\cdot B)\patl_r\Big\{\patl_r\phi\cdot
\frac1r\patl_r\Big(\frac1r\patl_r\psi\Big)\Big\}\\
&-2(B\cdot B)\frac1r\patl_r\psi\cdot
\frac1r\patl_r\Big(\frac1r\patl_r\phi\Big)\\
&-(B\cdot B)\patl_r\Big\{\patl_r\psi\cdot
\frac1r\patl_r\Big(\frac1r\patl_r\phi\Big)\Big\}\\
&+2\eta(B\cdot z)^2\frac1r\patl_r\Big\{\frac1r
\patl_r\psi\cdot\frac1r\patl_r\Big(\frac1r\patl_r\phi\Big)\Big\}\\
&+2\zeta(B\cdot z)^2\frac1r\patl_r\Big\{\frac1r
\patl_r\phi\cdot\frac1r\patl_r\Big(\frac1r\patl_r\psi\Big)\Big\}=0.\\
\end{split}\eeq
Calculating the difference of \eqref{11ph-AC} and \eqref{11ph-AC-RAz},
we get
\beq\label{11ph-AC-RAzx}\begin{split}
&\eta\{B\cdot(x-\rho_ax)\}\{B\cdot(x+\rho_ax)\}\frac1r\patl_r\Big\{\frac1r
\patl_r\psi\cdot\frac1r\patl_r\Big(\frac1r\patl_r\phi\Big)\Big\}\\
&+\zeta\{B\cdot(x-\rho_ax)\}\{B\cdot(x+\rho_ax)\}\frac1r\patl_r\Big\{\frac1r
\patl_r\phi\cdot\frac1r\patl_r\Big(\frac1r\patl_r\psi\Big)\Big\}=0,\\
\end{split}\eeq
\beq\label{11ph-AC-NR2}\begin{split}
&\eta\patl_r\Big\{\frac1r\patl_r\psi\cdot
\frac1r\patl_r\Big(\frac1r\patl_r\phi\Big)\Big\}\\
&+\zeta\patl_r\Big\{\frac1r\patl_r\phi\cdot
\frac1r\patl_r\Big(\frac1r\patl_r\psi\Big)\Big\}=0.\\
\end{split}\eeq
In fact, firstly the equation \eqref{11ph-AC-NR2} is satisfied for any
$x\in\R^3-\{x|x=\pm\rho_ax\}$. Next for $x\in\{x|x=\pm\rho_ax\}$,
choosing $x_n\in\R^3-\{x|x=\pm\rho_ax\}$ such that $x_n\rightarrow x$
as $n\rightarrow\infty$, we can prove that the equation
\eqref{11ph-AC-NR2} is also satisfied.

Putting \eqref{11ph-AC-NR1} and \eqref{11ph-AC-NR2} into
\eqref{11ph-AC}, we have
\beq\label{11ph-ACR}\begin{split}
&-6(A\cdot B)\frac1r\patl_r\phi\cdot
\frac1r\patl_r\Big(\frac1r\patl_r\phi\Big)\\
&-2(A\cdot B)r\patl_r\Big\{\frac1r
\patl_r\phi\cdot\frac1r\patl_r\Big(\frac1r\patl_r\phi\Big)\Big\}\\
&-2(B\cdot B)\frac1r\patl_r\phi\cdot
\frac1r\patl_r\Big(\frac1r\patl_r\psi\Big)\\
&-(B\cdot B)\patl_r\Big\{\patl_r\phi\cdot
\frac1r\patl_r\Big(\frac1r\patl_r\psi\Big)\Big\}\\
&-2(B\cdot B)\frac1r\patl_r\psi\cdot
\frac1r\patl_r\Big(\frac1r\patl_r\phi\Big)\\
&-(B\cdot B)\patl_r\Big\{\patl_r\psi\cdot
\frac1r\patl_r\Big(\frac1r\patl_r\phi\Big)\Big\}=0.\\
\end{split}\eeq

Applying the orthogonal transformation $y=\rho_bx$ in \eqref{11ph-BC},
we obtain
\beq\label{11ph-BC-RBy}\begin{split}
&2(A\cdot A)\frac1r\patl_r\phi\cdot
\frac1r\patl_r\Big(\frac1r\patl_r\phi\Big)\\
&-(A\cdot B)\patl_r\Big\{\patl_r\psi\cdot
\frac1r\patl_r\Big(\frac1r\patl_r\phi\Big)\Big\}\\
&-(A\cdot B)\patl_r\Big\{\patl_r\phi\cdot
\frac1r\patl_r\Big(\frac1r\patl_r\psi\Big)\Big\}\\
&-4(B\cdot B)\frac1r\patl_r\psi\cdot
\frac1r\patl_r\Big(\frac1r\patl_r\psi\Big)\\
&-2\{(B\cdot B)r^2-(B\cdot y)^2\}\frac1r\patl_r\Big\{\frac1r
\patl_r\psi\cdot\frac1r\patl_r\Big(\frac1r\patl_r\psi\Big)\Big\}\\
&+2(1-\xi)(A\cdot y)^2\frac1r\patl_r\Big\{\frac1r
\patl_r\phi\cdot\frac1r\patl_r\Big(\frac1r\patl_r\phi\Big)\Big\}\\
&+2(1-\eta)(A\cdot y)(B\cdot y)\frac1r\patl_r\Big\{\frac1r
\patl_r\psi\cdot\frac1r\patl_r\Big(\frac1r\patl_r\phi\Big)\Big\}\\
&+2(1-\zeta)(A\cdot y)(B\cdot y)\frac1r\patl_r\Big\{\frac1r
\patl_r\phi\cdot\frac1r\patl_r\Big(\frac1r\patl_r\psi\Big)\Big\}=0.\\
\end{split}\eeq
Calculating the difference of \eqref{11ph-BC} and \eqref{11ph-BC-RBy},
we get
\beq\label{11ph-BC-RByx}\begin{split}
&(1-\xi)\{A\cdot(x-\rho_bx)\}\{A\cdot(x+\rho_bx)\}\frac1r\patl_r\Big\{\frac1r
\patl_r\phi\cdot\frac1r\patl_r\Big(\frac1r\patl_r\phi\Big)\Big\}\\
&+(1-\eta)\{A\cdot(x-\rho_bx)\}(B\cdot x)\frac1r\patl_r\Big\{\frac1r
\patl_r\psi\cdot\frac1r\patl_r\Big(\frac1r\patl_r\phi\Big)\Big\}\\
&+(1-\zeta)\{A\cdot(x-\rho_bx)\}(B\cdot x)\frac1r\patl_r\Big\{\frac1r
\patl_r\phi\cdot\frac1r\patl_r\Big(\frac1r\patl_r\psi\Big)\Big\}=0.\\
\end{split}\eeq
Using the same arguments as in the proof of \eqref{11ph-AC-NR1}, we
have
\beq\label{11ph-BC-NR12b}\begin{split}
&(1-\xi)\{A\cdot(x+\rho_bx)\}\frac1r\patl_r\Big\{\frac1r
\patl_r\phi\cdot\frac1r\patl_r\Big(\frac1r\patl_r\phi\Big)\Big\}\\
&+(1-\eta)(B\cdot x)\frac1r\patl_r\Big\{\frac1r
\patl_r\psi\cdot\frac1r\patl_r\Big(\frac1r\patl_r\phi\Big)\Big\}\\
&+(1-\zeta)(B\cdot x)\frac1r\patl_r\Big\{\frac1r
\patl_r\phi\cdot\frac1r\patl_r\Big(\frac1r\patl_r\psi\Big)\Big\}=0.\\
\end{split}\eeq

Using the orthogonal transformation $z=\rho_ax$ in
\eqref{11ph-BC-NR12b}, we derive
\beq\label{11ph-BC-NR12az}\begin{split}
&(1-\xi)\{A\cdot(z+\rho_bz)\}\frac1r\patl_r\Big\{\frac1r
\patl_r\phi\cdot\frac1r\patl_r\Big(\frac1r\patl_r\phi\Big)\Big\}\\
&+(1-\eta)(B\cdot z)\frac1r\patl_r\Big\{\frac1r
\patl_r\psi\cdot\frac1r\patl_r\Big(\frac1r\patl_r\phi\Big)\Big\}\\
&+(1-\zeta)(B\cdot z)\frac1r\patl_r\Big\{\frac1r
\patl_r\phi\cdot\frac1r\patl_r\Big(\frac1r\patl_r\psi\Big)\Big\}=0.\\
\end{split}\eeq
Using $\rho_a\rho_b=\rho_b\rho_a$ and solving the difference of
\eqref{11ph-BC-NR12b} and \eqref{11ph-BC-NR12az}, we get
\beq\label{11ph-BC-NR12azy}\begin{split}
&(1-\eta)\{B\cdot(x-\rho_ax)\}\frac1r\patl_r\Big\{\frac1r
\patl_r\psi\cdot\frac1r\patl_r\Big(\frac1r\patl_r\phi\Big)\Big\}\\
&+(1-\zeta)\{B\cdot(x-\rho_ax)\}\frac1r\patl_r\Big\{\frac1r
\patl_r\phi\cdot\frac1r\patl_r\Big(\frac1r\patl_r\psi\Big)\Big\}=0.\\
\end{split}\eeq
Given $r$, thanks $x\in{\mathbb S}^2$ is arbitrary, we obtain
\beq\label{11ph-BC-NR2}\begin{split}
&(1-\eta)\patl_r\Big\{\frac1r
\patl_r\psi\cdot\frac1r\patl_r\Big(\frac1r\patl_r\phi\Big)\Big\}\\
&+(1-\zeta)\patl_r\Big\{\frac1r
\patl_r\phi\cdot\frac1r\patl_r\Big(\frac1r\patl_r\psi\Big)\Big\}=0.\\
\end{split}\eeq
Inserting \eqref{11ph-BC-NR2} into \eqref{11ph-BC-NR12b}, we have
\beq\label{11ph-BC-NR1}\begin{split}
&(1-\xi)\patl_r\Big\{\frac1r\patl_r\phi\cdot
\frac1r\patl_r\Big(\frac1r\patl_r\phi\Big)\Big\}=0.\\
\end{split}\eeq
In fact, firstly the equation \eqref{11ph-BC-NR1} is satisfied for any
$x\in\R^3-\{x|x=-\rho_bx\}$. Next for $x\in\{x|x=-\rho_bx\}$, choosing
$x_n\in\R^3-\{x|x=-\rho_bx\}$ such that $x_n\rightarrow x$ as $n
\rightarrow\infty$, we can prove that the equation \eqref{11ph-BC-NR1}
is also satisfied.

Putting \eqref{11ph-BC-NR1} \eqref{11ph-BC-NR2} into \eqref{11ph-BC},
and employing the orthogonal transformation $z=\rho_ax$, we derive
\beq\label{11ph-BC-RAz}\begin{split}
&2(A\cdot A)\frac1r\patl_r\phi\cdot
\frac1r\patl_r\Big(\frac1r\patl_r\phi\Big)\\
&-(A\cdot B)\patl_r\Big\{\patl_r\psi\cdot
\frac1r\patl_r\Big(\frac1r\patl_r\phi\Big)\Big\}\\
&-(A\cdot B)\patl_r\Big\{\patl_r\phi\cdot
\frac1r\patl_r\Big(\frac1r\patl_r\psi\Big)\Big\}\\
&-4(B\cdot B)\frac1r\patl_r\psi\cdot
\frac1r\patl_r\Big(\frac1r\patl_r\psi\Big)\\
&-2\{(B\cdot B)r^2-(B\cdot z)^2\}\frac1r\patl_r\Big\{\frac1r
\patl_r\psi\cdot\frac1r\patl_r\Big(\frac1r\patl_r\psi\Big)\Big\}=0.\\
\end{split}\eeq
Calculating the difference of \eqref{11ph-BC} and \eqref{11ph-BC-RAz},
we get
\beq\label{11ph-BC-RAzx}\begin{split}
\{B\cdot(x-\rho_ax)\}\{B\cdot(x+\rho_ax)\}\frac1r\patl_r\Big\{\frac1r
\patl_r\psi\cdot\frac1r\patl_r\Big(\frac1r\patl_r\psi\Big)\Big\}=0.\\
\end{split}\eeq
By the same arguments as in the proof of \eqref{11ph-AC-NR2}, we prove
\beq\label{11-NR-2}\begin{split}
\patl_r\Big\{\frac1r\patl_r\psi\cdot
\frac1r\patl_r\Big(\frac1r\patl_r\psi\Big)\Big\}=0.\\
\end{split}\eeq

Inserting \eqref{11ph-BC-NR1} \eqref{11ph-BC-NR2} \eqref{11-NR-2}
into \eqref{11ph-BC}, we have
\beq\label{11ph-BCR}\begin{split}
&2(A\cdot A)\frac1r\patl_r\phi\cdot
\frac1r\patl_r\Big(\frac1r\patl_r\phi\Big)\\
&-(A\cdot B)\patl_r\Big\{\patl_r\psi\cdot
\frac1r\patl_r\Big(\frac1r\patl_r\phi\Big)\Big\}\\
&-(A\cdot B)\patl_r\Big\{\patl_r\phi\cdot
\frac1r\patl_r\Big(\frac1r\patl_r\psi\Big)\Big\}\\
&-4(B\cdot B)\frac1r\patl_r\psi\cdot
\frac1r\patl_r\Big(\frac1r\patl_r\psi\Big)=0.\\
\end{split}\eeq

Putting \eqref{11ph-AC-NR1} \eqref{11ph-BC-NR1} together, we derive
\beq\label{11-NR-1}\begin{split}
\patl_r\Big\{\frac1r\patl_r\phi\cdot\frac1r\patl_r
\Big(\frac1r\patl_r\phi\Big)\Big\}=0.\\
\end{split}\eeq
Putting \eqref{11ph-AC-NR2} \eqref{11ph-BC-NR2} together, we derive
\beq\label{11-NR-12}\begin{split}
&\patl_r\Big\{\frac1r\patl_r\psi\cdot
\frac1r\patl_r\Big(\frac1r\patl_r\phi\Big)\Big\}\\
&+\patl_r\Big\{\frac1r\patl_r\phi\cdot
\frac1r\patl_r\Big(\frac1r\patl_r\psi\Big)\Big\}=0.\\
\end{split}\eeq

Inserting \eqref{11-NR-1} \eqref{11-NR-2} \eqref{11-NR-12} into
\eqref{11ps-RAB1}, we obtain
\beq\label{11ps-RAB12}\begin{split}
&4(A\cdot A)(A\cdot x)\frac1r\patl_r\phi\cdot
\frac1r\patl_r\Big(\frac1r\patl_r\phi\Big)\\
&+(A\cdot B)(A\cdot x)\patl_r\Big\{\patl_r\phi\cdot
\frac1r\patl_r\Big(\frac1r\patl_r\psi\Big)\Big\}\\
&+(A\cdot B)(A\cdot x)\patl_r\Big\{\patl_r\psi\cdot
\frac1r\patl_r\Big(\frac1r\patl_r\phi\Big)\Big\}\\
&-2(B\cdot B)(A\cdot x)\frac1r\patl_r\psi\cdot
\frac1r\patl_r\Big(\frac1r\patl_r\psi\Big)\\
&+2(A\cdot A)(B\cdot x)\frac1r\patl_r\psi\cdot
\frac1r\patl_r\Big(\frac1r\patl_r\phi\Big)\\
&+(A\cdot A)(B\cdot x)\patl_r\Big\{\patl_r\psi\cdot
\frac1r\patl_r\Big(\frac1r\patl_r\phi\Big)\Big\}\\
&+2(A\cdot A)(B\cdot x)\frac1r\patl_r\phi\cdot
\frac1r\patl_r\Big(\frac1r\patl_r\psi\Big)\\
&+(A\cdot A)(B\cdot x)\patl_r\Big\{\patl_r\phi\cdot
\frac1r\patl_r\Big(\frac1r\patl_r\psi\Big)\Big\}\\
&+6(A\cdot B)(B\cdot x)\frac1r\patl_r\psi\cdot
\frac1r\patl_r\Big(\frac1r\patl_r\psi\Big)=0.\\
\end{split}\eeq

Thanks vectors $A$ and $B$ are linearly independent, the coefficients
of $A$ and $B$ in the equation \eqref{11ps-RAB12} are zero.
Therefore we have
\beq\label{11ps-ACR}\begin{split}
&4(A\cdot A)\frac1r\patl_r\phi\cdot
\frac1r\patl_r\Big(\frac1r\patl_r\phi\Big)\\
&+(A\cdot B)\patl_r\Big\{\patl_r\phi\cdot
\frac1r\patl_r\Big(\frac1r\patl_r\psi\Big)\Big\}\\
&+(A\cdot B)\patl_r\Big\{\patl_r\psi\cdot
\frac1r\patl_r\Big(\frac1r\patl_r\phi\Big)\Big\}\\
&-2(B\cdot B)\frac1r\patl_r\psi\cdot
\frac1r\patl_r\Big(\frac1r\patl_r\psi\Big)=0,\\
\end{split}\eeq
\beq\label{11ps-BCR}\begin{split}
&2(A\cdot A)\frac1r\patl_r\psi\cdot
\frac1r\patl_r\Big(\frac1r\patl_r\phi\Big)\\
&+(A\cdot A)\patl_r\Big\{\patl_r\psi\cdot
\frac1r\patl_r\Big(\frac1r\patl_r\phi\Big)\Big\}\\
&+2(A\cdot A)\frac1r\patl_r\phi\cdot
\frac1r\patl_r\Big(\frac1r\patl_r\psi\Big)\\
&+(A\cdot A)\patl_r\Big\{\patl_r\phi\cdot
\frac1r\patl_r\Big(\frac1r\patl_r\psi\Big)\Big\}\\
&+6(A\cdot B)\frac1r\patl_r\psi\cdot
\frac1r\patl_r\Big(\frac1r\patl_r\psi\Big)=0.\\
\end{split}\eeq

The equation \eqref{11-NR-1} implies that
\beq\label{11-NR-1ph}\begin{split}
&\frac1r\patl_r\phi=\pm\{f_2r^2+f_1\}^{1/2},\;\;f_1\ge0,\\
\end{split}\eeq
\[\begin{split}
&\patl_r\Big(\frac1r\patl_r\phi\Big)=\pm f_2r\{f_2r^2+f_1\}^{-1/2},\\
\end{split}\]
\[\begin{split}
&\Del\phi=\frac3r\patl_r\phi+r\patl_r\Big(\frac1r\patl_r\phi\Big)\\
=&\pm4\{f_2r^2+f_1\}^{1/2}\mp f_1\{f_2r^2+f_1\}^{-1/2},\\
&\patl_r\Del\phi=\pm4f_2r\{f_2r^2+f_1\}^{-1/2}\pm f_1f_2r\{f_2r^2+f_1\}^{-3/2},\\
&\patl_r\{\patl_r\Del\phi\}
=\pm2f_1f_2\{f_2r^2+f_1\}^{-3/2}\pm3f_1^2f_2\{f_2r^2+f_1\}^{-5/2},\\
&\patl_r^2\{\patl_r\Del\phi\}
=\mp6f_1f_2^2r\{f_2r^2+f_1\}^{-5/2}\mp15f_1^2f_2^2r\{f_2r^2+f_1\}^{-7/2},\\
\end{split}\]
\beq\label{11-NR-1ph-L}\begin{split}
&\Del\{\patl_r\Del\phi\}=\patl_r^2\{\patl_r\Del\phi\}+\frac2r\patl_r\{\patl_r\Del\phi\}\\
=&\mp\frac2rf_1f_2\{f_2r^2+f_1\}^{-3/2}\mp\frac3rf_1^2f_2\{f_2r^2+f_1\}^{-5/2}\\
&\pm\frac{15}rf_1^3f_2\{f_2r^2+f_1\}^{-7/2},\\
\end{split}\eeq
\beq\label{11-NR-1ph-t}\begin{split}
\patl_t\{\patl_r\Del\phi\}
=&\pm2f_{2t}r\{f_2r^2+f_1\}^{-1/2}\\
&\pm\frac32f_1f_{2t}r\{f_2r^2+f_1\}^{-3/2}\mp f_{1t}f_2r\{f_2r^2+f_1\}^{-3/2}\\
&\pm\frac32f_1^2f_{2t}r\{f_2r^2+f_1\}^{-5/2}
\mp\frac32f_1f_{1t}f_2r\{f_2r^2+f_1\}^{-5/2},\\
\end{split}\eeq
where $f_2$ and $f_1$ are any functions of $t$.

Putting \eqref{11-NR-1ph-L}\eqref{11-NR-1ph-t} into equation
\eqref{11ph-RAB}, we derive
\beq\label{11-phRAB-1}\begin{split}
&2f_{2t}r\{f_2r^2+f_1\}^{-1/2}\\
&+\frac32f_1f_{2t}r\{f_2r^2+f_1\}^{-3/2}
-f_{1t}f_2r\{f_2r^2+f_1\}^{-3/2}\\
&+\frac32 f_1^2f_{2t}r\{f_2r^2+f_1\}^{-5/2}
-\frac32 f_1f_{1t}f_2r\{f_2r^2+f_1\}^{-5/2}\\
=&-2\nu\frac1rf_1f_2\{f_2r^2+f_1\}^{-3/2}
-3\nu\frac1rf_1^2f_2\{f_2r^2+f_1\}^{-5/2}\\
&+15\nu\frac1rf_1^3f_2\{f_2r^2+f_1\}^{-7/2}.\\
\end{split}\eeq
Assume that $f_2\not=0$. Let $r\rightarrow\infty$ in equation
\eqref{11-phRAB-1}, we obtain $f_{2t}=0$. Then equation
\eqref{11-phRAB-1} implies that
\beq\label{11-phRAB-2}\begin{split}
&f_{1t}r^3\{f_2r^2+f_1\}^{-3/2}
+\frac32f_{1t}f_1r^3\{f_2r^2+f_1\}^{-5/2}\\
=&2\nu f_1r\{f_2r^2+f_1\}^{-3/2}
+3\nu f_1^2r\{f_2r^2+f_1\}^{-5/2}\\
&-15\nu f_1^3r\{f_2r^2+f_1\}^{-7/2}.\\
\end{split}\eeq
Let $r\rightarrow\infty$ in equation \eqref{11-phRAB-2}, we obtain
$f_{1t}=0$. Then equation \eqref{11-phRAB-2} implies that
\beq\label{11-phRAB-3}\begin{split}
0=2\{f_2r^2+f_1\}^2+3f_1\{f_2r^2+f_1\}-15f_1^2,\;\;\forall r>0.\\
\end{split}\eeq
Thus $f_2=f_1=0$. This is contradictory.

Therefore $f_2=0$ and $\patl_r\phi=f_1r$ where $f_1$ is any function
of $t$.

Similarly, the equation \eqref{11-NR-2} implies that
\beq\label{11-NR-2ps}\begin{split}
\frac1r\patl_r\psi=\pm\{g_2r^2+g_1\}^{1/2},\;\;g_1\ge0.\\
\end{split}\eeq
Putting \eqref{11-NR-2ps} into equation \eqref{11ps-RAB}, we derive
\beq\label{11-psRAB-1}\begin{split}
&2g_{2t}r\{g_2r^2+g_1\}^{-1/2}\\
&+\frac32g_1g_{2t}r\{g_2r^2+g_1\}^{-3/2}
-g_{1t}g_2r\{g_2r^2+g_1\}^{-3/2}\\
&+\frac32 g_1^2g_{2t}r\{g_2r^2+g_1\}^{-5/2}
-\frac32 g_1g_{1t}g_2r\{g_2r^2+g_1\}^{-5/2}\\
=&-2\nu\frac1rg_1g_2\{g_2r^2+g_1\}^{-3/2}
-3\nu\frac1rg_1^2g_2\{g_2r^2+g_1\}^{-5/2}\\
&+15\nu\frac1rg_1^3g_2\{g_2r^2+g_1\}^{-7/2}.\\
\end{split}\eeq
Assume that $g_2\not=0$. Let $r\rightarrow\infty$ in equation
\eqref{11-psRAB-1}, we obtain $g_{2t}=0$. Then equation
\eqref{11-psRAB-1} implies that
\beq\label{11-psRAB-2}\begin{split}
&g_{1t}r^3\{g_2r^2+g_1\}^{-3/2}
+\frac32g_{1t}g_1r^3\{g_2r^2+g_1\}^{-5/2}\\
=&2\nu g_1r\{g_2r^2+g_1\}^{-3/2}
+3\nu g_1^2r\{g_2r^2+g_1\}^{-5/2}\\
&-15\nu g_1^3r\{g_2r^2+g_1\}^{-7/2}.\\
\end{split}\eeq
Let $r\rightarrow\infty$ in equation \eqref{11-psRAB-2}, we obtain
$g_{1t}=0$. Then equation \eqref{11-psRAB-2} implies that
\beq\label{11-psRAB-3}\begin{split}
0=2\{g_2r^2+g_1\}^2+3g_1\{g_2r^2+g_1\}-15g_1^2,\;\;\forall r>0.\\
\end{split}\eeq
Thus $g_2=g_1=0$. This is contradictory.

Therefore $g_2=0$ and $\patl_r\psi=g_1r$ where $g_1$ is any function
of $t$.

Provided $\patl_r\phi=f_1r$ and $\patl_r\psi=g_1r$, then all equations
\eqref{11-NR-1} \eqref{11-NR-2} \eqref{11-NR-12} \eqref{11ph-ACR}
\eqref{11ph-BCR} \eqref{11ps-ACR} \eqref{11ps-BCR} \eqref{11ph-RAB}
and \eqref{11ps-RAB} are satisfied. Therefore the equations
\eqref{NS-SR11-phi1} \eqref{NS-SR11-psi1} are satisfied.

In summary, Theorem \ref{SR11-Thm} is proved.
\end{proof}

\section{(2,2)-Symplectic Representation and \\Radial Symmetry Breaking
in ${\mathbb{R}}^3$}
\setcounter{equation}{0}

In this section, we assume that the velocity vector $u$ holds the
following (2,2)-symplectic representation
\beq\label{Sym-Rep-22-R3}\begin{split}
u(t,x)=&\{(A\times\nab)\times\nab\}\phi(t,x)
+\{(B\times\nab)\times\nab\}\psi(t,x)\\
=&\big(\nab^1_{att},\nab^2_{att},\nab^3_{att}\big)\phi(t,x)
+\big(\nab^1_{btt},\nab^2_{btt},\nab^3_{btt}\big)\psi(t,x),\\
\end{split}\eeq
where vectors $A=(a_1,a_2,a_3)\in\R^3-\{0\}$ and $B=(b_1,b_2,b_3)\in
\R^3-\{0\}$ are linearly independent, $(A\times\nab)\times\nab=
\big(\nab^1_{att},\nab^2_{att},\nab^3_{att}\big)$, $(B\times\nab)
\times\nab=\big(\nab^1_{btt},\nab^2_{btt},\nab^3_{btt}\big)$,
\beq\label{Def-att}\begin{split}
&\nab^1_{att}=(A\cdot\nab)\patl_1-a_1\Del,\\
&\nab^2_{att}=(A\cdot\nab)\patl_2-a_2\Del,\\
&\nab^3_{att}=(A\cdot\nab)\patl_3-a_3\Del,\\
\end{split}\eeq
\beq\label{Def-btt}\begin{split}
&\nab^1_{btt}=(B\cdot\nab)\patl_1-b_1\Del,\\
&\nab^2_{btt}=(B\cdot\nab)\patl_2-b_2\Del,\\
&\nab^3_{btt}=(B\cdot\nab)\patl_3-b_3\Del.\\
\end{split}\eeq

Thanks the following observations
\beq\label{SR22-phi}\begin{split}
(B\times\nab)\cdot u(t,x)&=-\Big(\{(A\times\nab)\times(B\times\nab)\}
\cdot\nab\Big)\phi(t,x)\\
&=-(A\times B)\cdot\nab\Del\phi(t,x),\\
\end{split}\eeq
\beq\label{SR22-psi}\begin{split}
(A\times\nab)\cdot u(t,x)&=\Big(\{(A\times\nab)\times(B\times\nab)\}
\cdot\nab\Big)\psi(t,x)\\
&=(A\times B)\cdot\nab\Del\psi(t,x),\\
\end{split}\eeq
taking scalar product of equation \eqref{NS1} with $B\times\nab$,
we have
\beq\label{NS22-phi}
(A\times B)\cdot\nab\Del\{\phi_t-\nu\Del\phi\}
-(B\times\nab)\cdot\{(u\cdot\nab)u\}=0.
\eeq
And taking scalar product of equation \eqref{NS1} with $A\times\nab$,
we derive
\beq\label{NS22-psi}
(A\times B)\cdot\nab\Del\{\psi_t-\nu\Del\psi\}
+(A\times\nab)\cdot\{(u\cdot\nab)u\}=0.
\eeq

Introduce symbol $\patl^j_{at}$ and $\patl^j_{bt}$ as follows
\beq\label{Def-at-bt}\begin{split}
&\patl^1_{at}=a_2\patl_3-a_3\patl_2,\;\patl^2_{at}=a_3\patl_1-a_1\patl_3,\;
\patl^3_{at}=a_1\patl_2-a_2\patl_1,\\
&\patl^1_{bt}=b_2\patl_3-b_3\patl_2,\;\patl^2_{bt}=b_3\patl_1-b_1\patl_3,\;
\patl^3_{bt}=b_1\patl_2-b_2\patl_1.\\
\end{split}\eeq
Then $A\times\nab=\big(\patl^1_{at},\patl^2_{at},\patl^3_{at}\big)$ and $
B\times\nab=\big(\patl^1_{bt},\patl^2_{bt},\patl^3_{bt}\big)$. Let us
rewrite the nonlinear terms in the equation \eqref{NS22-phi}
\beq\label{NT-Eq22-phi}\begin{split}
&(B\times\nab)\cdot\{(u\cdot\nab)u\}=\patl^j_{bt}u^k\patl_ku^j\\
=&\patl^j_{bt}\big(\nab^k_{att}\phi+\nab^k_{btt}\psi\big)
\patl_k\big(\nab^j_{att}\phi+\nab^j_{btt}\psi\big)\\
=&-\big(\nab^k_{att}\phi+\nab^k_{btt}\psi\big)
\patl_k(A\times B)\cdot\nab\Del\phi\\
&+\big(\patl^j_{bt}\nab^k_{att}\phi+\patl^j_{bt}\nab^k_{btt}\psi\big)
\patl_k\big(\nab^j_{att}\phi+\nab^j_{btt}\psi\big)\\
=&-(A\times B)\cdot\nab\big\{\big(\nab^k_{att}\phi+\nab^k_{btt}\psi
\big)\patl_k\Del\phi\big\}\\
&+\big\{(A\times B)\cdot\nab\nab^k_{att}\phi
+(A\times B)\cdot\nab\nab^k_{btt}\psi\big\}\patl_k\Del\phi\\
&+\big(\patl^j_{bt}\nab^k_{att}\phi+\patl^j_{bt}\nab^k_{btt}\psi\big)
\patl_k\big(\nab^j_{att}\phi+\nab^j_{btt}\psi\big)\\
=&-(A\times B)\cdot\nab\big(\{A\cdot\nab\patl_k\phi\}\patl_k\Del\phi
-\Del\phi\{A\cdot\nab\Del\phi\}\big)\\
&-(A\times B)\cdot\nab\big(\{B\cdot\nab\patl_k\psi\}\patl_k\Del\phi
-\Del\psi\{B\cdot\nab\Del\phi\}\big)\\
&+\nab\Del\phi\cdot(B\times\nab)(A\cdot\nab)^2\phi
+\nab\Del\phi\cdot(B\times\nab)(A\cdot\nab)(B\cdot\nab)\psi\\
&+\nab\Del\psi\cdot(B\times\nab)(B\cdot\nab)^2\psi
+\nab\Del\psi\cdot(B\times\nab)(A\cdot\nab)(B\cdot\nab)\phi,\\
\end{split}\eeq
where
\beq\label{NT-Att-k}\begin{split}
&\{\nab^k_{att}\phi\}\patl_k
=\big\{\{(A\cdot\nab)\patl_k-a_k\Del\}\phi\big\}\patl_k\\
=&\{(A\cdot\nab)\patl_k\phi\}\patl_k-\Del\phi\;\;A\cdot\nab,\\
\end{split}\eeq

\beq\label{NT-Btt-k}\begin{split}
&\{\nab^k_{btt}\psi\}\patl_k
=\big\{\{(B\cdot\nab)\patl_k-b_k\Del\}\psi\big\}\patl_k\\
=&\{(B\cdot\nab)\patl_k\psi\}\patl_k-\Del\psi\;\;B\cdot\nab,\\
\end{split}\eeq

\beq\label{NT22-bt-jk}\begin{split}
&\big(\patl^j_{bt}\nab^k_{att}\phi+\patl^j_{bt}\nab^k_{btt}\psi\big)
\patl_k\big(\nab^j_{att}\phi+\nab^j_{btt}\psi\big)\\
=&\big(\patl^j_{bt}\nab^k_{att}\phi+\patl^j_{bt}\nab^k_{btt}\psi\big)
\big(\{(A\cdot\nab)\patl_j-a_j\Del\}\patl_k\phi
+\{(B\cdot\nab)\patl_j-b_j\Del\}\patl_k\psi\big),\\
\end{split}\eeq

\beq\label{NT22-bt-phh1}\begin{split}
&\patl^j_{bt}\nab^k_{att}\phi\;\;\patl_j(A\cdot\nab)\patl_k\phi\\
=&\{(A\cdot\nab)\patl_k-a_k\Del\}\patl^j_{bt}\phi\;\;\patl_k\patl_j(A\cdot\nab)\phi\\
=&\patl^j_{bt}(A\cdot\nab)\patl_k\phi\;\;\patl_j(A\cdot\nab)\patl_k\phi
-\patl^j_{bt}\Del\phi\;\;\patl_j(A\cdot\nab)(A\cdot\nab)\phi\\
=&\patl_j\Del\phi\;\;\patl^j_{bt}(A\cdot\nab)^2\phi\\
=&\nab\Del\phi\cdot(B\times\nab)(A\cdot\nab)^2\phi,\\
\end{split}\eeq

\beq\label{NT22-bt-phh2}\begin{split}
-\patl^j_{bt}\nab^k_{att}\phi\;\;a_j\Del\patl_k\phi
=-(A\times B)\cdot\nab\nab^k_{att}\phi\;\;\patl_k\Del\phi,\\
\end{split}\eeq

\beq\label{NT22-bt-phs1}\begin{split}
&\patl^j_{bt}\nab^k_{att}\phi\;\;\patl_j(B\cdot\nab)\patl_k\psi\\
=&\{(A\cdot\nab)\patl_k-a_k\Del\}\patl^j_{bt}\phi\;\;\patl_k\patl_j(B\cdot\nab)\psi\\
=&\patl^j_{bt}(A\cdot\nab)\patl_k\phi\;\;\patl_j(B\cdot\nab)\patl_k\psi
-\patl^j_{bt}\Del\phi\;\;\patl_j(A\cdot\nab)(B\cdot\nab)\psi\\
=&-\patl_j(A\cdot\nab)\patl_k\phi\;\;\patl^j_{bt}(B\cdot\nab)\patl_k\psi
+\patl_j\Del\phi\;\;\patl^j_{bt}(A\cdot\nab)(B\cdot\nab)\psi\\
=&-\nab(A\cdot\nab)\patl_k\phi\cdot(B\times\nab)(B\cdot\nab)\patl_k\psi
+\nab\Del\phi\cdot(B\times\nab)(A\cdot\nab)(B\cdot\nab)\psi,\\
\end{split}\eeq

\beq\label{NT22-bt-phs2}\begin{split}
&-\patl^j_{bt}\nab^k_{att}\phi\;\; b_j\Del\patl_k\psi=0,\\
\end{split}\eeq

\beq\label{NT22-bt-psh1}\begin{split}
&\patl^j_{bt}\nab^k_{btt}\psi\;\;\patl_j(A\cdot\nab)\patl_k\phi\\
=&\patl^j_{bt}\{(B\cdot\nab)\patl_k-b_k\Del\}\psi\;\;\patl_j(A\cdot\nab)\patl_k\phi\\
=&\patl^j_{bt}(B\cdot\nab)\patl_k\psi\;\;\patl_j(A\cdot\nab)\patl_k\phi
-\patl^j_{bt}\Del\psi\;\;\patl_j(A\cdot\nab)(B\cdot\nab)\phi\\
=&\patl^j_{bt}(B\cdot\nab)\patl_k\psi\;\;\patl_j(A\cdot\nab)\patl_k\phi
+\patl_j\Del\psi\;\;\patl^j_{bt}(A\cdot\nab)(B\cdot\nab)\phi\\
=&(B\times\nab)(B\cdot\nab)\patl_k\psi\cdot\nab(A\cdot\nab)\patl_k\phi
+\nab\Del\psi\cdot(B\times\nab)(A\cdot\nab)(B\cdot\nab)\phi,\\
\end{split}\eeq

\beq\label{NT22-bt-psh2}\begin{split}
-\patl^j_{bt}\nab^k_{btt}\psi\;\; a_j\Del\patl_k\phi
=-(A\times B)\cdot\nab\nab^k_{btt}\psi\;\;\patl_k\Del\phi,\\
\end{split}\eeq

\beq\label{NT22-bt-pss1}\begin{split}
&\patl^j_{bt}\nab^k_{btt}\psi\;\;\patl_j(B\cdot\nab)\patl_k\psi\\
=&\patl^j_{bt}\{(B\cdot\nab)\patl_k-b_k\Del\}\psi\;\;\patl_j(B\cdot\nab)\patl_k\psi\\
=&\patl^j_{bt}(B\cdot\nab)\patl_k\psi\;\;\patl_j(B\cdot\nab)\patl_k\psi
-\patl^j_{bt}\Del\psi\;\;\patl_j(B\cdot\nab)b_k\patl_k\psi\\
=&-\patl^j_{bt}\Del\psi\;\;\patl_j(B\cdot\nab)^2\psi\\
=&\patl_j\Del\psi\;\;\patl^j_{bt}(B\cdot\nab)^2\psi
=\nab\Del\psi\cdot(B\times\nab)(B\cdot\nab)^2\psi,\\
\end{split}\eeq

\beq\label{NT22-bt-pss2}\begin{split}
&-\patl^j_{bt}\nab^k_{btt}\psi\;\;b_j\Del\patl_k\psi=0.\\
\end{split}\eeq

Similarly we rewrite the nonlinear terms in the equation
\eqref{NS22-psi}
\beq\label{NT-Eq22-psi}\begin{split}
&(A\times\nab)\cdot\{(u\cdot\nab)u\}=\patl^j_{at}u^k\patl_ku^j\\
=&\patl^j_{at}\big(\nab^k_{att}\phi+\nab^k_{btt}\psi\big)
\patl_k\big(\nab^j_{att}\phi+\nab^j_{btt}\psi\big)\\
=&\big(\nab^k_{att}\phi+\nab^k_{btt}\psi\big)
\patl_k(A\times B)\cdot\nab\Del\psi\\
&+\big(\patl^j_{at}\nab^k_{att}\phi+\patl^j_{at}\nab^k_{btt}\psi\big)
\patl_k\big(\nab^j_{att}\phi+\nab^j_{btt}\psi\big)\\
=&(A\times B)\cdot\nab\big\{\big(\nab^k_{att}\phi+\nab^k_{btt}\psi
\big)\patl_k\Del\psi\big\}\\
&-\big\{(A\times B)\cdot\nab\nab^k_{att}\phi+(A\times B)\cdot\nab
\nab^k_{btt}\psi\big\}\patl_k\Del\psi\\
&+\big(\patl^j_{at}\nab^k_{att}\phi+\patl^j_{at}\nab^k_{btt}\psi\big)
\patl_k\big(\nab^j_{att}\phi+\nab^j_{btt}\psi\big)\\
=&(A\times B)\cdot\nab\big(\{A\cdot\nab\patl_k\phi\}\patl_k\Del\psi
-\Del\phi\{A\cdot\nab\Del\psi\}\big)\\
&+(A\times B)\cdot\nab\big(\{B\cdot\nab\patl_k\psi\}\patl_k\Del\psi
-\Del\psi\{B\cdot\nab\Del\psi\}\big)\\
&+\nab\Del\phi\cdot(A\times\nab)(A\cdot\nab)^2\phi
+\nab\Del\phi\cdot(A\times\nab)(A\cdot\nab)(B\cdot\nab)\psi\\
&+\nab\Del\psi\cdot(A\times\nab)(B\cdot\nab)^2\psi
+\nab\Del\psi\cdot(A\times\nab)(A\cdot\nab)(B\cdot\nab)\phi,\\
\end{split}\eeq
where
\beq\label{NT22-at-jk}\begin{split}
&\big(\patl^j_{at}\nab^k_{att}\phi+\patl^j_{at}\nab^k_{btt}\psi\big)
\patl_k\big(\nab^j_{att}\phi+\nab^j_{btt}\psi\big)\\
=&\big(\patl^j_{at}\nab^k_{att}\phi+\patl^j_{at}\nab^k_{btt}\psi\big)
\big(\{(A\cdot\nab)\patl_j-a_j\Del\}\patl_k\phi
+\{(B\cdot\nab)\patl_j-b_j\Del\}\patl_k\psi\big),\\
\end{split}\eeq

\beq\label{NT22-at-phh1}\begin{split}
&\patl^j_{at}\nab^k_{att}\phi\;\;\patl_j(A\cdot\nab)\patl_k\phi\\
=&\{(A\cdot\nab)\patl_k-a_k\Del\}\patl^j_{at}\phi\;\;\patl_k\patl_j(A\cdot\nab)\phi\\
=&\patl^j_{at}(A\cdot\nab)\patl_k\phi\;\;\patl_j(A\cdot\nab)\patl_k\phi
-\patl^j_{at}\Del\phi\;\;\patl_j(A\cdot\nab)(A\cdot\nab)\phi\\
=&-\patl^j_{at}\Del\phi\;\;\patl_j(A\cdot\nab)^2\phi\\
=&\patl_j\Del\phi\;\;\patl^j_{at}(A\cdot\nab)^2\phi
=\nab\Del\phi\cdot(A\times\nab)(A\cdot\nab)^2\phi,\\
\end{split}\eeq

\beq\label{NT22-at-phh2}\begin{split}
-\patl^j_{at}\nab^k_{att}\phi\;\; a_j\Del\patl_k\phi=0,\\
\end{split}\eeq

\beq\label{NT22-at-phs1}\begin{split}
&\patl^j_{at}\nab^k_{att}\phi\;\;\patl_j(B\cdot\nab)\patl_k\psi\\
=&\{(A\cdot\nab)\patl_k-a_k\Del\}\patl^j_{at}\phi\;\;\patl_k\patl_j(B\cdot\nab)\psi\\
=&\patl^j_{at}(A\cdot\nab)\patl_k\phi\;\;\patl_j(B\cdot\nab)\patl_k\psi
-\patl^j_{at}\Del\phi\;\;\patl_j(A\cdot\nab)(B\cdot\nab)\psi\\
=&\patl^j_{at}(A\cdot\nab)\patl_k\phi\;\;\patl_j(B\cdot\nab)\patl_k\psi
+\patl_j\Del\phi\;\;\patl^j_{at}(A\cdot\nab)(B\cdot\nab)\psi\\
=&(A\times\nab)(A\cdot\nab)\patl_k\phi\cdot\nab(B\cdot\nab)\patl_k\psi
+\nab\Del\phi\cdot(A\times\nab)(A\cdot\nab)(B\cdot\nab)\psi,\\
\end{split}\eeq

\beq\label{NT22-at-phs2}\begin{split}
&-\patl^j_{at}\nab^k_{att}\phi\;\; b_j\Del\patl_k\psi
=\{(A\times B)\cdot\nab\nab^k_{att}\phi\}\Del\patl_k\psi,\\
\end{split}\eeq

\beq\label{NT22-at-psh1}\begin{split}
&\patl^j_{at}\nab^k_{btt}\psi\;\;\patl_j(A\cdot\nab)\patl_k\phi\\
=&\patl^j_{at}\{(B\cdot\nab)\patl_k-b_k\Del\}\psi\;\;\patl_j(A\cdot\nab)\patl_k\phi\\
=&\patl^j_{at}(B\cdot\nab)\patl_k\psi\;\;\patl_j(A\cdot\nab)\patl_k\phi
-\patl^j_{at}\Del\psi\;\;\patl_j(A\cdot\nab)(B\cdot\nab)\phi\\
=&-\patl_j(B\cdot\nab)\patl_k\psi\;\;\patl^j_{at}(A\cdot\nab)\patl_k\phi
+\patl_j\Del\psi\;\;\patl^j_{at}(A\cdot\nab)(B\cdot\nab)\phi\\
=&-\nab(B\cdot\nab)\patl_k\psi\cdot(A\times\nab)(A\cdot\nab)\patl_k\phi
+\nab\Del\psi\cdot(A\times\nab)(A\cdot\nab)(B\cdot\nab)\phi,\\
\end{split}\eeq

\beq\label{NT22-at-psh2}\begin{split}
-\patl^j_{at}\nab^k_{btt}\psi\;\; a_j\Del\patl_k\phi=0,\\
\end{split}\eeq

\beq\label{NT22-at-pss1}\begin{split}
&\patl^j_{at}\nab^k_{btt}\psi\;\;\patl_j(B\cdot\nab)\patl_k\psi\\
=&\patl^j_{at}\{(B\cdot\nab)\patl_k-b_k\Del\}\psi\;\;\patl_j(B\cdot\nab)\patl_k\psi\\
=&\patl^j_{at}(B\cdot\nab)\patl_k\psi\;\;\patl_j(B\cdot\nab)\patl_k\psi
-\patl^j_{at}\Del\psi\;\;\patl_j(B\cdot\nab)b_k\patl_k\psi\\
=&-\patl^j_{at}\Del\psi\;\;\patl_j(B\cdot\nab)^2\psi\\
=&\patl_j\Del\psi\;\;\patl^j_{at}(B\cdot\nab)^2\psi
=\nab\Del\psi\cdot(A\times\nab)(B\cdot\nab)^2\psi,\\
\end{split}\eeq

\beq\label{NT-at-pss2}\begin{split}
&-\patl^j_{at}\nab^k_{btt}\psi\;\;b_j\Del\patl_k\psi
=\{(A\times B)\cdot\nab\nab^k_{btt}\psi\}\Del\patl_k\psi.\\
\end{split}\eeq

Putting together \eqref{NS22-phi} and \eqref{NT-Eq22-phi}, we derive
\beq\label{NS22-phi1}\begin{split}
&(A\times B)\cdot\nab\Del\{\phi_t-\nu\Del\phi\}\\
&+(A\times B)\cdot\nab\big(\{A\cdot\nab\patl_k\phi\}\patl_k\Del\phi
-\Del\phi\{A\cdot\nab\Del\phi\}\big)\\
&+(A\times B)\cdot\nab\big(\{B\cdot\nab\patl_k\psi\}\patl_k\Del\phi
-\Del\psi\{B\cdot\nab\Del\phi\}\big)\\
&-\nab\Del\phi\cdot(B\times\nab)(A\cdot\nab)^2\phi
-\nab\Del\phi\cdot(B\times\nab)(A\cdot\nab)(B\cdot\nab)\psi\\
&-\nab\Del\psi\cdot(B\times\nab)(B\cdot\nab)^2\psi
-\nab\Del\psi\cdot(B\times\nab)(A\cdot\nab)(B\cdot\nab)\phi=0.\\
\end{split}\eeq
Putting together \eqref{NS22-psi} and \eqref{NT-Eq22-psi}, we derive
\beq\label{NS22-psi1}\begin{split}
&(A\times B)\cdot\nab\Del\{\psi_t-\nu\Del\psi\}\\
&+(A\times B)\cdot\nab\big(\{A\cdot\nab\patl_k\phi\}\patl_k\Del\psi
-\Del\phi\{A\cdot\nab\Del\psi\}\big)\\
&+(A\times B)\cdot\nab\big(\{B\cdot\nab\patl_k\psi\}\patl_k\Del\psi
-\Del\psi\{B\cdot\nab\Del\psi\}\big)\\
&+\nab\Del\phi\cdot(A\times\nab)(A\cdot\nab)^2\phi
+\nab\Del\phi\cdot(A\times\nab)(A\cdot\nab)(B\cdot\nab)\psi\\
&+\nab\Del\psi\cdot(A\times\nab)(B\cdot\nab)^2\psi
+\nab\Del\psi\cdot(A\times\nab)(A\cdot\nab)(B\cdot\nab)\phi=0.\\
\end{split}\eeq

Now we assume that $\phi$ and $\psi$ are radial symmetric functions
with respect to space variable $x\in\R^3$. It is that $\phi(t,x)=
\phi(t,r)$, $\psi(t,x)=\psi(t,r)$  and $r^2=x_1^2+x_2^2+x_3^2$.
Then we have
\beq\label{phE22-NT1}\begin{split}
&(A\times B)\cdot\nab\big(\{A\cdot\nab\patl_k\phi\}\patl_k\Del\phi
-\Del\phi\{A\cdot\nab\Del\phi\}\big)\\
=&(A\times B)\cdot\nab\big(\{A\cdot\nab\frac{x_k}r\patl_r\phi\}
\frac{x_k}r\patl_r\Del\phi
-\{A\cdot x\}\Del\phi\cdot \frac1r\patl_r\Del\phi\big)\\
=&(A\times B)\cdot\nab\big\{(A\cdot x)\big(\{\frac1r\patl_r\phi
+r\patl_r(\frac1r\patl_r\phi)\}\frac1r\patl_r\Del\phi
-\Del\phi\cdot \frac1r\patl_r\Del\phi\big)\big\}\\
=&(A\times B)\cdot x \{A\cdot x\}\frac1r\patl_r\big(\{\frac1r\patl_r\phi
+r\patl_r(\frac1r\patl_r\phi)\}\frac1r\patl_r\Del\phi
-\Del\phi\cdot \frac1r\patl_r\Del\phi\big)\\
=&-(A\times B)\cdot x \{A\cdot x\}\frac1r\patl_r\big\{\frac2r\patl_r\phi
\cdot\frac1r\patl_r\Del\phi\big\}\\
=&-(A\times B)\cdot x \{A\cdot x\}\big\{\frac2r\patl_r\big(\frac1r
\patl_r\phi\big)\cdot\frac1r\patl_r\Del\phi+\frac2r\patl_r\phi
\cdot\frac1r\patl_r\big(\frac1r\patl_r\Del\phi\big)\big\},\\
\end{split}\eeq

\beq\label{phE22-NT2}\begin{split}
&(A\times B)\cdot\nab\big(\{B\cdot\nab\patl_k\psi\}\patl_k\Del\phi
-\Del\psi\{B\cdot\nab\Del\phi\}\big)\\
=&(A\times B)\cdot\nab\big(\{B\cdot\nab x_k\frac1r\patl_r\psi\}x_k\frac1r
\patl_r\Del\phi-\{B\cdot x\}\Del\psi\cdot\frac1r\patl_r\Del\phi\big)\\
=&(A\times B)\cdot\nab\big\{(B\cdot x)\big(\{\frac1r\patl_r\psi
+r\patl_r(\frac1r\patl_r\psi)\}\frac1r\patl_r\Del\phi
-\Del\psi\cdot\frac1r\patl_r\Del\phi\big)\big\}\\
=&(A\times B)\cdot x\{B\cdot x\}\frac1r\patl_r\big(\{\frac1r\patl_r\psi
+r\patl_r(\frac1r\patl_r\psi)\}\frac1r\patl_r\Del\phi
-\Del\psi\cdot\frac1r\patl_r\Del\phi\big)\\
=&-(A\times B)\cdot x\{B\cdot x\}\frac1r\patl_r\big\{\frac2r\patl_r\psi
\cdot\frac1r\patl_r\Del\phi\big\}\\
=&-(A\times B)\cdot x\{B\cdot x\}\big\{\frac2r\patl_r\big(\frac1r
\patl_r\psi\big)\cdot\frac1r\patl_r\Del\phi+\frac2r\patl_r\psi
\cdot\frac1r\patl_r\big(\frac1r\patl_r\Del\phi\big)\big\},\\
\end{split}\eeq

\beq\label{phE22-NT3}\begin{split}
&-\nab\Del\phi\cdot(B\times\nab)(A\cdot\nab)^2\phi\\
=&-x\frac1r\patl_r\Del\phi\cdot(B\times\nab)\{
(A\cdot\nab)(A\cdot x)\frac1r\patl_r\phi\}\\
=&-x\frac1r\patl_r\Del\phi\cdot(B\times\nab)\{A\cdot A\frac1r\patl_r\phi+
(A\cdot x)^2\frac1r\patl_r(\frac1r\patl_r\phi)\}\\
=&2(A\times B)\cdot x(A\cdot x)\frac1r\patl_r\Del\phi\cdot\{
\frac1r\patl_r(\frac1r\patl_r\phi)\},\\
\end{split}\eeq

\beq\label{phE22-NT4}\begin{split}
&-\nab\Del\phi\cdot(B\times\nab)(A\cdot\nab)(B\cdot\nab)\psi\\
=&-x\frac1r\patl_r\Del\phi\cdot(B\times\nab)(A\cdot\nab)(B\cdot x)\frac1r\patl_r\psi\\
=&-x\frac1r\patl_r\Del\phi\cdot(B\times\nab)\{(A\cdot B)\frac1r\patl_r\psi
+(A\cdot x)(B\cdot x)\frac1r\patl_r(\frac1r\patl_r\psi)\}\\
=&(A\times B)\cdot x(B\cdot x)\frac1r\patl_r\Del\phi\cdot\{
\frac1r\patl_r(\frac1r\patl_r\psi)\},\\
\end{split}\eeq

\beq\label{phE22-NT5}\begin{split}
&-\nab\Del\psi\cdot(B\times\nab)(B\cdot\nab)^2\psi\\
=&-x\frac1r\patl_r\Del\psi\cdot(B\times\nab)(B\cdot\nab)(B\cdot x)\frac1r\patl_r\psi\\
=&-x\frac1r\patl_r\Del\psi\cdot(B\times\nab)\{B\cdot B\frac1r\patl_r\psi
+(B\cdot x)^2\frac1r\patl_r(\frac1r\patl_r\psi)\}\\
=&0,\\
\end{split}\eeq

\beq\label{phE22-NT6}\begin{split}
&-\nab\Del\psi\cdot(B\times\nab)(A\cdot\nab)(B\cdot\nab)\phi\\
=&-x\frac1r\patl_r\Del\psi\cdot(B\times\nab)(A\cdot\nab)(B\cdot x)\frac1r\patl_r\phi\\
=&-x\frac1r\patl_r\Del\psi\cdot(B\times\nab)\{(A\cdot B)\frac1r\patl_r\phi
+(A\cdot x)(B\cdot x)\frac1r\patl_r(\frac1r\patl_r\phi)\}\\
=&(A\times B)\cdot x(B\cdot x)\frac1r\patl_r\Del\psi\cdot
\{\frac1r\patl_r(\frac1r\patl_r\phi)\},\\
\end{split}\eeq


\beq\label{psE22-NT1}\begin{split}
&(A\times B)\cdot\nab\big(\{A\cdot\nab\patl_k\phi\}\patl_k\Del\psi
-\Del\phi\{A\cdot\nab\Del\psi\}\big)\\
=&(A\times B)\cdot\nab\big(\{A\cdot\nab x_k\frac1r\patl_r\phi\}
x_k\frac1r\patl_r\Del\psi
-\{A\cdot x\}\Del\phi\cdot\frac1r\patl_r\Del\psi\big)\\
=&(A\times B)\cdot\nab\big\{(A\cdot x)\big(\{\frac1r\patl_r\phi
+r\patl_r(\frac1r\patl_r\phi)\}\frac1r\patl_r\Del\psi
-\Del\phi\cdot\frac1r\patl_r\Del\psi\big)\big\}\\
=&(A\times B)\cdot x \{A\cdot x\}\frac1r\patl_r\big(\{\frac1r\patl_r\phi
+r\patl_r(\frac1r\patl_r\phi)\}\frac1r\patl_r\Del\psi
-\Del\phi\cdot\frac1r\patl_r\Del\psi\big)\\
=&-(A\times B)\cdot x \{A\cdot x\}\frac1r\patl_r\big\{\frac2r\patl_r\phi
\cdot\frac1r\patl_r\Del\psi\big\}\\
=&-(A\times B)\cdot x \{A\cdot x\}\big\{\frac2r\patl_r\big(\frac1r
\patl_r\phi\big)\cdot\frac1r\patl_r\Del\psi+\frac2r\patl_r\phi
\cdot\frac1r\patl_r\big(\frac1r\patl_r\Del\psi\big)\big\},\\
\end{split}\eeq

\beq\label{psE22-NT2}\begin{split}
&(A\times B)\cdot\nab\big(\{B\cdot\nab\patl_k\psi\}\patl_k\Del\psi
-\Del\psi\{B\cdot\nab\Del\psi\}\big)\\
=&(A\times B)\cdot\nab\big(\{B\cdot\nab x_k\frac1r\patl_r\psi\}x_k\frac1r
\patl_r\Del\psi-\{B\cdot x\}\Del\psi\cdot\frac1r\patl_r\Del\psi\big)\\
=&(A\times B)\cdot\nab\big\{(B\cdot x)\big(\{\frac1r\patl_r\psi
+r\patl_r(\frac1r\patl_r\psi)\}\frac1r\patl_r\Del\psi
-\Del\psi\cdot\frac1r\patl_r\Del\psi\big)\big\}\\
=&(A\times B)\cdot x\{B\cdot x\}\frac1r\patl_r\big(\{\frac1r\patl_r\psi
+r\patl_r(\frac1r\patl_r\psi)\}\frac1r\patl_r\Del\psi
-\Del\psi\cdot\frac1r\patl_r\Del\psi\big)\\
=&-(A\times B)\cdot x\{B\cdot x\}\frac1r\patl_r\big\{\frac2r\patl_r\psi
\cdot\frac1r\patl_r\Del\psi\big\}\\
=&-(A\times B)\cdot x\{B\cdot x\}\big\{\frac2r\patl_r\big(\frac1r
\patl_r\psi\big)\cdot\frac1r\patl_r\Del\psi+\frac2r\patl_r\psi
\cdot\frac1r\patl_r\big(\frac1r\patl_r\Del\psi\big)\big\},\\
\end{split}\eeq

\beq\label{psE22-NT3}\begin{split}
&\nab\Del\phi\cdot(A\times\nab)(A\cdot\nab)^2\phi\\
=&x\frac1r\patl_r\Del\phi\cdot(A\times\nab)(A\cdot\nab)(A\cdot x)\frac1r\patl_r\phi\\
=&x\frac1r\patl_r\Del\phi\cdot(A\times\nab)\{(A\cdot A)\frac1r\patl_r\phi
+(A\cdot x)^2\frac1r\patl_r(\frac1r\patl_r\phi)\}\\
=&0,\\
\end{split}\eeq

\beq\label{psE22-NT4}\begin{split}
&\nab\Del\phi\cdot(A\times\nab)(A\cdot\nab)(B\cdot\nab)\psi\\
=&x\frac1r\patl_r\Del\phi\cdot(A\times\nab)(A\cdot\nab)(B\cdot x)\frac1r\patl_r\psi\\
=&x\frac1r\patl_r\Del\phi\cdot(A\times\nab)\{(A\cdot B)\frac1r\patl_r\psi
+(A\cdot x)(B\cdot x)\frac1r\patl_r(\frac1r\patl_r\psi)\}\\
=&(A\times B)\cdot x(A\cdot x)\frac1r\patl_r\Del\phi\cdot\{
\frac1r\patl_r(\frac1r\patl_r\psi)\},\\
\end{split}\eeq

\beq\label{psE22-NT5}\begin{split}
&\nab\Del\psi\cdot(A\times\nab)(B\cdot\nab)^2\psi\\
=&x\frac1r\patl_r\Del\psi\cdot(A\times\nab)(B\cdot\nab)(B\cdot x)\frac1r\patl_r\psi\\
=&x\frac1r\patl_r\Del\psi\cdot(A\times\nab)\{(B\cdot B)\frac1r\patl_r\psi
+(B\cdot x)^2\frac1r\patl_r(\frac1r\patl_r\psi)\}\\
=&2(A\times B)\cdot x(B\cdot x)\frac1r\patl_r\Del\psi\cdot\{
\frac1r\patl_r(\frac1r\patl_r\psi)\},\\
\end{split}\eeq

\beq\label{psE22-NT6}\begin{split}
&\nab\Del\psi\cdot(A\times\nab)(A\cdot\nab)(B\cdot\nab)\phi\\
=&x\frac1r\patl_r\Del\psi\cdot(A\times\nab)(A\cdot\nab)(B\cdot x)\frac1r\patl_r\phi\\
=&x\frac1r\patl_r\Del\psi\cdot(A\times\nab)\{(A\cdot B)\frac1r\patl_r\phi
+(A\cdot x)(B\cdot x)\frac1r\patl_r(\frac1r\patl_r\phi)\}\\
=&(A\times B)\cdot x(A\cdot x)\frac1r\patl_r\Del\psi\cdot\{
\frac1r\patl_r(\frac1r\patl_r\phi)\}.\\
\end{split}\eeq
Here
\[
\Del f=\frac3r\patl_rf+r\patl_r(\frac1r\patl_rf).
\]

Putting together \eqref{NS22-phi1} and \eqref{phE22-NT1}--
\eqref{phE22-NT6}, we derive
\beq\label{NS22-phi2}\begin{split}
&\Del\{\phi_t-\nu\Del\phi\}\\
=&\{A\cdot x\}\big\{\frac2r\patl_r\big(\frac1r
\patl_r\phi\big)\cdot\frac1r\patl_r\Del\phi+\frac2r\patl_r\phi
\cdot\frac1r\patl_r\big(\frac1r\patl_r\Del\phi\big)\big\}\\
&+\{B\cdot x\}\big\{\frac2r\patl_r\big(\frac1r
\patl_r\psi\big)\cdot\frac1r\patl_r\Del\phi+\frac2r\patl_r\psi
\cdot\frac1r\patl_r\big(\frac1r\patl_r\Del\phi\big)\big\}\\
&-2(A\cdot x)\frac1r\patl_r\Del\phi\cdot\{
\frac1r\patl_r(\frac1r\patl_r\phi)\}
-(B\cdot x)\frac1r\patl_r\Del\phi\cdot\{
\frac1r\patl_r(\frac1r\patl_r\psi)\}\\
&-(B\cdot x)\frac1r\patl_r\Del\psi\cdot
\{\frac1r\patl_r(\frac1r\patl_r\phi)\}\\
=&\{A\cdot x\}\frac2r\patl_r\phi\cdot\frac1r\patl_r\big(\frac1r\patl_r
\Del\phi\big)-(B\cdot x)\frac1r\patl_r\Del\psi\cdot
\frac1r\patl_r(\frac1r\patl_r\phi)\\
&+\{B\cdot x\}\big\{\frac1r\patl_r\big(\frac1r
\patl_r\psi\big)\cdot\frac1r\patl_r\Del\phi+\frac2r\patl_r\psi
\cdot\frac1r\patl_r\big(\frac1r\patl_r\Del\phi\big)\big\}\\
=&\Big\{A\frac2r\patl_r\phi\cdot\frac1r\patl_r\big(\frac1r\patl_r
\Del\phi\big)-B\frac1r\patl_r\Del\psi\cdot
\frac1r\patl_r(\frac1r\patl_r\phi)\\
&+B\big\{\frac1r\patl_r\big(\frac1r
\patl_r\psi\big)\cdot\frac1r\patl_r\Del\phi+\frac2r\patl_r\psi
\cdot\frac1r\patl_r\big(\frac1r\patl_r\Del\phi\big)\big\}\Big\}\cdot x.\\
\end{split}\eeq

Putting together \eqref{NS22-psi1} and \eqref{psE22-NT1}--
\eqref{psE22-NT6}, we derive
\beq\label{NS22-psi2}\begin{split}
&\Del\{\psi_t-\nu\Del\psi\}\\
=&\{A\cdot x\}\big\{\frac2r\patl_r\big(\frac1r
\patl_r\phi\big)\cdot\frac1r\patl_r\Del\psi+\frac2r\patl_r\phi
\cdot\frac1r\patl_r\big(\frac1r\patl_r\Del\psi\big)\big\}\\
&+\{B\cdot x\}\big\{\frac2r\patl_r\big(\frac1r
\patl_r\psi\big)\cdot\frac1r\patl_r\Del\psi+\frac2r\patl_r\psi
\cdot\frac1r\patl_r\big(\frac1r\patl_r\Del\psi\big)\big\}\\
&-(A\cdot x)\frac1r\patl_r\Del\phi\cdot\{
\frac1r\patl_r(\frac1r\patl_r\psi)\}\\
&-2(B\cdot x)\frac1r\patl_r\Del\psi\cdot\{
\frac1r\patl_r(\frac1r\patl_r\psi)\}
-(A\cdot x)\frac1r\patl_r\Del\psi\cdot\{
\frac1r\patl_r(\frac1r\patl_r\phi)\}\\
=&\{A\cdot x\}\big\{\frac1r\patl_r\big(\frac1r\patl_r\phi\big)
\cdot\frac1r\patl_r\Del\psi+\frac2r\patl_r\phi\cdot\frac1r\patl_r
\big(\frac1r\patl_r\Del\psi\big)\big\}\\
&-(A\cdot x)\frac1r\patl_r\Del\phi\cdot\frac1r\patl_r(\frac1r\patl_r\psi)
+\{B\cdot x\}\frac2r\patl_r\psi\cdot\frac1r\patl_r
\big(\frac1r\patl_r\Del\psi\big)\\
=&\Big\{A\big\{\frac1r\patl_r\big(\frac1r\patl_r\phi\big)
\cdot\frac1r\patl_r\Del\psi+\frac2r\patl_r\phi\cdot\frac1r\patl_r
\big(\frac1r\patl_r\Del\psi\big)\big\}\\
&-A\frac1r\patl_r\Del\phi\cdot\frac1r\patl_r(\frac1r\patl_r\psi)
+B\frac2r\patl_r\psi\cdot\frac1r\patl_r
\big(\frac1r\patl_r\Del\psi\big)\Big\}\cdot x.\\
\end{split}\eeq

\begin{proof}[Proof of Theorem \ref{SR22-Thm}]
Let us select orthogonal transformation $\rho$ as follows
\beq\label{SR22-OT}
y=\rho x=x\left(\begin{array}{ccc}
0&0&1\\1&0&0\\0&1&0\\
\end{array}\right).
\eeq
Then $r^2=y\cdot y=\rho x\cdot\rho x=x\cdot x$.

Applying the orthogonal transformation \eqref{SR22-OT} in the equations
\eqref{NS22-phi2} \eqref{NS22-psi2}, we obtain
\beq\label{NS22-phi2y}\begin{split}
&\Del\{\phi_t-\nu\Del\phi\}\\
=&\Big\{A\frac2r\patl_r\phi\cdot\frac1r\patl_r\big(\frac1r\patl_r
\Del\phi\big)-B\frac1r\patl_r\Del\psi\cdot
\frac1r\patl_r(\frac1r\patl_r\phi)\\
&+B\big\{\frac1r\patl_r\big(\frac1r
\patl_r\psi\big)\cdot\frac1r\patl_r\Del\phi+\frac2r\patl_r\psi
\cdot\frac1r\patl_r\big(\frac1r\patl_r\Del\phi\big)\big\}\Big\}\cdot y,\\
\end{split}\eeq
\beq\label{NS22-psi2y}\begin{split}
&\Del\{\psi_t-\nu\Del\psi\}\\
=&\Big\{A\big\{\frac1r\patl_r\big(\frac1r\patl_r\phi\big)
\cdot\frac1r\patl_r\Del\psi+\frac2r\patl_r\phi\cdot\frac1r\patl_r
\big(\frac1r\patl_r\Del\psi\big)\big\}\\
&-A\frac1r\patl_r\Del\phi\cdot\frac1r\patl_r(\frac1r\patl_r\psi)
+B\frac2r\patl_r\psi\cdot\frac1r\patl_r
\big(\frac1r\patl_r\Del\psi\big)\Big\}\cdot y.\\
\end{split}\eeq
Employing the equations \eqref{NS22-phi2} \eqref{NS22-phi2y}, we get
\beq\label{22-ph2xy}\begin{split}
0=&\Big\{A\frac2r\patl_r\phi\cdot\frac1r\patl_r\big(\frac1r\patl_r
\Del\phi\big)-B\frac1r\patl_r\Del\psi\cdot
\frac1r\patl_r(\frac1r\patl_r\phi)\\
&+B\big\{\frac1r\patl_r\big(\frac1r\patl_r\psi\big)\cdot\frac1r\patl_r
\Del\phi+\frac2r\patl_r\psi\cdot\frac1r\patl_r\big(\frac1r\patl_r
\Del\phi\big)\big\}\Big\}\cdot(\rho x-x).\\
\end{split}\eeq
Similarly using the equations \eqref{NS22-psi2} \eqref{NS22-psi2y},
we derive
\beq\label{22-ps2xy}\begin{split}
0=&\Big\{A\big\{\frac1r\patl_r\big(\frac1r\patl_r\phi\big)
\cdot\frac1r\patl_r\Del\psi+\frac2r\patl_r\phi\cdot\frac1r\patl_r
\big(\frac1r\patl_r\Del\psi\big)\big\}\\
&-A\frac1r\patl_r\Del\phi\cdot\frac1r\patl_r(\frac1r\patl_r\psi)
+B\frac2r\patl_r\psi\cdot\frac1r\patl_r
\big(\frac1r\patl_r\Del\psi\big)\Big\}\cdot(\rho x-x).\\
\end{split}\eeq
Given $r$, thanks $x\in{\mathbb S}^2_r$ is arbitrary, the
equations \eqref{22-ph2xy} \eqref{22-ps2xy} imply that
\beq\label{NS22-phi2nr}\begin{split}
&A\frac2r\patl_r\phi\cdot\frac1r\patl_r\big(\frac1r\patl_r
\Del\phi\big)-B\frac1r\patl_r\Del\psi\cdot
\frac1r\patl_r(\frac1r\patl_r\phi)\\
&+B\big\{\frac1r\patl_r\big(\frac1r
\patl_r\psi\big)\cdot\frac1r\patl_r\Del\phi+\frac2r\patl_r\psi
\cdot\frac1r\patl_r\big(\frac1r\patl_r\Del\phi\big)\big\}=0,\\
\end{split}\eeq
\beq\label{NS22-psi2nr}\begin{split}
&A\big\{\frac1r\patl_r\big(\frac1r\patl_r\phi\big)
\cdot\frac1r\patl_r\Del\psi+\frac2r\patl_r\phi\cdot\frac1r\patl_r
\big(\frac1r\patl_r\Del\psi\big)\big\}\\
&-A\frac1r\patl_r\Del\phi\cdot\frac1r\patl_r(\frac1r\patl_r\psi)
+B\frac2r\patl_r\psi\cdot\frac1r\patl_r
\big(\frac1r\patl_r\Del\psi\big)=0.\\
\end{split}\eeq

Putting \eqref{NS22-phi2nr} into \eqref{NS22-phi2} and putting
\eqref{NS22-psi2nr} into \eqref{NS22-psi2}, we have
\beq\label{NS22-phi2r}\begin{split}
&\Del\{\phi_t-\nu\Del\phi\}=0,\\
\end{split}\eeq
\beq\label{NS22-psi2r}\begin{split}
&\Del\{\psi_t-\nu\Del\psi\}=0.\\
\end{split}\eeq

Since the vectors $A$ and $B$ are linearly independent, the equation
\eqref{NS22-phi2nr} implies that
\beq\label{22-ph2nrA}\begin{split}
&\patl_r\phi\cdot\patl_r\big(\frac1r\patl_r\Del\phi\big)=0,\\
\end{split}\eeq
\beq\label{22-ph2nrB}\begin{split}
&\patl_r\big(\frac1r\patl_r\psi\big)\cdot\patl_r\Del\phi
+2\patl_r\psi\cdot\patl_r\big(\frac1r\patl_r\Del\phi\big)\\
=&\patl_r\Del\psi\cdot\patl_r(\frac1r\patl_r\phi).\\
\end{split}\eeq
And the equation \eqref{NS22-psi2nr} implies that
\beq\label{22-ps2nrA}\begin{split}
&\patl_r\big(\frac1r\patl_r\phi\big)\cdot\patl_r\Del\psi
+2\patl_r\phi\cdot\patl_r\big(\frac1r\patl_r\Del\psi\big)\\
=&\patl_r\Del\phi\cdot\patl_r(\frac1r\patl_r\psi),\\
\end{split}\eeq
\beq\label{22-ps2nrB}\begin{split}
\patl_r\psi\cdot\patl_r\big(\frac1r\patl_r\Del\psi\big)=0.\\
\end{split}\eeq

If $\patl_r\phi=\patl_r\psi=0$, all equations \eqref{NS22-phi2r},
\eqref{NS22-psi2r}, \eqref{22-ph2nrA}, \eqref{22-ph2nrB},
\eqref{22-ps2nrA} and \eqref{22-ps2nrB} are satisfied. Here velocity
vector $u=0$. This is trivial.

Now assume that at least one of $\patl_r\phi$ and $\patl_r\psi$ is not
zero. Then at least one of the following equations
\beq\label{22-ph2A1}\begin{split}
&\patl_r\big(\frac1r\patl_r\Del\phi\big)=0\\
\end{split}\eeq
and
\beq\label{22-ps2B1}\begin{split}
\patl_r\big(\frac1r\patl_r\Del\psi\big)=0\\
\end{split}\eeq
is satisfied.

The equations \eqref{22-ph2A1} and \eqref{NS22-phi2r} imply that
\beq\label{22-ph22}\begin{split}
&\phi=f_4r^4+(12f_4\nu t+f_2)r^2+f_0(t),\\
\end{split}\eeq
where $f_2$ and $f_4$ are arbitrary constants, $f_0(t)$ is arbitrary
function of $t$.

Similarly the equations \eqref{22-ps2B1} and \eqref{NS22-psi2r} imply
that
\beq\label{22-ps22}\begin{split}
&\psi=g_4r^4+(12g_4\nu t+g_2)r^2+g_0(t),\\
\end{split}\eeq
where $g_2$ and $g_4$ are arbitrary constants, $g_0(t)$ is arbitrary
function of $t$.

The equations \eqref{22-ph2nrB} and \eqref{22-ps2nrA} imply
\beq\label{22-phs-c}\begin{split}
f_2g_4=f_4g_2.
\end{split}\eeq

It is obvious that
\[
\int_{\R^3}|(A\times\nab)\times\nab\phi|^2dx=\infty,
\]
\[
\int_{\R^3}|(B\times\nab)\times\nab\psi|^2dx=\infty,
\]
\[
\int_{\R^3}|u|^2dx=\int_{\R^3}|(A\times\nab)\times\nab\phi+
(B\times\nab)\times\nab\psi|^2dx=\infty,
\]
where $\phi$ and $\psi$ are defined by \eqref{22-ph22} \eqref{22-ps22}
respectively.

In summary, the equations  \eqref{NS22-phi2} \eqref{NS22-psi2} are
only satisfied by $(\phi,\psi)$ defined in \eqref{22-ph22}
\eqref{22-ps22} \eqref{22-phs-c}.

On the other hand, provided that at least one of \eqref{22-ph22},
\eqref{22-ps22} and \eqref{22-phs-c} is not satisfied, then the
equations \eqref{NS22-phi2} \eqref{NS22-psi2} can not be satisfied
by any radial symmetry functions $\phi$ and $\psi$.

In summary, Theorem \ref{SR22-Thm} is proved.
\end{proof}


\section*{Acknowledgments}
This work is supported by National Natural Science Foundation of China--NSF,
Grant No.11971068 and No.11971077.



\begin{thebibliography}{99}

\bibitem{Bac75}
A. V. B$\ddot{a}$cklund. Einiges $\ddot{U}$ber Curven und
Fl$\ddot{a}$chentransformationen.
\newblock {\em Lund Universitets Arsskrift }
\newblock 10(1875), 1--12.

\bibitem{BoP08}
J. Bourgain and N. Pavlovic. Ill-posedness of the Navier-Stokes
equations in a critical spaces in 3D.
\newblock {\em J. Funct. Anal.}
\newblock 255(2008), 2233--2247.

\bibitem{CKN82}
L. Caffarelli, R. Kohn and L. Nirenberg. Partial regularity of suitable
weak solutions of the Navier-Stokes equations.
\newblock {\em Comm. Pure Appl. Math.}
\newblock 35 (1982), 771--837.

\bibitem{Cal90}
C. P. Calder$\acute{o}$n. Existence of weak solutions for the Navier-
Stokes equations with initial data in $L^p$.
\newblock {\em Trans. Amer. Math. Soc.}
\newblock 318(1982), No.1, 179--200.

\bibitem{Can97}
M. Cannone. A generalization of a theorem by Kato on Navier-Stokes
equations.
\newblock {\em Rev. Mat. Iberoamericana}
\newblock 13(1997), 515--541.

\bibitem{Can03}
M. Cannone. Harmonis analysis tools for solving the incompressible
Navier-Stokes equations.
\newblock {\em Handbook of Mathematical Fluid Dynamics, Vol. 3,
Eds. S. Friedlander and D. Serre,}
\newblock Elsevier, 2003, 161--244.

\bibitem{Car96}
A. Carpio. Large time behavior in incompressible Navier-Stokes
equations.
\newblock {\em SIAM J. Math. Anal.}
\newblock 27(1996), 449--475.

\bibitem{CSTY08}
C.-C. Chen, R. M. Strain, T.-P. Tsai and H.-T. Yau. Lower bound on the
blow-up rate of the axisymmetric Navier-Stokes equations.
\newblock {\em Int. Math. Res. Not. }
\newblock 9(2008), Art. ID rnn016, 31pp.

\bibitem{CSTY09}
C.-C. Chen, R. M. Strain, T.-P. Tsai and H.-T. Yau. Lower bounds on the
blow-up rate of the axisymmetric Navier-Stokes equations II.
\newblock {\em Comm. Part. Diff. Eq. }
\newblock 34(2009), 203--232.

\bibitem{Che99}
J.-Y. Chemin. Th$\acute{e}$or$\grave{e}$mes d'unicit$\acute{e}$ pour
le syst$\grave{e}$me de Navier-Stokes tridimensionnel.
\newblock {\em J. Anal. Math.}
\newblock 77(1999), 27--50.

\bibitem{CMP14}
J. C. Cortissoz, J. A. Montero and C. E. Pinilla. On lower bounds for
possible blow up solutions to the periodic Navier-Stokes equations.
\newblock {\em J. Math. Phys.}
\newblock 55(2014), 033101.

\bibitem{ESS03}
L. Escauriaza, G. Seregin and V. \v{S}ver\'{a}k.  $L_{3,\infty}$-
solutions to the Navier-Stokes equations and backward uniqueness.
\newblock {\em Russian Math. Surveys}
\newblock 58(2003), 211--250.

\bibitem{GKP13}
I. Gallagher, G. S. Koch and F. Planchnon. A profile decomposition
approach to the $L^{\infty}_t(L^3_x)$ Navier-Stokes regularity
criterion.
\newblock {\em Math. Ann. }
\newblock  355(2013), 1527--1559.

\bibitem{Ger08}
P. Germain. The second iterate for the Navier-Stokes equation.
\newblock {\em J. Funct. Anal.}
\newblock 255(2008), 2248--2264.

\bibitem{Gig86}
Y. Giga. Solutions for semilinear parabolic equations in $L^p$ and
regularity of weak solutions of the Navier-Stokes system.
\newblock {\em J. Diff. Eq. }
\newblock  62(1986), 186--212.

\bibitem{GiInM99}
Y. Giga, K. Inui and S. Matsui. On the Cauchy problem for the Navier-
Stokes equations with nondecaying initial data.
\newblock {\em Quad. Mat. }
\newblock 4 (1999), 27--68.

\bibitem{GiM85}
Y. Giga and T. Miyakawa. Solutions in $L^r$ of the Navier-
Stokes initial value problem.
\newblock {\em Arch. Rational Mech. Math.}
\newblock 89(1985), 267--281.

\bibitem{GiM89}
Y. Giga and T. Miyakawa. Navier-Stokes flow in $\R^3$ with measures
as initial vorticity and Morrey spaces.
\newblock {\em Comm. Partial Diff. Eq.}
\newblock 14(1989), 577--618.

\bibitem{HM77}
A. Hasegawa and K. Mima. Stationary spectrum of strong turbulence
in magnetized nonuniform plasma.
\newblock {\em Phys. Rev. Lett.}
\newblock 39 (1977), 205.

\bibitem{HM78}
A. Hasegawa and K. Mima. Pseudo-three-dimensional turbulence
in magnetized nonuniform plasma.
\newblock {\em Phys. Fluids}
\newblock 21 (1978), 87--92.

\bibitem{Hop51}
E. Hopf. Uber die Anfangswertaufgabe f$\ddot{u}$r die hydrodynamischen
Grundgleichungen.
\newblock {\em Math. Nachr. }
\newblock  4(1951), 213--231.

\bibitem{Ka75}
T. Kato. Quasi-linear equations of evolution, with applications to
partial differential equations.
\newblock {\em Springer Lecture Notes in Mathematics 448. }
\newblock pp.25--70, Berlin, Springer, 1975.

\bibitem{Ka84}
T. Kato. Strong $L^p$ solutions of the Navier-Stokes equations in $\R^m
$, with applications to weak solutions.
\newblock {\em Math. Z. }
\newblock 187(1984), 471--480.

\bibitem{KaF62}
T. Kato and H. Fujita. On The nonstationary Navier-Stokes system.
\newblock {\em Rend Sem. Mat. Univ. Padova}
\newblock 32(1962), 243--260.

\bibitem{KaP94}
T. Kato and G. Ponce. The Navier-Stokes equation with weak initial data.
\newblock {\em International Math. Res. Notices}
\newblock 10 (1994), 435--444.

\bibitem{KeK11}
C. E. Kenig and G. S. Koch. An alternativ approach to regularity for
The Navier-Stokes equations in criticle spaces.
\newblock {\em Ann. l'Inst. H. Poincara (C) Non Linear Anal.}
\newblock 28(2011), 159--187.

\bibitem{KNSS09}
G. Koch, N. Nadirashvili, G. A. Seregin and V. $\breve{S}$ver$\acute{a}$k.
Liouville theorems for the Navier-Stokes equations and applications.
\newblock {\em Acta Math. }
\newblock 203(2009), 83--105.

\bibitem{KoT01}
H. Koch and D. Tataru. Well-posedness for The Navier-Stokes equations.
\newblock {\em Adv. Math.}
\newblock 157(2001), 22--35.

\bibitem{La69}
O. A. Ladyzenskaya.
\newblock {\em The Mathematical Theory of Viscous Incompressible
Flows (2nd edition).}
\newblock Gordon and Breach, 1969.

\bibitem{LaS99}
O. A. Ladyzenskaya and G. A. Geregin. On partial regularity of suitable
weak solutions to the three-demensional Navier-Stokes equations.
\newblock {\em J. Math. Fluid Mech.}
\newblock 1(1999), 356--387.

\bibitem{Le34}
J. Leray. Sur le mouvement d'un liquid visqueux emplissant l'espace.
\newblock {\em Acta Math.}
\newblock 63 (1934), 193--248.

\bibitem{LOW18}
K. Li, T. Ozawa and B. Wang. Dynamical behavior for the solutions of
the Navier-Stokes equation.
\newblock {\em Comm. Pure Appl. Anal.}
\newblock 17(4), 2018, 241--257.

\bibitem{LiW19}
K. Li and B. Wang. Blowup criterion for Navier-Stokes equation in
critical Besov space with spatial dimensions $d\ge4$.
\newblock {\em Ann. Inst. Henri Poincare (C) Anal. Nonlinear}
\newblock 36(6), 2019, 1679--1707.

\bibitem{Lin98}
F.-H. Lin. A new proof of the Caffarelli-Kohn-Nirenberg theorem.
\newblock {\em Comm. Pure Appl. Math.}
\newblock 51 (1998), 241--257.

\bibitem{Lun95} A. Lunardi.
\newblock {\em Analytic Semigroups and Optimal Regularity in
Parabolic Problems. }
\newblock Basel, Boston; Berlin: Birkh\"{a}user, 1995.

\bibitem{MiS01}
T. Miyakawa and M. E. Schonbek. On optimal decay rates for weak
solutions to the Navier-Stokes equations in $\R^n$.
\newblock {\em Math. Bohem.}
\newblock 126(2), 2001, 443--455.

\bibitem{NeRS96}
J. Ne\v{c}as, M. Ru\v{z}i\v{c}ka and V. \v{S}ver\'{a}k.
On Leray's self-similar solutions of the Navier-Stokes
equations
\newblock {\em Acta Math.}
\newblock 176(1996), 283--294.

\bibitem{Pla96}
F. Planchon. Global strong solutions in Sobolev or Lebesgue spaces to
the incompressible Navier-Stokes equations in $\R^3$.
\newblock {\em Ann. Inst. H. Poincare (C) Anal. Nonlinear}
\newblock 13(1996), 319--336.

\bibitem{Pro59}
G. Prodi. Un teorema di unicit$\grave{a}$ per le equazioni di
Navier-Stokes.
\newblock {\em Ann. Mat. Pura Appl.}
\newblock 48(1959), 173--182.

\bibitem{RSS12}
J. C. Robinson, W. Sadowski and R. P. Silva. Lower bounds on blow up
solutions of the three dimensional Navier-Stokes equations in
homogeneous Sobolev spaces.
\newblock {\em J. Math. Phys.}
\newblock 53(2012), 115618-1-15.

\bibitem{RogS82}
C. Rogers, W. R. Shadwick. \newblock {\em B$\ddot{a}$cklund
Transformations and Their Applications. }
\newblock Academic Press, New York, 1982.

\bibitem{Sch76}
V. Scheffer. Partial regularity of solutions to the
Navier-Stokes equations.
\newblock {\em Pacific J. Math.}
\newblock 66 (1976), 535--552.

\bibitem{Sch77}
V. Scheffer. Hausdorff measure and the Navier-Stokes equations.
\newblock {\em Comm. Math. Phys.}
\newblock 55 (1977), 97--112.

\bibitem{Sch91}
M. E. Schonbek. Lower bounds of rates of decay for solutions to the
Navier-Stokes equations.
\newblock {\em  J. Amer. Math. Soc.}
\newblock 4(1991), 423--449.

\bibitem{Sch92}
M. E. Schonbek. Asymptotic behavior of solutions to the three-
dimensional Navier-Stokes equations.
\newblock {\em Indiana Univ. Math. J.}
\newblock 41(1992), 809--823.

\bibitem{Sere12}
G. Seregin. A certain necessary condition of potential blow up for
Navier-Stokes equations.
\newblock {\em Comm. Math. Phys.}
\newblock 312(2012), 833--845.

\bibitem{Ser62}
J. Serrin. On the interior regularity of weak solutions of the
Navier-Stokes equations.
\newblock {\em Arch. Rational Mech. Anal.}
\newblock 9 (1962), 187--195.

\bibitem{Ste70}
E. M. Stein.
\newblock {\em Singular Integrals and Differentiability Properties
of Functions.}
\newblock Princeton University Press 1970.

\bibitem{Tay92}
M. Taylor. Analysis on Morrey spaces and applications to Navier-Stokes
equation.
\newblock {\em Comm. Partial Diff. Eq.}
\newblock 17(1992), 1407--1456.

\bibitem{Te01}
R. Temam.
\newblock {\em Navier-Stokes Equations. Theory and numerical analysis.
(Reprint of the 1984 edition).}
\newblock AMS Chelsea Publishing, Providence, RI, 2001.

\bibitem{Tsai98}
T.-P. Tsai. On Leray's self-similar solutions of the Navier-Stokes
equations satisfying local energy estimates.
\newblock {\em Arch. Rational Mech. Anal. }
\newblock 143 (1998), 29--51.

\bibitem{Vas07}
A. Vasseur. A new proof of partial regularity of solutions to the
Navier-Stokes equations.
\newblock {\em Nonlin. Diff. Eq. Appl.}
\newblock 14(2007), 753--785.

\bibitem{Wah86}
W. von Wahl. Regularity of weak solutions of the Navier-Stokes equations.
\newblock {\em Proc. Symp. Pure Appl. Math. }
\newblock  45(1986), 497--503.

\bibitem{Wang15}
B. Wang. Ill-posedness for the Navier-Stokes equation in critical
Besov spaces $\dot{B}^{-1}_{\infty,q}$.
\newblock {\em Adv. in Math.}
\newblock 268(2015), 350--372.

\bibitem{WaW14}
Y. Wang and G. Wu. A unified proof on the partial regularity for suitable
weak solutions of nonstationary and stationary Navier-Stokes equations.
\newblock {\em J. Diff. Eq.}
\newblock 256(2014), 1224--1249.

\bibitem{Wei80}
F. B. Weissler. The Navier-Stokes initial value problem in $L^p$.
\newblock {\em Arch. Rational Mech. Anal.}
\newblock  74(1980), 219--230.

\bibitem{Yon10}
T. Yoneda. Ill-posedness of the 3D Navier-Stokes equations in
generalized Besov space near BMO$^{-1}$.
\newblock {\em J. Funct. Anal.}
\newblock 258(2010), 3376--3387.


\end{thebibliography}
\end{document}